\documentclass[11pt]{amsart}
\usepackage{amscd,amssymb,graphics,color,a4wide,hyperref,mathtools}
\usepackage{graphicx}
\usepackage{psfrag} 
\usepackage{marvosym}
\usepackage{tikz}
\usepackage{tkz-euclide}
\usepackage{multicol}
\usetikzlibrary{matrix}
\usetikzlibrary{mindmap,trees,calc}

\usepackage{color}
\usepackage{mathrsfs}
\usepackage[all]{xy}
\DeclareFontFamily{U}{rsf}{}
\DeclareFontShape{U}{rsf}{m}{n}{
  <5> <6> rsfs5 <7> <8> <9> rsfs7 <10->  rsfs10}{}
\DeclareMathAlphabet{\mathscr}{U}{rsf}{m}{n}

\newcommand\dashto{\mathrel{
  -\mkern-6mu{\to}\mkern-20mu{\color{white}\bullet}\mkern12mu
}}

\newtheorem{theorem}{Theorem}[section]
\newtheorem{lemma}[theorem]{Lemma}
\newtheorem{proposition}[theorem]{Proposition}
\newtheorem{corollary}[theorem]{Corollary}

\theoremstyle{definition}
\newtheorem{definition}[theorem]{Definition}
\newtheorem{construction}[theorem]{Construction}
\newtheorem{example}[theorem]{Example}

\newtheorem*{acknowledgement}{Acknowledgement}
\newtheorem{setting}[theorem]{Setting}

\theoremstyle{remark}
\newtheorem{remark}[theorem]{Remark}
\numberwithin{equation}{section}

\usepackage{color}
\definecolor{forestgreen}{rgb}{0.13, 0.55, 0.13}

\newcommand{\NN} {\mathbb{N}}
\newcommand{\ZZ} {\mathbb{Z}}
\newcommand{\QQ} {\mathbb{Q}}
\newcommand{\RR} {\mathbb{R}}

\newcommand{\CC} {\mathbb{C}}

\newcommand{\HH} {\mathbb{H}}
\newcommand{\PP} {\mathbb{P}}

\newcommand{\GG} {\mathbb{G}}
\newcommand{\DD} {\mathbb{D}}

\newcommand {\shB} {\mathcal{B}}
\newcommand {\shC} {\mathcal{C}}

\newcommand {\shF} {\mathcal{F}}

\newcommand {\shH} {\mathcal{H}}

\newcommand {\shL} {\mathcal{L}}
\newcommand {\shM} {\mathcal{M}}

\newcommand {\shO} {\mathcal{O}}

\newcommand {\shS} {\mathcal{S}}
\newcommand {\shT} {\mathcal{T}}

\newcommand {\shX} {\mathcal{X}}
\newcommand {\shY} {\mathcal{Y}}
\newcommand {\shZ} {\mathcal{Z}}

\newcommand {\Aut} {\operatorname{Aut}}

\newcommand {\Bir} {\operatorname{Bir}}

\newcommand{\contr}{\operatorname{contr}}

\newcommand {\codim} {\operatorname{codim}}

\newcommand {\DNV} {\operatorname{DNV}}

\newcommand{\YP}{Y_\mathscr{P}}
\newcommand{\YT}{Y_\mathscr{T}}

\newcommand{\YYP}{\shY_\mathscr{P}}
\newcommand{\YYT}{\shY_\mathscr{T}}

\newcommand {\Ex} {\operatorname{Ex}}

\newcommand {\GL} {\operatorname{GL}}

\newcommand {\Og} {\operatorname{O}}

\newcommand {\id} {\operatorname{id}}

\newcommand {\Int} {\operatorname{Int}}

\renewcommand {\ker } {\operatorname{ker}}

\newcommand {\Morifan}{\operatorname{MF}}
\newcommand {\Mov} {\operatorname{Mov}}

\newcommand {\Mod}{\operatorname{Mod}}
\newcommand {\N} {\operatorname{N}}

\newcommand {\Nef}{\operatorname{Nef•}}
\newcommand{\NE}{\operatorname{NE}}
\renewcommand{\O} {\mathcal{O}}

\newcommand{\PMod}{\operatorname{PMod}}
\renewcommand{\P} {\mathscr{P}}

\newcommand {\Pic} {\operatorname{Pic}}

\newcommand {\rank} {\operatorname{rank}}

\newcommand{ \Relint}{\operatorname{Relint}}

\newcommand {\Sing} {\operatorname{Sing}}

\newcommand {\Spec} {\operatorname{Spec}}
\newcommand {\Spf} {\operatorname{Spf}}

\newcommand {\T} {\shT}

\def\mydate{\ifcase\month \or January\or February\or March\or
April\or May\or June\or July\or August\or September\or October\or 
November\or December\fi \space\number\day,\space\number\year}
\newlength{\picwidth} 
\newlength{\miniwidth} 
\newlength{\nanowidth} 
\newlength{\melowidth} 
\newlength{\leftminiwidth} 
\newlength{\rightminiwidth} 
\newlength{\minipagewidth} 
\begin{document}
\title
[Mori fan of the DNV family
  ]
{The Mori fan of the Dolgachev-Nikulin-Voisin family in genus $2$}
\author{Klaus Hulek}
\address{Institut f\"ur Algebraische Geometrie, Leibniz Universit\"at Hannover, Welfengarten 1, 30167 Hannover,
Germany}
\email{hulek@math.uni-hannover.de}

\author{Carsten Liese}
\address{Institut f\"ur Algebraische Geometrie, Leibniz Universit\"at Hannover, Welfengarten 1, 30167 Hannover,
Germany}
%%%\curraddr{}
\email{liese@math.uni-hannover.de}
%%\thanks{This work was partially supported by NSF grants
%%0854987 and 1105871.}
\maketitle
\begin{abstract} In this paper we  study the Mori fan of the Dolgachev-Nikulin-Voisin family in degree $2$ as well as the associated secondary fan. The main result is an enumeration of all maximal dimensional cones  of the two fans.
\end{abstract}

\tableofcontents
\bigskip

\section{Introduction}

To construct modular compactifications of the moduli space $\shF_{2d}$ of polarized $K3$ surfaces of degree $2d$ is a notoriously  difficult problem. This has been studied  from various aspects, such
as Hodge theory, locally symmetric domains, GIT and log-geometry. For small degree $d=1,2,3$ these $K3$ surfaces can be studied via concrete geometric models, namely $2:1$ covers of the 
projective plane $\PP^2$ branched along a sextic curve, degree $4$ surfaces in $\PP^3$ and complete $(2,3)$ intersections in $\PP^4$  respectively. Various authors have used this approach to construct compactifications
of  $\shF_{2d}$ in these degrees and to relate the various models to each other. Here we would like to mention in particular the work of  
Shah \cite{Shah}), Friedman \cite{Frie83}, \cite{Loo86}), Laza-O'Grady \cite{log16},   
\cite{log18}, \cite{logabel},  and Laza \cite{Laz12}, Thompson \cite{tho14} and Alexeev-Engel-Thompson \cite{ATk3}. 
 
Some years ago Gross, Hacking, Keel and Siebert \cite{theta15} introduced a new approach. This is based on two main concepts: mirror symmetry and the minimal model program (MMP). 
They start by considering the mirror family of  $\shF_{2d}$. This is a $1$-dimensional family of lattice polarized $K3$ surfaces which in the literature is called the 
Dolgachev-Nikulin-Voisin family of degree $2d$. The base of this family 
is a modular curve which, if $d$ is squarefree, has exactly one cusp. The lattice polarization is given by the lattice $\check{M}_{2d}=U\oplus 2E_8(-1)\oplus\langle-2d\rangle $ where $U$ is the hyperbolic plane and $E_8(-1)$ is the negative defnite $E_8$ lattice. The first step in their programme is to extend the Dolgachev-Nikulin-Voisin family over the cusp and to consider the various models by which this can be done. This allows them to define the Mori fan of the 
Dolgachev-Nikulin-Voisin family, wich is a fan in $\N^1(\shY/S)_{\RR}$, where $\shY \to S$ is a  model of the Dolgachev-Nikulin-Voisin family, a scheme of dimension $19+d$, with base  $S=\Spec \CC[[t]]$.
The second step is to use 
a  piecewise linear section of the restriction map $r: \Pic(\shY) \to \Pic(\shY_{\eta})$, where $\shY_{\eta}$  the Dolgachev-Nikulin-Voisin family over a punctured neighbourhood of the cusp, to obtain a fan 
in the hyperbolic space  $({\check{M}_{2d}})_{\RR}$.  
An important aspect of their work is the construction of a universal family (at least over a neighbourhood of the $0$-dimensional cups).
A detailed description of the GHKS programme for $K3$ surfaces is 
contained in  \cite{theta15}. 
Yuecheng Zhu \cite{zhu18} has carried this out  in the 
case of polarized abelian varieties and has shown that this approach can be used to recover the second Voronoi compactifiation of $\mathcal A_g$, which is known to be a modular compactification by the work of Alexeev \cite{aleAB} and Olsson \cite{OPM}.

Our aim is to start a concrete investigation of the GHKS approach for small degree. 
To be precise, we want to understand the first step in the GHKS programme in degree $2$.
As it turns out this is a nontrivial problem in its own right. The main result of our paper is a concrete description of the Mori fan
of the  Dolgachev-Nikulin-Voisin family in degree $2$. Concretely,  we enumerate the maximal dimensional cones and prove the following:

\begin{theorem}
Let $\shY\to S$ be a model of the Dolgachev-Nikulin-Voisin family. Then $\Morifan(\shY/S)$ has $3460$ maximal cones. Of these $753$ are associated to a model of class $\mathscr{T}$ and $2707$ are associated to a model of class $\mathscr{P}$. The number of orbits of maximal cones of $\Morifan(\shY/S)$ under the natural action of the birational group $\Bir(\shY/S)$ is $588$.
The orbits decompose into $457$ orbits of models of class $\mathscr{P}$ and $131$ orbits of models of class $\mathscr{T}$.
\end{theorem}

This is a consequence of Theorems \ref{modelsT},  \ref{modelsPtotal},   
and \ref{theo:countingcones}. Here class $\mathscr{P}$ and class $\mathscr{T}$ refer to the combinatorial structure of the central fibre, which in turn 
correspond to the two possible triangulations of the sphere $S^2$ into two triangles. 
The knowledge of the Mori fan is necessary for the next step in the GHKS programme: it is used to construct another fan in a certain $19$-dimensional hyperbolic space and it is this fan which 
determines the toroidal compactification of the moduli space of polarized of polarized $K3$ surfaces.      

We also investigate the so called {\em {secondary fan}} which was used by  
by Hacking, Keel and Yue \cite{HKY} and is a generalisation of the secondary fan  for toric varieties due to Gelfand-Kapranov-Zelevinskij \cite{GKZ}. 
This is a coarsening of the Mori fan. Its relevance  is that in the del Pezzo case the toric variety defined by the secondary fan admits a finite morphism to the moduli space  of stable pairs, as was explained in \cite{keeltalk}.
 Note that for toric varieties, the secondary fan and the Mori fan coincide. 

 As there is no published proof available that the secondary fan is
indeed a fan in the $K3$ setting, we will include a proof of this fact using the techniques of our paper and we will also compute its maximal cones:  

\begin{theorem}Let $\shY\to S$ be a model of the Dolgachev-Nikulin-Voisin family of degree $2$. The secondary fan 
contains precisely $4$ maximal cones. There are $2$ orbits of maximal cones under the natural action of $\Bir(\shY/S)$.
\end{theorem}

This is Theorem \ref{teo:countingsecondary} and Remark \ref{orbits:2nd}.

We shall now briefly describe the structure of this paper. We start in Section \ref{review} by recalling the theory of mirror symmetry for lattice polarized $K3$ surfaces, which is due to Dolgachev, Nikulin and Pinkham. We then recapitulate the 
basics of the degeneration theory of $K3$ surfaces as developed by Carlson, Friedman, Kulikov, Person, Pinkham, Scattone and others. This allows us to define the notion of the 
{\em Dolgachev-Nikulin-Voisin family} of degree $2d>0$ and their models in Definitions \ref{def:DNV} and \ref{def:DNVmodel}, following \cite{theta15}.  We further recall the relation between triangulations of the sphere $S^2$ and
$d$-semistable models in $(-1)$-form of the Dolgachev-Nikulin-Voisin family of degree $2d$ with maximal Picard rank for which we describe the geometry of the special fibres in detail (Construction \ref{constructioncomponents}). 
We also prove that these surfaces and their maximal smoothings are
projective (Propositions \ref{optimalglue} and \ref{max:proj}). Finally we  discuss the $(-1)$-models in degree $2$ in detail.  
In Section \ref{sec:mmpmorifan} we start by recalling some basic facts of the minimal model program and introduce the main object of this paper, the Mori fan of the Dolgachev-Nikulin-Voisin family (Definition \ref{def:morifan}). 
We describe its main properties in 
(Proposition \ref{ThmGHKS}), due to the work of Gross, Hacking, Keel and Siebert. Next we explain the relationship between interior facets and flops (Proposition \ref{prop:intconesflops}) and finally discuss the action of the group of birational automorphisms on the Mori fan (Proposition \ref{prop:actionbir}). 

In Section \ref{sec:curvestructures} we mostly specialise to degree $2$. Corresponding to the  two triangulations of the sphere $S^2$ with two triangles we have two possible $d$-semistable $K3$ surfaces in $(-1)$-form, which we denote $Y_{\P}$ 
and $Y_{\T}$  respectively.  
We shall  consider all models whose central fibres can be transformed by a series of type I flops into 
$Y_{\P}$  or $Y_{\T}$, and call these models of type $\P$ and type $\T$ respectively. 
As we shall see later, see Corollary \ref{cor:typeItypeII},  these are all models in degree $2$.
The main object of this section is a detailed analysis of the configuration of certain curves forming an anticanonical divisor on components  of the central fibre of a given model of the Dolgachev-Nikulin-Voisin family. 
This leads to the notion of 
{\em curve structure}, see Definition \ref{def:curvestructure} and will also provide us with a natural $\QQ$-basis of the Picard group of the normalisations of the central fibre (Proposition \ref{qbasis}).
The main application will be existence theorems of ample line bundles of prescribed degree on the components of the anticanonical divisor (Propositions \ref{degample}, \ref{degample2} and \ref{degample2} ). 
This discussion will become vital in Section \ref{sec:projmodel} where 
we prove
projectivity criteria for models of type $\T$ (Proposition \ref{projTcurves}) and type $\P$ (Propositions \ref{nondegproj2}, \ref{nondegproj3}, \ref{nondegproj4} and \ref{nondegproj5} ).    

Section \ref{sec:flopauto} is in many ways the technical heart of the paper. Here we analyse flops between models of the Dolgachev-Nikulin-Voisin family in some detail and study the action of the birational automorphism group on the Mori fan.
For this we introduce various concepts describing  (augmented) curve structures (see Definitions \ref{def:aloneannex}, \ref{def:completeetc} and Construction \ref{reflection}). This allows us to establish  two crucial 
facts about models of the Dolgachev-Nikulin-Voisin family in degree $2$. The first we have already mentioned above, namely that any such model can be related by type I %and type II 
flops to a model of type $\P$ (or equivalently $\T$) (see Corollary \ref{cor:typeItypeII}). The 
second is that any two models of the Dolgachev-Nikulin-Voisin family can be transformed into each other using only type I and type II flops (Corollary \ref{factor}). We note that these results are specific to degree $2$.
Finally, this enables us to study the action of the birational automorphism group of a model of the  Dolgachev-Nikulin-Voisin family on the Mori fan. In Proposition \ref{orbitmain} we determine the possible orbits lengths of maximal 
cones of this fan under the  birational automorphism group, which can be $1$, $3$ or $6$.

In Section \ref{sec:counting} we finally enumerate all models of the Dolgachev-Nikulin-Voisin family in degree $2$ and determine the maximal cones in the Mori fan (Theorem \ref{theo:countingcones}). 
This section is rather combinatorial in nature. To obtain our result we use the tools which we have developed before, in particular we use curve structures: these allow us to characterise the isomorphism classes of the central fibres of projective
models of the Dolgachev-Nikulin-Voisin family. This allows us an explicit enumeration of the models. The second ingredient is the action of the birational automorphism group and its action on the set of maximal cones in the Mori fan which we analysed in the previous section. The main result then follows from a careful enumeration of all models, which we do for types $\P$ (Theorem \ref{modelsP}) and  $\T$ (Theorem \ref{modelsT}) separately. 
In Section \ref{sec:secondaryfan} we finally describe the secondary Mori fan. 
In this section we give a fairly elementary proof that the secondary fan is indeed a fan and compute its maximal cones (Theorem \ref{teo:countingsecondary}).

Throughout the paper we will work over the complex numbers $\CC$.

\begin{acknowledgement}We have benefitted enormously from discussions with Christian Lehn. We thank Mark Gross, Paul Hacking, Sean Keel, Bernd Siebert and Tony Yu Yue for sharing drafts of \cite{theta15} and  \cite{HKY} and  for 
answering our questions. We are grateful to  DFG for financial support under grant Hu 337/7-1.
\end{acknowledgement}

\section{The Dolgachev-Nikulin-Voisin mirror of degree $2d$: definition and construction}\label{review}
\subsection{The Dolgachev-Nikulin-Voisin mirror of degree $2d$}

We first recall the  mirror construction due to Dolgachev  \cite{Dol96} and Pinkham \cite{MR0429876} for polarized $K3$ surfaces, using Nikulin's theory of 
lattice polarized $K3$ surfaces \cite{Nik79}.
For  the basic facts about $K3$ surfaces and their moduli which we will need, we refer the reader to e.g. \cite{Huy}, \cite[Section VIII]{BHPV} and \cite{Handbook}. 
The second cohomology group of a $K3$ surface, together with the cup product (intersection pairing), define a lattice which is  isomorphic
 to the $K3$-lattice 
\[L_{K3}=2E_8(-1)\oplus 3U\] 
where $U$ is the hyperbolic plane and $E_8(-1)$ is the even negative definite  unimodular lattice of rank $8$. 
A polarization on $X$ is an ample line bundle  $\shL$ and, since $K3$ surfaces are regular, we can identify a polarization  with its first Chern class $h=c_1(\shL) \in H^2(X,\ZZ)$.
We assume $h$ to be primitive and of degree $h^{2}=2d >0$. 
We note that the group of isometries of the $K3$ lattice operates transitively on the set of primitive vectors of given positive degree.
Instead of working with the degree of a polarization we will often also use its genus, by which we mean the genus of a general element in the linear system defined by   $\shL$. The genus $g$ and the degree $d$
are related by the 
adjunction formula $2g-2=2d$. Note that degree $2$ coincides with genus $2$.

The moduli theory of $K3$ surfaces builds on the Torelli theorem. To describe this, we first notice that 
the orthogonal complement of $h$ in $L_{K3}$ defines a lattice 
\[L_{2d}\cong 2E_8(-1)\oplus 2U\oplus\langle -2d\rangle.\] 
We obtain the period domain $\Omega_{2d}$ by

\[ \Omega_{2d}=\{x\in \PP({L_{2d}} \otimes \CC)  \mid x^2=0, \langle x\bar{x}\rangle>0\}.\]
This is a $19$-dimensional manifold which has two connected components of which we fix one, say $D_{2d}$. 
A quasi-polarized $K3$ surface is a pair $(X,\shL)$ where $\shL$ is big and nef. 
Recall that a multiple of $\shL$ will embed $X$ as a $K3$ surface with ADE-singularities.
It is a classical application of the Torelli theorem that the moduli space of degree $2d$ polarized $K3$ surfaces with ADE-singularities 
is isomorphic to the quotient 
$$
\shF_{2d}=\Gamma_{2d} \backslash \Omega_{2d} = \Gamma^+_{2d} \backslash D_{2d}
$$ 
where 
$$
\Gamma_{2d}= \{g \in \Og(L_{K3}) \mid g(h)=h \}
$$ 
is the group 
of all isometries of $L_{K3}$ which fix the polarization $h$ and $\Gamma^+_{2d}$ is the subgroup of elements of real spinor norm $1$ (which is equivalent to the property that these elements fix the components of $\Omega_{2d}$.) To obtain the moduli space of all \emph{ polarized} $K3$ surfaces one has to remove finitely many hyperplanes from  $\shF_{2d}$, see e.g. \cite[p. 355]{BHPV}. 

 To simplify the following discussion we will  now assume that $d$ is square free. Then we have a well defined mirror moduli space which was described in \cite[\S 6]{Dol96}.
 This parameterizes lattice polarized $K3$ surfaces of Picard  rank $19$ whose Picard lattice is isomorphic to 
 \[\check{M}_{2d}=U\oplus 2E_8(-1)\oplus\langle -2d\rangle\] 
 and we note that
 \begin{equation}\label{equ:mirror1}
L_{2d}=h^{\perp}_{L_{K3}}=\check{M}_{2d} \oplus U.
  \end{equation}
 By Nikulin's theory the lattice $\check{M}_{2d}$ has a unique primitive embedding into the $K3$ lattice $L_{K3}$ (up to isometries), and here we fix once and for all the obvious 
 embedding, which maps a generator of the  summand
 $\langle -2d\rangle$ to $e- df$  in a summand $U$, where $e,f$ are a basis of $U$ with $e^2=f^2=0$ and $e.f=1$.
Similar to above, this leads to the moduli space 
$$
\check{\shF}_{2d} =\check{ \Gamma}_{2d}\backslash \Omega_{{(\check{M}_{2d})^{\perp}_{L_{K3}}}}
$$
where
$$
 \check{\Gamma}_{2d} = \{g \in \Og(L_{K3})\mid g|_{\check{M}_{2d}} = id \}
 $$ 
 is now the group of 
all isometries of $L_{K3}$ which restrict to the identity on $\check{M}_{2d}$. 
In our case
\begin{equation}\label{equ:mirror2}
{{(\check{M}_{2d})^{\perp}_{L_{K3}}}}  = U \oplus \langle 2d \rangle
   \end{equation}
 and this is dual to relation (\ref{equ:mirror1}).
The period domain $\Omega_{{(\check{M}_{2d})^{\perp}_{L_{K3}}}}$ is $1$-dimensional, more precisely it is two copies of the upper half plane $\HH_1$ (which are interchanged
by the group  $\check{\Gamma}_{2d}$). Hence $\check{\shF}_{2d}$ is a connected (non-compact) modular curve.
Mirror symmetry interchanges the roles of  complex moduli and  K\"ahler moduli. This corresponds to  the fact that  the mirror moduli space $\check{\shF}_{2d}$  is one dimensional 
and that the very general $K3$ surface in $\check{\shF}_{2d}$ has Picard group  $\check{M}_{2d}$.
 
 Since we assumed that  $d$ is square-free it follows from Scattone's calculations in \cite[\S 4]{Sca87},
that there is a unique $0$-dimensional boundary component in the Baily-Borel compactification  of $\shF_{2d}$.
By  \cite[Proposition 7.3]{Dol96} the same is true for  $\check{\shF}_{2d}$.
The mirror family which we are interested in is the universal family over $\check{\shF}_{2d}$ near the cusp. This requires an explanation.
The moduli spaces  $\shF_{2d}$ and $\check{\shF}_{2d}$ do not carry  universal families in the category of schemes (due to the existence of non-trivial automorphisms). 
Nevertheless, this concept can be made precise in the neighbourhood of the cusp and we will do this below where we define the  {\em Dolgachev-Nikulin-Voisin mirror family} in a rigorous way, see Definition \ref{def:DNV}.
 In what follows, we will typically work in the following situation. Let $(R,m)$ be a local complete DVR with residue field $k=R/m$. We will always assume here that $k=\mathbb C$. 
 Let $K= \operatorname{Q}(R)$ be the field of fractions of $R$. We set $S= \operatorname{Spec}(R)$. Typically we will work with the completion 
 $R=\hat{\mathcal O}_{C,p}$ of the local ring of an affine curve. We denote by $0=\operatorname{Spec}(k)$ the closed point of $S$ and by $\eta= \operatorname{Spec}(K)$ the generic point of $S$.
It will essentially be enough to consider the case $R=\mathbb C[[t]]$ of formal power series whose field of fractions is the field  $K=\mathbb C((t))$ of Laurent series. Indeed, if 
$(C,p)$ is a curve germ, then we can choose a local parameter $\pi$, and this defines an isomorphism of $k$-algebras $R \to  \hat{\mathcal O}_{C,p}$.  If $\mathcal Y \to S$ is a scheme over $S$, then we denote the 
generic fibre by ${\mathcal Y}_{\eta}$ and the special  (central) fibre by $\mathcal Y_c$. 
Alternatively, we can also work in the analytic category and consider families $\mathcal Y \to \DD$ over the disc and 
their restriction to the origin $0 \in \DD$ and the punctured disc $\DD^*$ respectively. We will sometimes use the analytic category in proofs.

In this paper we will use the term normal crossing to denote a scheme which is locally (not necessarily globally) normal crossing with reduced components.    
We say that  a normal crossing scheme $Y$ is {\em smoothable} if there exists a regular
scheme $\mathcal Y$, a proper flat map $\shY \to S$ and an isomorphism $\mathcal Y_c \cong Y$. 
In this case the restriction of the normal bundle of $\mathcal Y_c$ in $\mathcal Y$ to the singular locus $D$ of $\mathcal Y_c$ is trivial:
$$
N_D:=N_{\mathcal Y_c/\mathcal Y}|_D= \shO_{\shY}(\mathcal Y_c)|_D \cong {\mathcal O}_D.
$$
 The line bundle $N_D$ is called the {\em infinitesimal normal bundle} and can also be defined purely in terms of the singular scheme $Y$
by using the normal bundle of the components of $D$ in the respective components of $Y$,
see \cite[\S 1]{Frie83}.  We say that $Y$ is $d$-{\em semistable} if $N_D$ is trivial which is a non-trivial condition if we just consider an abstract surface $Y$. The triviality of $N_D$ is a necessary condition for smoothability of $Y$. It was shown by Friedman that it is also sufficient  \cite[Theorem 5.10]{Frie83}.

We will now recall the basic facts about degenerations of $K3$ surfaces in so far as they are relevant for us. We will denote by $Y$ a proper normal crossing surface and by $Y_i$ the components of $Y$. 
If $Y$ is a simple normal crossing (snc) surface, then the components $Y_i$ are smooth. We will, however, also allow self-intersections of the components and we will denote the normalisation
of a component
$Y_i$ by $ Y^\nu_i$. As before, we will denote the singular locus of $Y$ by $D$. We set $D_{ij}=Y_i \cap Y_j$ and consider this as a curve on $Y_i$.  We we also allow $i=j$ and in this case we mean by $D_{ii}$ is 
the self-intersection of the component $Y_i$.  In our case the curves $D_{ij}$ will always be irreducible.  
Intersection numbers $D^2_{ij}$ will always be calculated on the normalisations $Y_i^\nu$. Note that this may depend on the ordering of $\{i,j\}$.

There are three types of degenerations of $K3$ surfaces, classically called type I, II and III. The type of a degeneration is a measure of how far the Hodge structure degenerates.
 Type I are smooth $K3$ surfaces. The building blocks of type II degenerations are rational surfaces and the curves $D_{ij}$ consist of elliptic curves (which still carry some Hodge structure).  
 The components of type III degenerations are rational surfaces, intersecting in curves whose components are also rational. 
 Another characterization of the type can be given in terms of the nilpotency of the monodromy, see e.g. \cite[\S 1.2]{Sca87}.  
 Our interest will be in type III degenerations.  The following definitions are fundamental:
 
 \begin{definition}\cite[5.5]{Frie83}\label{def:dsemistable}
 A $d$-{\em semistable K3 surface of type III} 
 is a normal crossing surface $Y$ such that
 \begin{itemize} 
 \item[(i)] $Y$ is $d$-semistable
 \item[(ii)] $\omega_Y=\O_Y$
 \item[(iii)] $Y=\cup Y_i$ where each $Y_i$ is rational and the preimage of the double curves $\sum_j D_{ij}$ form anticanonical cycles of rational curves on the normalisation $Y^\nu_i$
 \item[(iv)] The dual intersection complex of $Y$ is a triangulation of the $2$-sphere $S^2$.
 \end{itemize}  
 \end{definition}

In the projective situation we make the following 
 
 \begin{definition}\label{def:typeIIIdeg} 
 A {\em type III degeneration} of $K3$ surfaces is a flat, projective scheme $\shY\to S$ over the spectrum  $S$ of a complete DVR, where $\shY$ is a regular $3$-fold whose generic fibre $\shY_{\eta}$  is a $K3$ surface and whose 
 central fibre $\shY_c$ is a type III $d$-semistable $K3$ surface. We will also refer to such a family $\shY\to S$ as a \emph{Kulikov model}.
 \end{definition}
\begin{remark} 
Note that the total space $\shY$ is Calabi-Yau, i.e. $\omega_{\shY}=\shO_\shY$. 
\end{remark}
\begin{remark} 
By definition the central fibre of a Kulikov model is always projective. We note, however, that a priori Definition \ref{def:dsemistable} does not require $Y$ to be projective. 
\end{remark}

\begin{remark}
One can also consider (not necessarly projective) \emph{analytic} smoothings. In the analytic category, a type III degeneration of $K3$ surfaces is a morphism $\shX\to \mathbb{D}$, with $\mathbb{D}$
a small disc, such that  the general fibres are smooth  $K3$ surfaces and the central fibre is a  type III $d$-semistable $K3$ surface as above. In particular, such an $\shX$ is smooth.  
Every $d$-semistable surface admits an analytic smoothing by \cite[Theorem 5.10]{Frie83}.
\end{remark}

Next, we discuss certain modifications of $d$-semistable $K3$ surfaces, the \emph{elementary modifications}. Recall that an \emph{$F$-flopping contraction}  of a threefold $\shY\to S $ with trivial canonicial class, where   $F$ is  a $\QQ$-Cartier divisor,  is a proper  birational contraction $f\colon \shY\to \shZ$  to a normal scheme (or complex analytic space) $\shZ \to S$ such that the exceptional locus is of codimension at least $2$ and the divisor  $-F$ is $f$-ample.  An \emph{$F$-flop} of $\shY\to S$ is a  scheme  $\shY^+\to S$ together with a proper birational morphism $f^+\colon\shY^+\to \shZ$ such that the birational transform $F^+$ of $F$ is $\QQ$-Cartier  on $\shY^+$ and $F^+$ is $f^+$-ample and the exceptional locus of $f^+$ has codimension at least $2$. The induced birational map $\phi\colon\shY\dashrightarrow \shY^+$ is, by abuse of language,  also called a {\em flop} of $\shY$.
The situation can be summarized by the following diagram
\[
\xymatrix{ \shY\ar@{-->}[rr]^\phi\ar[dr]_f&&\shY^+\ar[dl]^{f^{+}}\\
&\shZ.&
}
\] 

If $\shY\to \shZ$ is the contraction of an extremal ray, then the $F$-flop is independent of the choice of $F$, see \cite[\S 6.1]{KoMo}. In this paper, we will always assume that a flop is  given by a contraction of an extremal ray. 

One can also consider flops in the analytic category. Here we will recall  certain types of  analytic flops, the elementary modifications.
 There are three types of these, classically known as type 0,I and II \cite{FrMo}. Here we will only be concerned with elementary modifications of type I and II.
For the convenience of the reader we will recall these here.  Let $Y$ be a $d$-semistable $K3$ surface as above. Let $\shX \to \mathbb{D}$ be an \emph{analytic} smoothing over a small disc.  Let $C \subset Y$ be a smooth rational curve which intersects the double locus $D$ 
in exactly one point, more precisely, transversally in a point of some $D_{ij}$ which is a smooth point of $D$.
The curve $C$ lies on a unique component $Y_i$ and we assume that $C^2=-1$ on the normalisation $Y_i^{\nu}$. Then one can blow up $\shX$ in $C$ and the exceptional divisor will be
isomorpic to $\mathbb P^1 \times \mathbb P^1$ with normal bundle $\shO_{\mathbb P^1 \times \mathbb P^1}(-1,-1)$. The blow-down map contracts one ruling of the exceptional divisor to $C$. Contracting the other ruling gives another 
model $\shX' \to  \mathbb{D}$. The exceptional divisor contracts to a curve $C'$ on $Y_j$ and we have {\em flopped} the curve $C$ to  $Y_j$. The curve $C'$ is again a $(-1)$ curve on $Y_j^{\nu}$. 
Here we allow $Y_j$ and $Y_i$ to coincide.  
This defines  an \emph{elementary modification of type I}, see Figure \ref{emtypeI}. This construction induces a modification $\psi\colon Y\dashrightarrow Y'$ of the central fibre $Y$. Following standard 
terminology we will  also refer to this induced map on the central fibre as an elementary modification of type $I$.

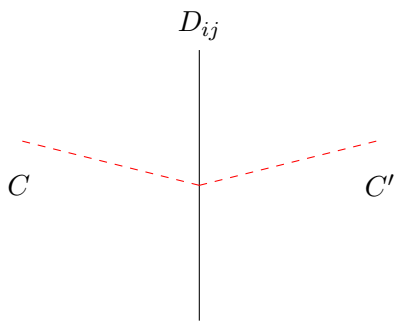
\begin{figure}\centering
\begin{tikzpicture}[scale=0.6]
\draw[] (0,0)--(0,6);

\draw[dashed, color=red] (0,3)--(-4,4)
(0,3)--(4,4);
\node(DIJ)[above] at (0,6) {$D_{ij}$};

\node(C1) at (-4,3) {$C$};
\node(C2) at (4,3) {$C'$};
  \end{tikzpicture}
\caption{An elementary modification of type I}
\label{emtypeI}
\end{figure}

Alternatively, let $C$ be a smooth rational component of the double curve $\sum_j D_{ij}$ with $C^2=-1$ on both $Y_i^{\nu}$ and $Y_j^{\nu}$, where we again allow the components to coincide. Blow up $\shX$ in $C$; the resulting exceptional divisor will again be   $\mathbb P^1 \times \mathbb P^1$.
As before, we can contract the other ruling to obtain a degeneration  $\shX' \to  \mathbb{D}$, yielding  an \emph{elementary modification of type II}, see Figure \ref{emtypeII}.  Again, one obtains a modification $Y'$ of the surface $Y$ to which we will also refer to as elementary modification of type II. \\

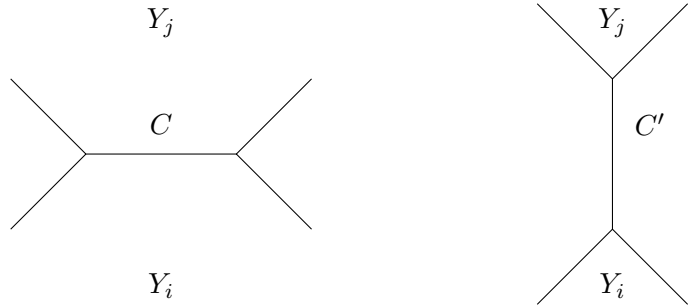
\begin{figure}[]\centering
\begin{tikzpicture}[scale=1]
\draw[] (-5,1)--(-4,0)--(-2,0)--(-1,1) %first
(-5,-1)--(-4,0)
(-2,0)--(-1,-1) 

(3,-1)--(3,1)--(2,2)% second
(3,1)--(4,2)
(3,-1)--(2,-2)
(3,-1)--(4,-2);
\draw[white](0,3)--(2,3);% line to get space on top.
%labels
\node(A) at (-3,1.75) {$Y_j$};
\node(B) at (-3,-1.75) {$Y_i$};
\node(C) at (3,1.75) {$Y_j$};
\node(D) at (3,-1.75) {$Y_i$};
\node (E) at (-3,0.4) {$C$};
\node (F) at (3.5,0.4) {$C'$};
\end{tikzpicture}
\caption{An elementary modification of type II}
\label{emtypeII}
\end{figure}

%%%%%%%%%%%%%%%%%%%%%%%%
By the nature of birational geometry for threefolds, for each Kulikov model, there are in general  many birational Kulikov models. One can, however, pick out a special class of such models, namely the 
Kulikov models in $(-1)$-form.  

\begin{definition}\label{def:-1form}
Let $Y$ be a type III  $d$-semistable $K3$ surface. 
Then we say that $Y$  is in $(-1)$-{\em form} if 
for each smooth double curve $D_{ij}=Y_i\cap Y_j$, we have $D^2_{ij}=D^2_{ji}=-1$ and if $D_{ij}$ is singular (and hence a nodal rational curve) and $Y_i$ is singular, then $D^2_{ij}=1$  and $D^2_{ji}=-1$.

\end{definition}
We will call a Kulikov model $\shY\to S$   with central fibre $Y$ in $(-1)$-form a \emph{ Kulikov model in $(-1)$-form}.
Note that by  a theorem of Miranda and Morrison \cite[Main Theorem 1.2]{MiMo} 
any \emph{analytic}  type III Kulikov model can, by a series of elementary modifications of type I 
and II, be brought into $(-1)$-form, but that $(-1)$-forms are still not unique. As discussed above,  one can also think of this sequence of modifications as a sequence of modifications on the central fibre. Thus one may interpret the theorem of Miranda and Morrison as a result on type III $d$-semistable $K3$ surfaces. 
\begin{theorem} \cite[Main Theorem 1.2]{MiMo} \label{teo:-1form}
Let $Y$ be a $d$-semistable $K3$ surface of type III. Then there is a sequence $Y\dashrightarrow Z $  of elementary modifications of type I and II such that $Z$ is a $d$-semistable $K3$ surface in $(-1)$-form.  
\end{theorem}

%%%%%%%%%%%%%%%%%%%%%%%%%%%%%%%%%%%%%%%%%%%%%
For future use we also want to recall the index of the monodromy. Let $Y$ denote a $d$-semistable $K3$ surface in $(-1)$-form, let  $\Gamma$ be its dual graph, a triangulation of the sphere $S^2$. We say $Y$ has \emph{ special $n$-bands of hexagons} if $\Gamma$ is a refinement of another triangulation $\Gamma'$ of $S^2$, and is in fact obtained from $\Gamma'$ by subdividing each triangle of $\Gamma'$ into $n^2$ triangles, see Figure \ref{Fig:bands} for $n=4$.

\begin{centering}\begin{figure}
\begin{tikzpicture}
\draw[] (0,0) --++(60:1)--++(-60:1)--++(-180:1);
\draw[] (1,0) --++(60:1)--++(-60:1)--++(-180:1);
\draw[] (2,0) --++(60:1)--++(-60:1)--++(-180:1);
\draw[] (3,0) --++(60:1)--++(-60:1)--++(-180:1);
\draw[] (0,0)++(60:1)--++(0:3);

\draw[] (0,0)++(60:1)--++(60:1)--++(-60:1)--++(-180:1);
\draw[] (1,0)++(60:1)--++(60:1)--++(-60:1)--++(-180:1);
\draw[] (2,0)++(60:1)--++(60:1)--++(-60:1)--++(-180:1);
\draw[] (0,0)++(60:2)--++(60:1)--++(-60:1)--++(-180:1);
\draw[] (1,0)++(60:2)--++(60:1)--++(-60:1)--++(-180:1);
\draw[] (0,0)++(60:3)--++(60:1)--++(-60:1)--++(-180:1);
\draw[] (0,0)--++(60:4)--++(-60:4)--cycle;

\draw[] (-6,0)--++(60:4)--++(-60:4)-- cycle;

\node(A) at (-2,2){};
\node(B) at (0,2){};
%\node (A) at (-6,0)++(60:4)++(-60:2){};
%\node (B) at (0,0)++(60:2) {} ;
\draw[->] (A) --(B);

\end{tikzpicture}
\caption{ The refinement of a triangle of $\Gamma'$.}
\label{Fig:bands}
\end{figure}
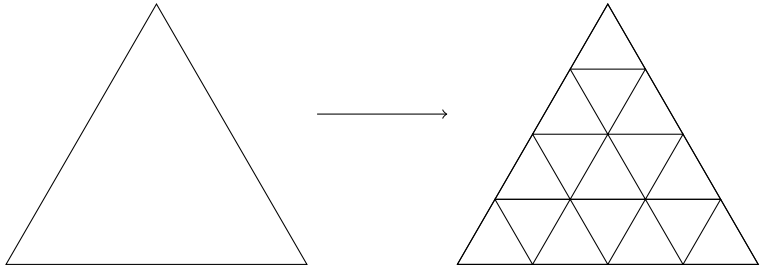
\end{centering}

Now let  $Y$  be any $d$-semistable $K3$ surface of type III. Let $Y'$ denote a $d$-semistable $K3$ surface in $(-1)$-form that is obtained from  $Y$ by a sequence of elementary modifications. Let $k$ be the largest integer such that $Y'$ has  special $k$-bands of hexagons. Note that this number is independent of the choice of $Y'$ by \cite[Theorem 0.5]{FS}.

\begin{definition}
The integer $k$ is  called the \emph{index} of $Y$. If $k=1$,   $Y$ is called {\em primitive}.
\end{definition}
\begin{remark}The index is an invariant of $Y$ that may also be defined in terms of monodromy in the analytic setting. For details see \cite[p. 4]{FS}.
\end{remark}

We finally have to recall some facts about the Picard group of $d$-semistable $K3$ surfaces and their smoothings. 
We first start with a normal crossing surface $Y=\cup Y_i$  fulfilling the conditions of Definition \ref{def:dsemistable} with possibly the exception that  $Y$ is $d$-semistable. Then there 
is an exact sequence
\begin{equation}\label{equ:L}
0 \to L \to \bigoplus_i H^2(Y^{\nu}_i,\mathbb Z) \to \bigoplus_C H^2(C,\mathbb Z)
\end{equation}
where $C$ runs through all  components of preimages of $D$ under the normalisation $Y^\nu\to Y$  
and $L$ is defined as  the kernel of the right hand map given by the differences of the restrictions, see e.g.  \cite[p. 151]{AGIII}. By \cite[Section 3.1]{Laz08}, $\rank L= 18+n$, with $n$ the  number of components of $Y$. Obviously, $\Pic(Y)\subset L$.
Recall also from \cite[\S 4]{Carl} 
the Carlson homomorphism
$$
c_Y: L \to \mathbb C^*.
$$ One way to define it as follows, see \cite[\S 3]{FS}: Let $W_0=W_1\subset W_2$ be the weight filtration of the natural mixed Hodge structure on $H^2(Y)$. We have $L=(W_2/W_0)_\ZZ$.  
Choose a section $s$ of the projection $W_2\to W_2/W_0$ 
preserving the Hodge filtration and a retraction $r\colon W_2\to W_0$ of the inclusion $W_0\subset W_2$. Then $c_Y=r\circ s \text{ mod }(W_0)_\ZZ \colon L\to ({W_0})_\CC/ ({W_0})_\ZZ$.

Its significance is that 
\begin{equation}\label{equ:Carlson}
\Pic(Y)= \ker(c_Y).
\end{equation}
We note that for each component $Y_i$, there is an element $\xi_i=\sum_j D_{ij}-D_{ji} \in L$. By results of Friedman and Scattone \cite[p. 25]{FS}, $Y$ is smoothable if and only if $\xi_i\in \ker(c_Y)$ for all $i$.

We can then consider locally trivial deformations  of $Y$, i.e. deformations $\mathfrak{X}\to B$, with $B$ an analytic set, such that each point of  $\mathfrak{X} $ has a neighbourhood $U$ such that $\mathfrak{X}_{|U}$  is a product. 
By \cite[\S 4]{FS} these are parameterized by the Carlson map $c_Y$.
The $d$-semistable deformations define a divisor in this $20$-dimensional family.
From the relation (\ref{equ:Carlson}) we obtain in particular the following result.

\begin{lemma}\label{degree}
Let $Y=\cup Y_i$ be a $d$-semistable $K3$ surface of type III with $c_{Y}=1$. Then 
\[\Pic(Y)\cong \{ (\shL_i)_i\in \bigoplus_{i}\Pic({Y_i^\nu}) \mid \deg \shL_{i_{|C}}=\deg\shL_{j_{|C}}, C \subset D^\nu\}\]
where $D^\nu$ denotes the preimage of $D$ under the normalisation map $\nu \colon Y^\nu\to Y$  and
$C$ runs through all components of $D$.
\end{lemma}

The following Lemma is an application of  \cite[Lemma 5.5]{FS}.  

\begin{lemma}\label{carlsonunique}
Let $Y$ be a $d$-semistable K$3$ surface with $n$ components. There is a unique locally trivial $d$-semistable deformation $Y_c$ of $Y$ such that $\text{rank } \Pic(Y_c)=18+n$. 
The Carlson homomorphism $c_{Y_c}$ is trivial, i.e. $c_{Y_c}=1$ and hence $\Pic(Y_c)=L$. 
\end{lemma}
\begin{proof}A detailed proof is in \cite[Section  10.4]{theta15}.
If $Y$ is such that its Picard group has rank $18+n$, one can in particular choose $19$ linearly independent divisors $L_i$ different from the $\xi_i$. The result then follows by iterating the proof of  \cite[Lemma 5.5]{FS}.
\end{proof}

\begin{remark}
We will now consider deformations $\shY\to S$ of $d$-semistable K$3$ surfaces. The relevant relative notions of divisors and cones are recalled in Section \ref{sec:mmpmorifan}. Here we simply remark that under our assumptions on $S$, linear equivalence over the base coincides with the usual linear equivalence, in particular   $\Pic(\shY/S)=\Pic(\shY)$. 
\end{remark}

The following will play an important role for us. 
\begin{proposition}\label{lem:unoquedeg} 
Let $Y$ be a projective type III  $d$-semistable  $K3$ surface with  $c_Y=1$.  
Then there is  an up to isomorphism unique  Kulikov model $\shY \to S$ of $Y$ such that
\begin{equation}\label{equ:maximal} 
r_c\colon\Pic(\shY/S) \cong \Pic(Y)
\end{equation}
with $r_c$ being the restriction map.  Also, if $k$ is the index and $t$ is the number of triple points of $Y$ 
 then
 \[\Pic(\shY_\eta)\cong 2E_8(-1)\oplus U\oplus\langle \frac{-t}{k}\rangle.\]
\end{proposition}

\begin{proof}
This goes back to \cite{FS}.
If $c_Y=1$, then $\Pic(Y)\cong L$. The surface $Y$ also determines the 
invariants $t$,$k$. Let $n$ be the number of components of $Y$. Consider  the  divisors $\xi_i=\sum_j D_{ij}-D_{ji}$, $i=1,\dots,n$  defined  above. 
Note that  $\sum_i \xi_i=0$. The lattice $\langle \xi_1,\dots \xi_n | \sum_i \xi_i=0\rangle $ is a primitive sublattice of $L$ by \cite[(4.13)]{FS}.
Hence we can pick linearly independent divisors $L_1,\dots L_{19}$ that generate $\Pic(Y) \text{ mod } K$. If $\shY\to S$ is a deformation with $\Pic(\shY/S)\cong \Pic(Y)$ via restriction, then by definition $\shY\to S$ is a 
deformation of the tuple $(Y; L_1,\dots, L_{19})$. We shall show that there is a unique such 1-parameter deformation.

Let $\mathfrak{X}\to V$ be the semiuniversal  analytic deformation of $Y$ as defined in in the proof of  \cite[Theorem 5.10]{Frie83}. By the arguments in the proof of  \cite[Lemma 5.5]{FS}, 
the locus $V'$  in the smoothing component  of $V$ where the $L_i$ deform is $1$-dimensional and smooth. 
Let $\mathfrak{X}'\to V'$ be the restriction of the semiuniversal family. By \cite[Theorem 5.10]{Frie83}, this is a smoothing of $Y$. Let $W'$ be the analytic algebra defining the germ $V'$, and let $W$ be the completion of $W'$ with respect to the maximal ideal. 
Then $W\cong \CC[[t]]$. This defines a formal scheme $\hat{\shY}\to \Spf \CC[[t]]$, and by the condition that all $L_i$ deform, there is an $\shL\in \Pic(\hat{\shY})$ restricting to an ample line bundle on $Y$ and thus by Grothendieck's existence theorem a deformation $\shY\to S$ with $S=\Spec\CC[[t]]$ such that $\hat{\shY}$ is the completion of $\shY\to S$ along $Y$. By construction, $\Pic(\shY/S)\cong\Pic(Y)$ via restriction.

We show that $\shY\to S$ is a smoothing of $Y$. The total space of the deformation  $\mathfrak{X}'\to V'$  is smooth, as follows from \cite[Theorem 5.10]{Frie83}.
In particular, its local rings in closed points are regular, and thus by \cite[Theorem 23.7]{Mats} the local rings $\O_{\hat{\shY},y}$ for $y\in Y$ of the formal smoothing are regular. By the same theorem, this implies that the stalks of $\O_{{\shY},y}$ 
at closed points of the central fibre of $\shY\to S$ are regular local rings. This implies that $\shY$ is regular by \cite[Remark 6.25]{GW}.
 In particular, the generic fibre is a smooth K$3$ surface. Also, by adjunction, $\shY$ has trivial canonical bundle.  So $\shY\to S$ is indeed a semistable model. 
 
Also, for the 
degeneration $\shY \to S$, it follows from \cite[Lemma 4.2]{Kaw97}, using the fact that $S$ is a $DVR$, that we have an exact sequence  
\begin{equation}\label{equ:basicsequ}
0 \to \ZZ^\shY\to \Pic(\shY/S)\to \Pic(\shY_\eta)\to 0
\end{equation}
with $\ZZ^\shY$ the abelian group generated by the components $Y_i$ of the central fibre modulo the relation $\sum_i Y_i=0$.
The statement  about the Picard group of the generic fibre then follows from \cite[Prop. 4.3]{Laz08} and \cite[Corollary 4.6]{Laz08} together with Sequence \ref{equ:basicsequ}. 
 So $\shY\to S$  is a model as claimed.

Now, suppose $\shY'\to S$ is a second  model . By formal semiuniversality,  
$\shY'\to S$  is pulled back from $\shY\to S$ via a homomorphism $\CC[[t]] \to \CC[[t]]$. Because $\shY'\to S$  is regular, the uniformizing parameter $t$ maps to $at$ with $a$ a unit. Hence $\shY'\to S$ is canonically isomorphic to $\shY\to S$. This proves the result. 
\end{proof}

\begin{definition} 
We shall call a $d$-semistable K$3$ surface $Y$ with $c_{Y}=1$  {\em maximal} and  a degeneration $\shY \to S$ of a maximal K$3$ surface  with $r_c\colon\Pic(\shY/S) \cong \Pic(Y)$  
a {\em maximal degeneration}.
\end{definition}

The next proposition says that maximal degenerations behave well under flops.

\begin{proposition} \label{modelflop}
Let  $\shY\to S$ be  a maximal degeneration with central fibre $Y=\shY_c$ and 
let  $\shY^+\to S$ be a flop of $\shY\to S$. 
Then  $\shY^+\to S$ is again a maximal degeneration and the dual graph  of the central fibre  $Y^+=\shY^+_c$ is a triangulation of $S^2$ with the same number of triangles as the dual graph of $Y$.   
\end{proposition}
\begin{proof} 
We shall prove this result using the analytic theory.
For this we first note that any  flop factors into flops given by contractions of extremal rays, see \cite[\S 6.4]{KoMo}.
Thus we can assume the flop is given by a small contraction $\contr_R$ with $R$ an extremal ray of $\shY\to S$.
  Let $\pi\colon\shY\to \bar{\shY}$ be the flopping contraction over $S$, $F$ be a divisor that is anti-ample on the fibres of $\pi$. The morphism $\pi$ is  
  given by a divisor $G$ on $\shY$ with restriction $G_c$ to $Y$. This  defines a divisor  $G'$ on the  maximal analytic smoothing $\shX\to\mathbb{D}$ of 
  $Y$ and a contraction  $\shX\to \bar{\shX}$ of the extremal ray $R$.
   Hence $\shX\to \bar{\shX}$ is a small contraction. Restricting the  divisor $F$ to the central fibre 
   and extending by maximality to $\shX$ we obtain the induced divisor $F'$ on $\shX$.  
  This is anti-ample on the fibres of $\shX\to 
  \bar{\shX}$. Hence there is a flop $\shX^+\to \bar{\shX}$. By a result of Kulikov, $\shX^+\to \mathbb{D}$ is semistable, see e.g. \cite[Corollary 3.7]
  {Corti95}.  
 Completion along the central fibre gives morphisms of the corresponding formal schemes, and by Grothendieck's existence theorem, we get  proper
 morphisms $f\colon \shY\to \shZ$ and $f^+\colon \shY^+\to \shZ$. We claim that the morphism $f\colon \shY\to \shZ$ is given by the contraction of $R$. Indeed, we have $f_*\O_{\shY}\cong \O_\shZ$: the restriction to the central fibres $f_c\colon \shY_c\to \shZ_c$ is a proper surjective birational morphism with connected fibres.  It is is straightforward to check that  $(f_c)_*\O_Y\cong\O_Z$, using that every regular section of $\O_\shY$ is constant on the fibres.
 By the global version of \cite[Lemma 1.2]{wahl76},  
we have $(f_n)_*\O_{\shY_n}\cong \shO_{\shZ_n}$ for the truncations of order $n$ 
and thus $\hat {f}_*\widehat{\O_{\shY}}\cong\widehat{\O_{\shZ}}$, by \cite[Theorem 8.2.2]{FGAex}. By the same theorem, $\hat {f}_*\widehat{\O_{\shY}}\cong\widehat{f_*\O_\shY}$, so  $\widehat{\O_{\shZ}}\cong\widehat{f_*\O_\shY}$ and from  \cite[Theorem 8.4.2]{FGAex}, it follows that $f_*\O_\shY\cong \O_\shZ$. 
 As $\shZ$ is a Nagata scheme, its normalisation is finite over $\shZ$. Because of the universal property of normalisation and connectedness of the fibres of $f$, it follows from finiteness that $\shZ$ is normal. By uniqueness of contractions, $f=\pi$ and, in particular $\bar{\shY}=\shZ$.

 Hence $f\colon \shY^+\to \shZ$ is the flop of $\pi$. It has central fibre $Y^+\cong (\shX^+)_c$, which is a $d$-semistable K3 surface of type III.    
 As  $Y^+$  is the central fibre of  a degeneration, it follows that the $\xi_i$ classes from above are in fact Cartier, so the dual graph is indeed a triangulation of the sphere $S^2$ by Kulikov's theorem  \cite[Theorem II]{Kul77}, 
 with the same number of triangles as the dual graph of $Y$, because the number of components is the same. It has trivial Carlson extension since it follows from $\text{rank }\Pic(\shY^+/S)=18+n$ that its Picard group has rank $18+n$, with $n$ the number of components of $Y^+$,  so by  Lemma \ref{carlsonunique} it follows that  $c_{Y^+}=1$. This proves the proposition.
\end{proof}

We shall now consider the special case arising from the mirror families of $2d$-polarized $K3$ surfaces.  We will also
construct explicit models. Here we first state the more general  
\begin{proposition}\label{modelexist}
Let $d>0$ and $\check{M}_{2d}=U\oplus 2E_8(-1)\oplus\langle -2d\rangle$.
Then there exists a primitive maximal Kulikov model $\shY \to S$ such that 
$\Pic(\shY_{\eta}) \cong \check{M}_{2d}$.
Any two such Kulikov models are related by a sequence of  flops. \end{proposition}
\begin{proof}
The existence follows from Proposition \ref{lem:unoquedeg},
all we require is the existence of a primitive type III $d$-semistable K3 surface with $t=2d$ triple points and primitivity $k=1$, which exists by \cite[Theorem 0.6]{FS}.
We now show that two such models are related by flops. For this  let $\shY \to S$ and $\shY'\to S$ be two distinct deformations with the properties stated and central fibres $Y$ and  $Y'$ respectively.
 It follows from Sequence (\ref{equ:basicsequ}) and e.g.  \cite[Proposition 4.3]{Laz08} that both central fibres have 
 exactly  $2d$ triple points. Also, both degenerations are primitive. 
As above, we have the maximal analytic family  $\shX\to \mathbb{D}$ over a small disc with central fibre $Y$, $\Pic(\shX)\cong\Pic(Y)=L$ and smooth fibres  K$3$ surfaces with Picard rank $19$. 
Similarly we have $\shX'\to\mathbb{D}$ with central fibre $Y'$.

By \cite[Theorem 0.6]{FS}  
there is a sequence of  type I and type II modifications $\shX\dashrightarrow\shX'$. As both models are projective, this sequence factors into a sequence of projective flops given by contractions of extremal rays,  by \cite[Remark 6.37]{KoMo}.  Let $F$ be an effective divisor defining the first flop, say $ \shX \dashrightarrow \shX^+$ in this sequence.  Restriction to the central fibre and then lifting to $\shY\to S$ via maximality yields an effective divisor $F'$ inducing a contraction of an extremal ray.  As the exceptional locus on $Y$ does not deform to $\shY$  (because it does not deform on $\shX$, so it does not fullfill the numerical conditions for lifting to the family), we obtain a flop $\shY^+$ of $\shY$  with $\shY^+_c=\shX^+_c$. 
Proceeding in this way, one  obtains a sequence of flops 
 \[\shY\dashto\dots\dashto \shY^{''}\] such that the central fibre of $\shY^{''}$ is $Y'$. By uniqueness of maximal smoothings, see Proposition \ref{lem:unoquedeg} , $\shY'\cong\shY^{''}$ over $S$. 
\end{proof}

%%%%%%%%%%%%%%%%%%%%%%%%%%%%%%%%%%%

\begin{remark} Proposition \ref{modelflop} shows that the generic fibre $\shY_{\eta}$ of a maximal primitive degeneration $\shY \to S$ with  $\Pic(\shY_{\eta})=\check{M}_{2d}$ is 
independent of $\shY \to S$.  
\end{remark}

This allows us to make the following definition, where we follow the established terminology in the existing literature:
 \begin{definition}\label{def:DNV}  Fix $2d>0$ with $d$ square free. The (unique)
 $K3$ surface $\shY_{\eta}\to \Spec(\CC((t)))$ with $\Pic(\shY_{\eta})=\check{M}_{2d}$ which is the generic fibre of some
maximal primitive smoothing  $\shY \to S$ 
 is called the \emph{Dolgachev-Nikulin-Voisin family of degree $2d$}.
 \end{definition}

 \begin{remark}
 Here, we require $d$ to be square free to obtain the Dolgachev-Nikulin-Voisin family as in \cite{theta15}. For $d$ not square free, there are several distinct local models. 
 \end{remark}

\begin{definition}\label{def:DNVmodel}
We fix $2d>0$ with $d$ square free. Then any  primitive projective type III degeneration $\shY\to S$  that is maximal  and whose generic fibre has Picard group  $
\check{M}_{2d}$ is called a {\em model of the Dolgachev-Nikulin-Voisin family} of degree $2d$. 
\end{definition}

We will abbreviate Dolgachev-Nikulin-Voisin family by DNV family.
We note the following result. 

\begin{proposition}\label{-1optimal}Let $Y$ be a maximal projective $d$-semistable $K3$ surface.   Then there exists a sequence of elementary modifications of type I and II
\[ Y\dashrightarrow Y_1 \dashrightarrow \dots\dashrightarrow Y_i \dashrightarrow \dots \dashrightarrow Y_n\] such that all $Y_i$ are maximal $d$-semistable $K3$ surfaces and $Y_n$ is in $(-1)$-form.
\end{proposition}
\begin{proof}
Let $\shX\to\mathbb{D}$ be the maximal analytic smoothing of $Y$. By the $(-1)$-Theorem, see Theorem \ref{teo:-1form},
there is a sequence of type I and type II modifications 
\begin{equation}\label{seq:analytic} 
\shX\dashrightarrow \shX_1 \dashrightarrow \dots\dashrightarrow \shX_i \dashrightarrow \dots \dashrightarrow \shX_n
\end{equation} 
such that the central fibre of $\shX_n$ is in $(-1)$-form.
The Picard groups of the  $\shX_i$ all have the same rank and hence the central fibres are maximal.  
Thus restricting the sequence (\ref{seq:analytic}) to the central fibre gives the result.
\end{proof}

\subsection{Construction of Models}
 There is a bijection between  triangulations  $\mathscr{G}$ of the sphere such that no vertex has valency greater than $6$ and locally trivial deformation classes  $[Y]_\mathscr{G}$ of  d-semistable K3 surfaces of type III in $(-1)$-form, see 
 \cite[\S 5.1]{Laz08}. By \cite[\S 3.9]{FS}, 
there is a unique element  $Y_\mathscr{G} \in  [Y]_\mathscr{G}$ with trivial Carlson extension. Here we describe the possible components of this surface.  
Each vertex $v$ of  $\mathscr{G}$ corresponds to a component with normalisation a (weak) del Pezzo surface of degree the valency of $v$. The edges of $\mathscr{G}$ then determine which components are glued.
We will describe the gluing later. We emphasize  that this construction is at least implicit in \cite{theta15}.

We first make the following definition.
 \begin{definition}
Let $(Y,D)$ be an anticanonical pair. Let $D=\sum D_i$ and $p$ be a smooth point of $D$ lying on the component $D_i$.  If $n=1$, the \emph{ $n$-fold blow-up  } of $Y$ in $p$ is the usual blow-up, if $n>1$,  the n-fold blow-up of $Y$ in $p$ is the blow-up of the $n-1$-fold blow-up $\pi\colon Y'\to Y$ in the point $\text{ex}(\pi)\cap D_i'$, where $D_i'$ is the strict transform of $D_i$ on $Y'$. More generally, if $(p_1,\dots ,p_k)$ is an ordered  set of points $p_j\in Y$,
which lie on the smooth locus of $D$ such that each $D_i$ contains at most one point $p_j$,
we define by the obvious generalisation the \emph{$(n_1,\dots, n_k)$-blow-up of $Y$ in $(p_1,\dots, p_k)$}.
\end{definition}

 \begin{construction}\label{constructioncomponents}
 We now describe the  possible anticanonical pairs $(Y,D)$ that appear as (normalisations) of components of  $Y_\mathscr{G}$.
  By \cite[Proposition 5.2]{Laz08}, we know the description in general terms. The  (normalisations of the) components are (weak) del Pezzo surfaces of degree $d=1,\dots,6$ and the double locus of the 
 central fibre gives rise to an anticanonical cycle on the normalisations of the components. The pairs $(Y,D)$ we are looking for are rigid pairs. To obtain rigidity we first demand that $D$ be a cycle of $(-1)$-curves 
 if $d>1$. In degree $d=5,6$ this leads to a unique pair $(Y,D)$ where $Y$ is a del Pezzo surface. For $d=1, \ldots ,4$ we obtain weak del Pezzo surfaces and we ensure rigidity by asking that $Y$ contains a
 maximal number of $(-2)$-curves.  This leads to unique pairs $(Y,D)$, see  \cite[Lemma 5.14]{Laz08}. Note that
the configurations of $(-2)$-curves defines a root lattices. 
Contracting the $(-2)$-curves one obtains singular del Pezzo surfaces with a unique maximal ADE-singularity.
We further specify  \emph{special} points on each component of the anticanonical cycle $D$.   
We will call a special point \emph{interior} if it is not a node of $D$. The special points will be used to define the gluing of  $Y_\mathscr{G}$.  We will now describe the pairs $(Y,D)$ and the special points explicitly.

\begin{itemize}
\item[d=1:] Let $Q=\PP^1\times\PP^1$ with toric boundary $\tilde{D}=\tilde{D}_1+\tilde{D}_2+\tilde{D}_3+\tilde{D}_4$, ordered cyclically. Let $p_i\in \tilde{D}_i$, $i=1\dots4$, be points in the smooth part of $\tilde{D}$ such that $p_i, p_{i+2}$ are in the same fibre of one of the two  rulings, see the right hand side of Figure \ref{model2}.
Let  $\tilde{Q}$ be  the $(1,5,1,3)$-blow-up of $Q$ in $(p_1,p_2,p_3,p_4)$.
 The strict transforms of  $\tilde{D}_1,\tilde{D}_3$ have self-intersection $-1$ on $\tilde{Q}$. Blowing 
 down these yields a surface such that the strict transform of $\tilde{D}_4$ has  self intersection $(-1)$.
Then, blowing this down gives a surface $\mathfrak{Y}_1$ with an anticanonical cycle ${D}$ of 
self intersection $1$ and an $E_8$ root system of effective $(-2)$-curves. This is a 
weak del Pezzo of degree 1, see Figure  \ref{d1bild}. 
There is a unique $(-1)$-curve $E$ meeting ${D}$. The special points are the node of ${D}$ and the point  $E\cap{D}$ (which is an interior special point).

\begin{figure}[]\centering
\begin{tikzpicture}[scale=1]
\draw[] (0,0)--(3,0)--(3,3)--(0,3)-- cycle; %left rectangle

\draw[->] (4, 1.5)--(6,1.5); % map

%points left
\draw[fill=black] (1.5,0) circle (0.03cm);
\draw[fill=black] (1.5,3) circle (0.03cm);
\draw[fill=black] (0,1.5) circle (0.03cm);
\draw[fill=black] (3,1.5) circle (0.03cm);

%curves
\draw[] (1.5,3) .. controls (1.3,2.5) .. (1.6,2) ; %through p1
\draw[] (1.5,0) .. controls (1.3,0.5) .. (1.6,1) ; %through p3
\draw[] (1.5,0.7) -- (1.5,2.3); %vertical ruling

\begin{scope}[rotate=90,shift={(0,-3)}]% shift and rotate to get horizontal curves
\draw[] (1.5,3) .. controls (1.3,2.5) .. (1.6,2) ; 
\draw[] (1.5,0) .. controls (1.3,0.5) .. (1.6,1) ; 
\draw[] (1.5,0.7) -- (1.5,2.3); 
\end{scope}

\begin{scope}[shift={(2,0)}]
%right rectangle%%%%%%%%%%%%%%%%
\draw[] (5,0)--(8,0)--(8,3)--(5,3)-- cycle; 
\draw[] (6.5,0) -- (6.5,3); % rulings
\draw[] (5,1.5) -- (8,1.5);
%points right
\draw[fill=black] (6.5,0) circle (0.03cm);
\draw[fill=black] (6.5,3) circle (0.03cm);
\draw[fill=black] (5,1.5) circle (0.03cm);
\draw[fill=black] (8,1.5) circle (0.03cm);
%labeling
\node[above] (p1) at (6.5,3) {$\small{p_1}$};
\node[right] (p2) at (8,1.5) {$\small{p_2}$};
\node[below] (p3) at (6.5,0) {$\small{p_3}$};
\node[left] (p4) at (5,1.5) {$\small{p_4}$};
\end{scope}

\end{tikzpicture}
\caption{The $(1,1,1,1)$- blow-up of $\PP^1\times\PP^1$ in $(p_1,p_2,p_3,p_4)$.}

\label{model2}
\end{figure}
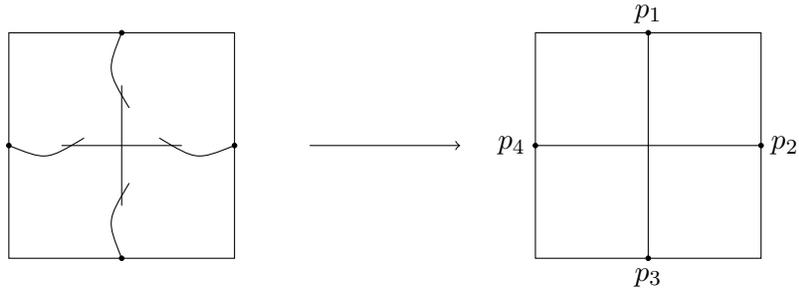

\item[d=2:]For degree $2$, we take  $\PP^2$ together with its toric boundary $(xyz=0)$. We can fix three 
collinear points, one on each boundary divisor, say $p,q,r$.  Let $Y'$ be the $(3,3,2)$-blow-up of $\PP^2$ in $(p,q,r)$. 
 This yields a weak del Pezzo surface of degree $1$, as we have blown up $8$ points that are not on $(-2)$ curves.
 Now, blow down the strict transform of the toric divisor that is a $(-1)$-curve.  The resulting surface $\mathfrak{Y}_2$ is  a weak del Pezzo surface of degree $2$ with anticanonical cycle $D=D_1+D_2$.  It carries an $E_6$ configuration of effective $(-2)$-curves by construction. 
 There are also 2 exceptional curves $E_1,E_2$ of the first kind each meeting a long end of the root system and a component of the anticanonical divisor, see Figure \ref{curveconfig}. 
  The special points of $D_i$ are the points $D_i\cap E_i$ and the two points of the intersection $D_1\cap D_{2}$. 
 
 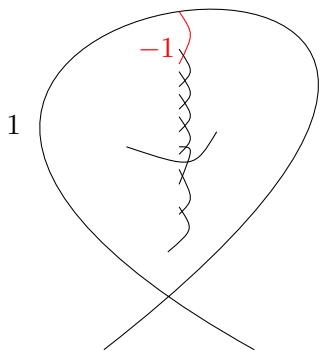
\begin{figure}\centering

\begin{tikzpicture}[scale=1]
\draw[] 
  (-1,0) 
    .. controls (8,7) and (-8,5) .. 
  (1,0);
  
  \draw[red] %e1
  (0,4.5) 
    .. controls (0.2,4.2) .. 
  (0,3.8) ;
   \draw[] 
  (0,4) 
    .. controls (0.2,3.7) .. 
  (0,3.5) ;
\draw[] 
  (0,3.7) 
    .. controls (0.2,3.4) .. 
  (0,3.2) ;
  \draw[] 
  (0,3.4) 
    .. controls (0.2,3.1) .. 
  (0,2.9) ;
  \draw[] 
  (0,3.1) 
    .. controls (0.2,2.8) .. 
  (0,2.6) ;
\draw[] %special
  (0,2.7) 
    .. controls (0.2,2.7) .. 
  (0,2.2) ;

\draw[] %special
  (-0.7,2.7) 
    .. controls (0.2,2.4) .. 
  (0.5,2.9) ;

 \draw[] 
  (0,2.4) 
    .. controls (0.2,2) .. 
  (0,1.8) ;
 \draw[] 
  (0,1.9) 
    .. controls (0.2,1.6) .. 
  (-0.15,1.3) ;
 %\draw[thick, blue] 
 % (0,1.6)     --
 %% (-0.14,0.7);
 % \node[blue](1) at (-0.25, 1.1){$0$};
\node[red](2) at (-0.3, 4){$-1$};
\node[](3) at (-2.2, 3){$1$};
  \end{tikzpicture}
\caption{The surface $\mathfrak{Y}_1$.  }
\label{d1bild}
\end{figure}

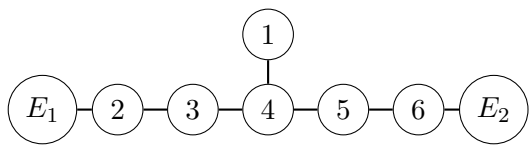
\begin{figure}\centering
\begin{tikzpicture}
\tikzstyle{every node}=[draw,shape=circle];
\node (1) at (0,0){$E_1$}; 
\node (2) at (1,0){$2$};

\node (3) at (2,0){$3$}; 
\node (4) at (3,0){$4$};
\node (5) at (3,1){$1$}; 
\node (6) at (4,0){$5$}; 

\node (7) at (5,0){$6$}; 
\node (8) at (6,0){$E_2$}; 
\draw[thick] (1)--(2)
(2)--(3)
           (3)--(4)
           (4)--(5)
           (4)--(6)
           (6)--(7)
           (7)--(8);
\end{tikzpicture}
\caption{The $E_6$ root system and the exceptional curves $E_1,E_2$. }
\label{curveconfig}
\end{figure}

 \item[d=3:] For degree $3$, we take again   $\PP^2$ together with its toric boundary $(xyz=0)$. As above, we fix three 
collinear points, one on each boundary divisor, say $p,q,r$.  Let $Y$ be the $(2,2,2)$-blow-up of $\PP^2$ in $(p,q,r)$, a weak del Pezzo surface of degree $3$.  It carries a $D_4$ root system of effective $(-2)$-curves, given by the strict transform of the line through the points $p,q,r$ together with the strict transforms of the first blow-ups in each point. 
The anticanonical cycle $D=D_1+D_2+D_3$ is given by the strict transform of the toric boundary. There are also three irreducible $(-1)$-curves  $ E_i$ in the exceptional locus of the blow-up, each meeting a unique component $D_i$. The special points are again the nodes of the anticanonical cycle and the points $E_i\cap D_i$, see Figure \ref{curveconfig2}.
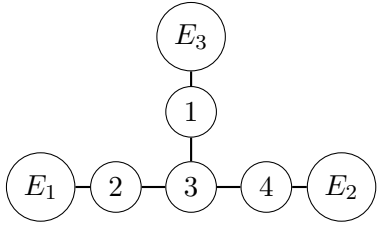
\begin{figure}\centering
\begin{tikzpicture}
\tikzstyle{every node}=[draw,shape=circle];
%\node (1) at (0,0){$E_1$}; 
\node (2) at (1,0){$E_1$};

\node (3) at (2,0){$2$}; 
\node (4) at (3,0){$3$};
\node (5) at (3,1){$1$}; 
\node (6) at (4,0){$4$}; 

\node (7) at (5,0){$E_2$}; 
%\node (8) at (6,0){$E_2$}; 
\node(9) at (3,2){$E_3$};
\draw[thick] 
(2)--(3)
           (3)--(4)
           (4)--(5)
           (4)--(6)
           (6)--(7)
           (5)--(9);
\end{tikzpicture}
\caption{The $D_4$ root system and the exceptional curves $E_1,E_2$, $E_3$. }
\label{curveconfig2}
\end{figure}

 \item[d=4:] Again, let $Q=\PP^1\times\PP^1$ with toric boundary $\tilde{D}=\tilde{D}_1+\tilde{D}_2+\tilde{D}_3+\tilde{D}_4$. Let $p_i$, $i=1\dots4$, be points on the intersection of the fibres of the two ruling with the toric boundary components. Blow up $Q$  once in each $p_i$, the resulting surface $\mathfrak{Y}_4$ is a weak del Pezzo of degree $4$, with an $A_2$  root system of effective $(-2)$-curves and an anticanonical cycle 
${D}={D}_1+{D}_2+{D}_3+{D}_4$ where the ${D}_i$ are the strict transforms of the $\tilde{D}_i$. There are $4$ $(-1)$-curves $E_i$, $i=1\dots 4$ on $\mathfrak{Y}_4$ that are not components of ${D}$, each meeting
exactly one of the $D_i$ transversally.  The special points of ${D}$ are the points ${D}_i\cap E_i$ and ${D}_i\cap {D}_{i+1}$, where the indices are considered cyclically, see Figure \ref{model2}.
\item[d=5:] Pick $4$ points $\{p,q,r,s\}$ in $\PP^2$, no three on a line. Then $\text{Bl}_{\{p,q,r,s\}}(\PP^2)$ is the del Pezzo surface of degree $5$. The anticanonical divisor $D=D_1+D_2+D_3+D_4+D_5$ is a cycle of 
$(-1)$-curves forming a pentagon, which we obtain as follows: we choose the strict transforms of the lines spanned by $(p,q)$, $(p,r)$ and $(q,s)$, together with the exceptional lines which arise form blowing up $p $ and $q$.
Note that permutations of the points $p,q,r,s$ gives rise to isomorphic pairs $(Y,D)$.
There are $5$ more $(- 1)$-curves $E_i$, $i=1\dots 5$. Each $E_i$ meets exactly one of the $D_i$ transversally. 
The special points of $D_i$ are the points $D_i\cap E_i$ and $D_i\cap D_{i+1}$.
\item[d=6:] Pick the three coordinate points $\{p,q,r\}$ in $\PP^2$ which are torus orbits. Then $\mathfrak{Y}_6=\text{Bl}_{\{p,q,r\}}(\PP^2)$ is the del Pezzo surface of degree $6$. Let $D=D_1+D_2+D_3+D_4+D_5+D_6$ be the anticanonical divisor 
which consists of the strict transforms of the coordinate lines and the exceptional lines.
This is a cycle of $(-1)$-curves forming a hexagon. 
The toric structure of $\mathfrak{Y}_6$ indentifies a copy of $\GG_m\subset D_i$ for each $i$.  The special points of $D_i$ are the points $-1\in \GG_m$, and $D_i\cap D_{i+1}$.
 \end{itemize} 

\begin{definition}\label{def:wdPd}
We will denote the (weak) del Pezzo surfaces of degree $d$ constructed here by $\mathfrak{Y}_d$. 
\end{definition}

Now, the triangulation $\mathscr{G}$ defines a locally trivial deformation class $[Y]_\mathscr{G}$ of $d$-semistable K3 surfaces and we can take the member $Y_\mathscr{G}$ of $[Y]_\mathscr{G} $ such that the (normalisation) of each component is  isomorphic  (as an anticanonical pair)  to a surface from the above list, by \cite[Lemma 5.14]{Laz08}, where we recall that the valency of a vertex is equal to the degree of the corresponding (weak) del Pezzo surface.
The gluing is such that the special points are identified pairwise where, in particular, interior special points are identified with interior special points.
\end{construction}

 We learned the next result from  \cite{theta15}:
 
 \begin{proposition}\label{optimalglue} 
 Let $\mathscr{G}$ be a triangulation with valency at most $6$ and no $k$-bands of hexagons. 
Then the surfaces $Y_\mathscr{G}$ 
are maximal, i.e. 
have trivial Carlson map $c_{Y_\mathscr{G}}=1$.
\end{proposition}
\begin{proof}A proof can be found in \cite[Construction 10.14]{GHKS}. Here we sketch a  proof in the case of simple normal crossing.
 Each of the surfaces $\mathfrak{Y}_i$, for $i=1\dots 5$, 
has a $\QQ$-basis $B_i$ of $\Pic(\mathfrak{Y}_i)$ given by the $(-2)$-curves and the interior $(-1)$-curves, i.e. those that are not components of the 
anticanonical divisor. Let $Y_\mathscr{G}=\cup_{j \in J} Y_j$. We shall choose an order on $J$. 
 The Picard 
group of $Y_\mathscr{G}$ is given by the kernel of the Carlson map. Let $Y_{\operatorname{simp}}$ be a semi-simplicial resolution of $Y_\mathscr{G}$, see \cite{Carl},\cite[\S 4.2.2]{AGIII}.  Then 
one defines $Y_p$ to be 
\[Y_p=\coprod Y_{j_0}\cap\dots\cap Y_{j_p}\qquad (j_0<\dots < j_p) \] for $p=0,1,2$. Face maps are given by maps $\delta_p\colon Y_p\to Y_{p-1}$ 
such that $\delta_p$ is the inclusion on the components of $Y_p$, see \cite[\S 4.2.2]{AGIII}. We start with a collection of line bundles $L_j, j\in J $ on the components $Y_j$ 
such that the degrees on the double curves $D_{ij}$ coincide. We want to show that these glue to a line bundle $L$ on  $Y_\mathscr{G}$.
For $i(j) \leq5$ we write $L_j$ in the basis $B_{i(j)}$.
Also, given a line bundle $L_j$ on a hexagonal component $\mathfrak{Y}_6$, we can find a linearly equivalent divisor on $Y_j$  whose restriction to a component $D_k$ of the anticanonical divisor of $\mathfrak Y_6$ 
is $(\deg L_{j|D_k}) p_k$ where 
$p_k$ denotes the interior special point on $D_k$. 
This follows from  (the proof of) \cite[Lemma 2.8]{GHK15}. 
The line bundles $L_j$ can therefore be represented by divisors whose restrictions to the double curves coincide.
It follows that the collection $L=\{L_j, j \in J\}$ is in the kernel of the Carlson map as follows from \cite[p. 277 ff]{Carl} applied to $Y_{\operatorname{simp}}$.
This implies the result.
\end{proof}

We also obtain that maximal $d$-semistable $K3$ surfaces  in  $(-1)$-form are always projective. 

\begin{proposition}\label{max:proj}Let $Y=\cup Y_i$ be a  $d$-semistable $K3$ surface of type III with $c_{Y}=1$. If $Y$ is in $(-1)$-form, then $Y$ is projective.
\end{proposition}
\begin{proof} Let $Y$ be a surface as in the proposition. By \cite[Proposition 5.2]{Laz08}, the normalisations $Y^\nu_i$ of the components  $Y_i$ of $Y$ are weak del Pezzo surfaces of degree $d_i\leq 6$ and if $D=\sum D_i$ is the anticanonical cycle on  $Y^\nu_i$, the orthogonal complement of the sublattice generated by the $D_i$  in $\Pic(Y_i)$ is a root lattice  $E_8$, $E_6$, $D_4$ or $A_2$ for $d_i=1,2,3$ and $4$, respectively and empty for  $d_i=5,6$ (see our above discussion).  
These classes are effective by construction.  We show that $Y$ is projective by giving an ample bundle $A_i$  on each $Y^\nu_i$ such that $A_i.D_{ij}=\ A_i.D_{ik}$ for all $j,k$, i.e. such 
that the $A_i$ have the same degree on all components of the anticanonical cycle on  $Y^\nu_i$. By Lemma \ref{degree}, after taking suitable multiples, one can then glue these line bundles to obtain a bundle $A$ on $Y$. The bundle $A$ is ample as its restriction to each irreducible component is ample, by \cite[Proposition 1.2.16]{Lazar}.  
As the structure morphism is proper, Y  is projective. 
To simplify notation, we taciturnly assume $Y_i^\nu=Y_i$. 

For $d_i=1$ there is nothing to show. For $d_i=5,6$ the claim is obvious, we can take the anticanonical divisor.
We consider the cases $d_i=2,3$.  
For the root lattices $R_i$ described above we have a natural root basis given by the set $B(Y_i)$ of  $(-2)$-curves constructed in the blow-up procedure. 
Note also that for all such $Y_i$ by construction we have $(-1)$-curves  as in Figures \ref{curveconfig} and \ref{curveconfig2} connecting the root system to the boundary. We denote the set of these curves by $E(Y_i)$.
  
Suppose there is an integer $e>0$  and a divisor \[A_i=\sum_{B_j\in B(Y_i)} b_jB_j+\sum_{E_k\in E(Y_i)} eE_k\] in $\Pic(Y_i)$  such that $A_i.C>0$ for all $C\in B(Y_i)\cup E(Y_i)$.
  Then $A_i$ defines an ample bundle on $Y_i$: 
  being a (weak) del Pezzo surface, $Y_i$ is a Mori Dream space by \cite[Theorem 2.9]{TVAV}, so in particular,  the cone of curves of $Y_i$ is rational polyhedral. The Picard rank is greater than $3$ and thus, by \cite[Proposition 1]{ArLa}, $\Pic(Y_i)$ generated by curves $C$ with $C^2<0$, i.e. by $(-2)$ and $(-1)$-curves.  
  Let $C$ be such a curve. If $C\in B(Y_i)\cup E(Y_i)$, then $A_i.C>0$ by assumption. If $C$ is not in $B(Y_i)\cup E(Y_i)$,   $C$ is a component of $D$,
   as follows from Proposition \ref{pro:canflops} below (which is independent of this result). 
 Then $C.D=e$.  Hence $A_i$ is strictly positive on the (polyhedral)  cone of curves, and hence ample by Kleiman's criterion.  We define such $A_i$'s as follows:  for $E_6$,  take the numbers $(7,11,13,16,13,11)$ for the roots, where the number in the $i$-th position is the coefficient of the $i$-th root, with indexing as in the figures, and set $e=10$. For $D_4$ we can take $(11,11,13,11)$ and $e=10$. 
 
It remains to consider the case $d_i=4$. In this case $Y_4=\mathfrak{Y}_4$ with the $A_2$ root system and the same construction as before with numbers $(3,3)$ on the root system and $e=2$ defines an ample bundle $A_i$. 
\end{proof}

We  have the following proposition. 

\begin{proposition}
Let $\mathscr{G}$ be a triangulation of $S^2$  with $2d$ triangles with valency at most $6$ and without $k$-bands of hexagons for any $k>1$.  

The surface $Y_\mathscr{G}$ is a maximal primitive $d$-semistable K3 surface. 
The associated maximal smoothing $\shY_\mathscr{G}\to S$ is a  model in $(-1)$-form of the Dolgachev-Nikulin-Voisin family of degree $2d$, i.e.
 $\Pic(\shY_\mathscr{G}{_\eta})=\check{M}_{2d}$ and $\Pic(\shY_\mathscr{G}/S)\cong\Pic(Y_\mathscr{G})$ by restriction. 
The threefold $\shY_\mathscr{G}$ is Calabi-Yau with $H^1(\shY_\mathscr{G},\O_{\shY_\mathscr{G}})=0$ and projective over the base $S$. 
\end{proposition}
\begin{proof} $Y_\mathscr{G}$ is a  $d$-semistable K3 surface in $(-1)$-form by  construction. It is primitive because the dual graph does not have $k$-bands of hexagons. It is maximal, i.e. $c_{Y_\mathscr{G}}=1$, due to our
choice of the gluing. It is also projective by Proposition \ref{max:proj}. Hence there is a deformation with the claimed properties 
by Proposition \ref{lem:unoquedeg}. Letting $f\colon\shY_\mathscr{G}\to 
S$ denote the structure morphism,  we prove the vanishing of  
$H^1(\shY_{\mathscr{G}},\O_{\shY_\mathscr{G}})=0$ as follows: 
We have $H^1(Y_{\mathscr{G}},\O_{Y_\mathscr{G}})=0$  
by \cite[Lemma 5.7]
{Frie83} so by semi-continuity it follows from a result of Grauert 
\cite[Cor 12.9]{Har77}  that $R^1f_*\O_{\shY_\mathscr{G}}=0$. As we work over an affine base, $H^1(\shY_{\mathscr{G}},\O_{\shY_\mathscr{G}})=0$ follows 
from  \cite[Proposition 8.5]{Har77}.
\end{proof}

We give examples for low degree. 

 \begin{example}[\cite{thesis}]\label{DNV2}
  Specialising to genus $2$, there are precisely two different triangulations of $S^2$ with two triangles, corresponding to dual intersection 
 complexes of degenerations with central fibres having three components. By e.g. \cite{ES17}, these  are given by two triangles glued 
along the boundary and two triangles glued along one side to each other, with the 
remaining sides identified. We shall denote the first of these triangulations by $\mathscr{P}$  and the latter by $\mathscr{T}$,  see Figure \ref{ZETA}.
\begin{figure}\centering
\begin{tikzpicture}[scale=0.75]
\draw[]    (0,1) -- ++(0:1.5) node(C1)[below]{f}--++(0:1.5)  %first triangle
-- ++(120:1.5) node(A1){}--++ (120:1.5) node(top1){} --++ (240:1.5)node(B1)[left]{e} --cycle; 
\draw[] (top1.center)--++(0:1.5) node(C2)[above]{e}--++(0:1.5) --++(-120:1.5) node[right]{f}--++(-120:1.5);
\draw node at (1.75,0) {$\mathscr{P}$};

\draw[]    (6,1) -- ++(0:1.5) node(C3)[below]{f}--++(0:1.5)  %first triangle
-- ++(120:1.5) node(A3){}--++ (120:1.5) node(top2){} --++ (240:1.5)node(B3)[left]{f} --cycle; 
\draw[] (top2.center)--++(0:1.5) node(C4)[above]{e}--++(0:1.5) --++(-120:1.5) node[right]{e}--++(-120:1.5);
\draw node at (7.75,0) {$\mathscr{T}$};
\end{tikzpicture}\caption{The triangulations $\mathscr{P}$ and $\mathscr{T}$.}
\label{ZETA}
\end{figure}
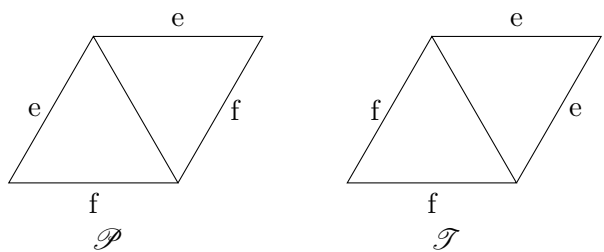
 Let $\YP$ be the surface obtained by gluing three copies of $\mathfrak{Y}_2$ by identifying  components of the boundary cycle such that interior special points  are identified. 
For the triangulation $\mathscr{T}$, the associated maximal surface is a copy of $\mathfrak{Y}_4$, with two opposite components, say ${D}_1$ and ${D}_3$, of the anticanonical curve ${D}$ identified, and two copies of $\mathfrak{Y}_1$ glued to the (images) of ${D}_2$ and  $D_4$, such that the images of the special points match. 
 \end{example}

 \begin{example}There are $4$ combinatorially distinct triangulations of $S^2$ with $4$ triangles, see \cite[Example 2.5]{ES17} and Figure \ref{triangulation4}. We collect the valencies of the vertices in a tuple. The tuples are $(3,3,3,3), (3,2,1,6), (5,5,1,1)$ and $(4,4,2,2)$. None of these triangulations have $k$-bands of hexagons for $k>1$, because the least number of triangles in such a triangulation is $8$.
To take an example let $\mathscr{G}_1$ denote the triangulation such that every vertex has valency $3$.  The resulting surface $Y_{\mathscr{G}_1}$ are three copies of $\mathfrak{Y}_3$ glued according to the triangulation. Here, this means that each component meets each of the remaining components in a curve $C\cong\PP^1$ in such a way that the special points are identified.

\begin{figure}\centering
\begin{tikzpicture}[scale=0.75]
\draw[]    (0.75,1) -- ++(0:1.5) node(C1)[below]{e}--++(0:1.5)  %first triangle
-- ++(120:1.5) node(A1){}--++ (120:1.5) node(top1){} --++ (240:1.5)node(B1)[left]{g} --cycle; 
\draw[] (top1.center)--++(0:1.5) node(C2)[above]{}--++(0:1.5)node(top2){} --++(-120:1.5) node[right]{}--++(-120:1.5);
\draw[] (top2.center)--++(120:1.5)node[right]{f}--++(120:1.5)--++(240:1.5)node[left]{g}--++(240:1.5);
\draw[] (top2.center)--++(-60:1.5)node[right]{f}--++(-60:1.5)--++(180:1.5)node[below]{e}--++(180:1.5);
\node at (3.75,0){$\mathscr{G}_1$};

\draw[]    (8.75,1) -- ++(0:1.5) node[below]{e}--++(0:1.5)  %first triangle
-- ++(120:1.5) node{}--++ (120:1.5) node(2top1){} --++ (240:1.5)node[left]{e} --cycle; 
\draw[] (2top1.center)--++(0:1.5) node[above]{}--++(0:1.5)node(2top2){} --++(-120:1.5) node[right]{}--++(-120:1.5);
\draw[] (2top2.center)--++(120:1.5)node[right]{f}--++(120:1.5)--++(240:1.5)node[left]{g}--++(240:1.5);
\draw[] (2top2.center)--++(-60:1.5)node[right]{f}--++(-60:1.5)--++(180:1.5)node[below]{g}--++(180:1.5);
\node at (11.75,0){$\mathscr{G}_2$};

\draw[]    (0,-4) -- ++(0:1.5) node[below]{e}--++(0:1.5)  %first triangle
-- ++(120:1.5) node{}--++ (120:1.5) node(3top1){} --++ (240:1.5)node[left]{e} --cycle; 
\draw[] (3top1.center)--++(0:1.5) node[above]{f}--++(0:1.5)node(3top2){} --++(-120:1.5) node[right]{}--++(-120:1.5);
\draw[] (3top2.center)--++(0:1.5)node[above]{g}--++(0:1.5)--++(-120:1.5)node[right]{g}--++(-120:1.5);
\draw[] (3top2.center)--++(-60:1.5)node[right]{}--++(-60:1.5)--++(180:1.5)node[below]{f}--++(180:1.5);
\node at (3.75,-5){$\mathscr{G}_3$};
\draw[]    (8,-4) -- ++(0:1.5) node[below]{g}--++(0:1.5)  %first triangle
-- ++(120:1.5) node{}--++ (120:1.5) node(3top1){} --++ (240:1.5)node[left]{e} --cycle; 
\draw[] (3top1.center)--++(0:1.5) node[above]{e}--++(0:1.5)node(3top2){} --++(-120:1.5) node[right]{}--++(-120:1.5);
\draw[] (3top2.center)--++(0:1.5)node[above]{g}--++(0:1.5)--++(-120:1.5)node[right]{f}--++(-120:1.5);
\draw[] (3top2.center)--++(-60:1.5)node[right]{}--++(-60:1.5)--++(180:1.5)node[below]{f}--++(180:1.5);
\node at (11.75,-5){$\mathscr{G}_4$};

\end{tikzpicture}\caption{The triangulations with $4$ triangles.}
\label{triangulation4}
\end{figure}
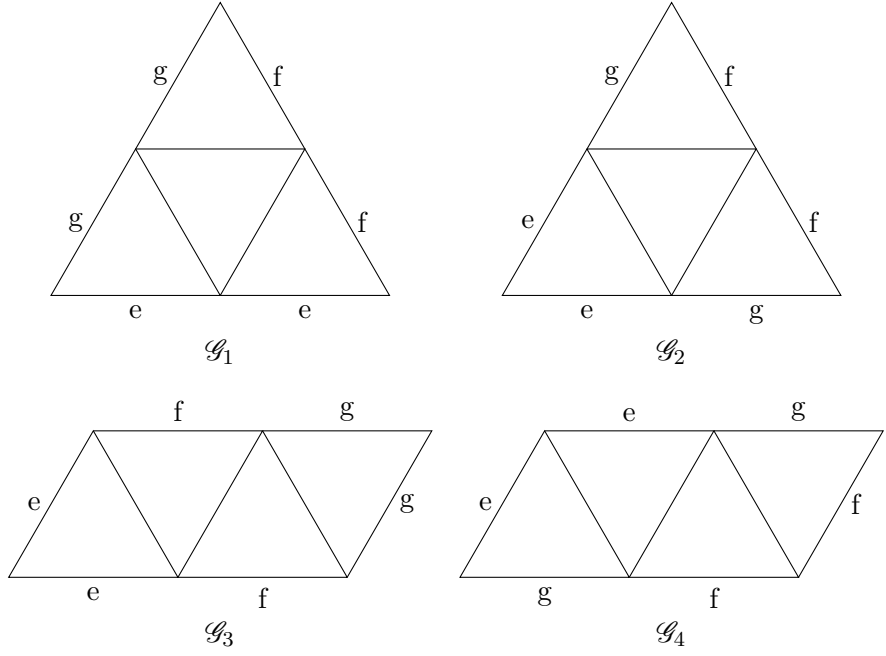
 \end{example}

\section{Relative Notions and the Mori fan}\label{sec:mmpmorifan}

\subsection{Relative Notions}
As we will make extensive use of various concepts of the minimal model program we will recall the relevant basic concepts here, following \cite{Kaw97}. Let $\pi\colon\shX\to U$ be a projective morphism of normal schemes. An $\RR$-Cartier divisor $L$ on $\shX$ is an $\RR$-linear combination of Cartier divisors.  A $\QQ$-Cartier divisor $L$ on $\shX$ is a $\QQ$-linear combination of Cartier divisors. Similarly, an $\RR$-Weil divisor is an $\RR$-linear combination of prime divisors. Two $\RR$-Cartier divisors $L,L'$ are linearly equivalent over $U$ if their difference is an $\RR$-linear combination of principal divisors and an $\RR$-Cartier divisor pulled back from $U$. Two $\RR$-Cartier diviors $L,L'$ are numerically equivalent over $U$, denoted by $L\equiv_U L'$, if $L.C= L'.C$ for all curves -integral $1$-dimensional closed subschemes -  $C$  in fibres of $\pi$.

We define
\[\Pic(\shX/U)_\RR=\{ \RR\text{-Cartier divisors on } X \}/\text{linear  equivalence over } U\]
and 

 \[\N^1(\shX/U)=\{ \RR\text{-Cartier divisors on } X \}/\text{numerical equivalence over } U.\]  The latter is a finite dimensional vector space 
\cite[Proposition IV.4.3]{kl66}.

In this context, we call an element of $\N^1(\shX/U)$ a $\ZZ$-divisor if we want to emphasise that it is  the class of a Cartier divisor.  
\begin{remark}
For a type III degeneration $\shY\to S$ of K$3$ surfaces, we have 
$\Pic(\shY/S)_\RR\cong\N^1(\shY/S)$:
indeed, by maximality, it is enough to show this on the central fibre $Y$, where it follows from the fact that any divisor on the central fibre is given by an element of the lattice $L$ defined in Sequence \ref{equ:L} and all (normalised) components are smooth rational surfaces.
\end{remark}

We say that a Cartier divisor $L$ is \emph{nef over $U$}, or \emph{$\pi$-nef}, if $L.C\geq 0$ for all curves $C$ in $\shX$ mapping to a point. A Cartier divisor $L$  is \emph{ample over $U$}, or \emph{$\pi$-ample}, if the restriction $L_u$ is ample on the fibre $\shX_u$ for all $u\in U$.
 A Cartier divisor $L$ is \emph{semiample over U }, or \emph{semiample},
  if there is a projective morphism $g\colon\shZ\to U$,  a $U$-morphism  $\shX\to\shZ$ and a $g$-ample divisor $A$ on $\shZ$ such that 
 $L^{\otimes n}=g^*A$ for some integer $n>0$. 
 A Cartier divisor $L$ is \emph{$\pi$-movable} (or \emph{$\pi$-moving}) if $\dim \operatorname{Supp}\operatorname{Coker}(\pi^*\pi_*\O_X(L)\to\O_X(L))\geq 2$.

\subsubsection{Cones}

The \emph{$\pi$-nef cone}  $\Nef(\shX/U)$ is the closure of the $\pi$-ample cone in $\N^1(\shX/U)$, the \emph{closed $\pi$-movable cone} $\bar{\shM}(\shX/U)$ is the closure of the  convex cone generated by $\pi$-movable divisors. A Cartier divisor $L$ is \emph{ $\pi$-effective }if $\pi_*\O(L)\neq 0$. Let $\shB^e(\shX/U)$ be the convex cone generated by all $\pi$-effective divisors, and let $\overline{\shB^e}(\shX/U)$ denote its closure, the cone of \emph{pseudo-effective} divisors. Set \[\Mov (\shX/U)=\bar{\shM}(\shX/U)\cap\shB^e(\shX/U).\] We will refer to $\Mov(\shX/U)$ as the \emph{moving cone}.

\subsection{The Mori fan}
The Mori fan was first introduced in \cite{HK}.  Here we shall recall the definition in the situation which we consider, namely for models of the  DNV family.

First however, we recall the pullback under dominant rational maps in general. For this we fix a dominant rational map $f\colon \shX\dashrightarrow \shZ$ of projective $\QQ$-factorial schemes over $U$. Let $L$ be an effective  Cartier divisor on $\shZ$. We define  $f^*(L)$ to be the unique Weil divisor that is equal to the pullback of $L$ by $f$  on the open set of codimension $2$ where $f$ is a morphism.  The divisor $f^*(L)$ is $\QQ$-Cartier 
by $\QQ$-factorality of $\shX$.  In this way we obtain a linear map $f^*: \N^1(\shZ/U)\to \N^1(\shX/U)$.

Now we specialize to the case where $\shY\to S$ is a model of the DNV family, of degree $2d$, see also \cite{theta15}.  
Let $f\colon \shY\dashrightarrow  \shY'$ be a \emph{small} modification over $S$, i.e. a birational map which is an isomorphism in codimension $1$ such that $\shY'$ is projective over $S$. 
In this situation, we will also write $(\shY',f)$ for $f\colon \shY\dashrightarrow  \shY'$ and call it a \emph{marked minimal model} of the DNV family (where we  have fixed $\shY \to S$ as a reference model).

Define the cone
$$
C(f):=f^*\Nef(\shY'/S)\subset \N^1(\shY/S).
$$
\begin{definition}\label{def:morifan}
Let $\shY\to S$ be  a model of the DNV family of degree $2d$. The set of all  cones $C(f)$  and their faces, with $f$ a small modification, is the \emph{Mori fan} of $
\shY\to S$, denoted by $\Morifan(\shY/S)$.
\end{definition}

\begin{remark}
If $(\shY',f)$ and $(\shY'',g)$ are two marked minimal models and $C(f)\cap C(g)$ has codimension $0$, then by \cite[Lemma 1.5]{Kaw97} there is an isomorphism $\beta\colon\shY''\to \shY'$ with $f=\beta \circ g$,
and hence $C(f)=C(g)$.
\end{remark}

Note that if $\shY'\to S$ is another model, $\Morifan(\shY'/S)$ is canonically identified with $\Morifan(\shY/S)$ via the isomorphism $\N^1(\shY/S)\cong\N^1(\shY'/S)$ induced by taking strict transforms. In particular,
the Mori fan only depends on the DNV family, not on a specific model.

We recall the following result of \cite{theta15}, which implies that $\Morifan(\shY/S)$ is indeed a fan, and, in particular, is closed under intersection and taking faces.  If $\Sigma$ is a fan in a vector space $V$, then we denote its support by $|\Sigma|$.
\begin{theorem}[\cite{theta15}, Thm 6.5]\label{ThmGHKS}
The following holds:
\begin{itemize}
\item[(i)] Let $\triangle\subset \Nef(\shY_\eta)$ be a rational polyhedral cone. Let $r_{\eta}$ be the restriction map $\Pic(\shY/S)\to \Pic(\shY_\eta)$. Then $r^{-1}_\eta(\triangle)$ and $r^{-1}_\eta(\triangle)\cap \Mov(\shY/S)$ are rational polyhedral cones and  
\[ \{r_\eta^{-1}(\triangle)\cap \gamma | \gamma \in \Morifan(\shY/S)\}\] is a finite set of rational polyhedral cones, with support $r^{-1}_\eta(\triangle)\cap \Morifan(\shY/S)$. 
\item[(ii)] The support of $\Morifan(\shY/S)$ is $\Mov(\shY/S)$. 
\end{itemize}
\end{theorem}
\begin{proof}We sketch the approach of \cite{GHKS}. 
The crucial point is that one can construct models of the DNV family not only over $S$, but over the spectrum of the local ring of a smooth algebraic curve, namely the 
local ring of a cusp of the compactified Dolgachev mirror space which, in this case, is a modular curve.  For these models one can run MMP, see \cite[\S 3.6]{KoMo}.
Then, invoking \cite{Sho96}, one concludes by  \cite[Theorems 3 and 4]{KAW2}.
\end{proof}

\begin{remark}\label{rem:factorizationflops}
The fact that one can run MMP for $\shY \to S$
implies that a small modification  $f \colon \shY\dashrightarrow \shY'$ factors into flops. 
Hence $\shY'\to S$ is also a model of the DNV family, which justifies the use of the name marked model.
\end{remark}

\begin{remark}\label{rem:abundance}
The proof we have just sketched also shows that log-abundance holds for models $\shY \to S$ of the DNV family. In particular nef line bundles are semi-ample.
\end{remark}

\begin{remark}\label{finiteMF}
The nef cone of $\shY_\eta$ is finitely polyhedral if and only if $d=1$. 
To see this we argue by means of the analytic category. Using the sequence (\ref{equ:basicsequ}) and its analytic analogue
we can identify the nef cone of  $\shY_\eta$ with the nef cone of a very general fibre of an analytic Kulikov model $\mathcal X \to \DD$. 
The claim then follows  from Nikulin's classification in \cite{niktable}, 
as the Picard lattice of $\shY_\eta$ is $U\oplus 2E_8(-1)\oplus \langle -2d \rangle$. See \cite[Theorem 1]{niktable2} for a more explicit statement of Nikulin's classification.
Hence the  Mori fan of a model  $\shY\to S$ of degree $2$ is a finite collection of rational polyhedral cones. Note that this is special to the case $d=1$.
\end{remark}

\subsubsection{Interior facets}

 It is well known that the codimension $1$ faces separating  maximal  (dimensional) cones of the Mori fan correspond to flopping contractions. We recall the details here. Let $\shY, \shY'$ be  models of the DNV family of degree $2d$.
 Let  $\phi\colon\shY\dashto \shY'$  be a flop. Then there exists a divisor $F$ on $\shY$ and morhpisms $\psi,\psi'$  with $\psi$ a contraction of an  $F$-negative extremal ray 
such that the diagram 
\[
\xymatrix{ \shY\ar@{-->}[rr]^\phi\ar[dr]_\psi&&\shY'\ar[dl]^{\psi'}\\
&{\shZ}&
}
\] 
commutes.   Write $\phi_*F=F'$ for the birational transform. By the contraction theorem, we have
\[\N^1(\shY/S)=\psi^*\N^1({\shZ}/S)\oplus \RR[F]
\] and 
\[\N^1(\shY'/S)={(\psi')}^*\N^1({\shZ}/S)\oplus \RR[F'].
\] 
Using the isomorphism $\N^1(\shY/S)\cong \N^1(\shY'/S)$ given by taking strict transforms  under $\phi$, 
we can identify $\psi^* \N^1(\shZ/S)={(\psi')}^*\N^1(\shZ/S)$. This cone spans a hyperplane in $\N^1(\shY/S)$ 
and $\RR[F]=\RR[F']$. Hence the cones   $\Nef(\shY / S)$, which we can think of as $C(\operatorname{id}_{\shY/S})$, 
and $C(\phi)$  in the Mori fan meet in codimension $1$,  compare \cite[Proposition 12.2.2]{MatBook}.

Conversely, we now assume that
 $\sigma, \sigma'$ are two  maximal cones of $\Morifan(\shY/S)$ with $\sigma\cap\sigma'=\tau$ a codimension $1$ face. Without loss of generality we can assume $\sigma=\Nef(\shY)$.
 By the definition of the Mori fan there is a small modification  $f\colon\shY\dashrightarrow \shY'$ of projective minimal models with $f^*(\Nef(\shY'/S))=\sigma'$.  Let $\psi\colon\shY \to\shZ$ be the contraction defined by the linear system of a suitable multiple of a very general effective divisor $L$ in $\tau$. Note that this is indeed a morphism as $L$ is semiample by abundance, see Remark \ref{rem:abundance}.
Then $\tau= \psi^*(\Nef(\shZ))$.
 The facet $\tau$ defines an extremal ray $R$  in $\N^1(\shY/S)$ that is cut out by $L$, i.e. $\psi=\contr_R$. 
 Also, any divisor in the interior of $\sigma'$ is negative on $R$. Let $F \in \Morifan(\shY/S)$ be such a divisor. Then 
 $\psi$ is an $F$-flopping contraction.  Thus the  diagram

\[
\xymatrix{ \shY\ar@{-->}[rr]^f\ar[dr]_\psi&&\shY'\ar[dl]^{\psi'}\\
&{\shZ}&
}
\] 
exhibits $f$ as an $F$-flop.

 \begin{definition}\label{def:intcones}
We say that a codimension $1$ cone $\tau$ of $\Morifan(\shY/S)$ is an \emph{interior facet} if there are two marked  models $(\shY',f) ,(\shY'',g)$  of the DNV family such that  $C(f)\cap C(g)=\tau$.
\end{definition}

The above discussion can now be summarized as follows:
\begin{proposition}\label{prop:intconesflops}
Every facet of $\Nef(\shY/S)\in \Morifan(\shY/S)$  that is an interior facet of $\Morifan(\shY/S)$ defines a  flop $\shY\dashrightarrow \shY'$ and conversely any such flop is defined by an interior facet.
\end{proposition}

\subsubsection{Action of $\Bir(\shY/S)$}

The birational $S$-automorphisms $\Bir(\shY/S)$ do not contract divisors, as $\shY\to S$ is a minimal model. Hence  we can define a representation 
\[ \sigma\colon\Bir(\shY/S)\to \GL(\N^1(\shY/S)) \] 
by $\sigma(\theta)(D)=\theta_*(D)$ for $\theta \in \Bir(\shY/S)$.  
Since $\theta$ is a small modification, this defines a permutation action on the set of  maximal cones of $\Morifan(\shY/S)$. The following result,
which will be crucial for us, follows directly as in \cite[Lemma 1.5]{Kaw97}. 

\begin{proposition} \label{prop:actionbir}
Let $f_i\colon \shY\dashrightarrow \shY_i$,  $i=1,2$ be small modifications of $\shY$ over $S$, with $\shY_1$ and  $\shY_2$  models of the DNV family (of a given degree). Suppose that the associated Mori  cones are in the same orbit under the action of $\Bir(\shY/S)^{\operatorname{op}}$, i.e. that there 
is a a birational automorphism $\theta\in \Bir(\shY/S)$ such that  
\[ C(f_1)=\theta^* C(f_2).\] 
Then  there is an isomorphism $\beta:\shY_1\to \shY_2$ such that $\beta \circ f_1=f_2$.
\end{proposition}

\section{$(-1)$-Curves and Curve Structures}\label{sec:curvestructures}
 
From this section on, $Y_c$ will denote a $d$-semistable $K3$ surfaces of type III  while $Y$ will denote a component of $Y_c$. 
For any degree $2d$, 
let $\DNV_{2d}$ be the set of $d$-semistable $K3$ surfaces of type III in $(-1)$-form  with $t=2d$ triple points. For $d=1$, these are the surfaces $\YP$ and $\YT$ described in Example \ref{DNV2}. Let $\Mod_{2d}$ denote the set of surfaces $Y_c$ such that there is  $Y_0\in \DNV_{2d}$ 
and a sequence  of elementary modifications of type I  \[ Y_0\dashrightarrow Y_1 \dashrightarrow \dots\dashrightarrow Y_i \dashrightarrow \dots \dashrightarrow Y_n=Y_c.\]
This also defines an anticanonical divisor on each component of  $Y_c$.
We denote by $\PMod_{2d}$ the subset of projective surfaces in $\Mod_{2d}$. Each of these surfaces determines a model of the Dolgachev family of degree $2d$. In general, one does not obtain the full set of models in this way 
as, usually, a central fibre of a model of the DNV family of degree $2d$ cannot be obtained by type I modifications alone. However, we will later see that for $d=1$ the set $\PMod_2$ parametrizes the isomorphism classes of 
 models of the Dolgachev-Nikulin-Voisin family of degree $2$.
 
  \begin{definition} Let $Y_c\in \Mod_2$. We say $Y_c$  is of  \emph{class} $\mathscr{T}$ (resp. $\mathscr{P}$) if the dual intersection graph of $Y_c$ is given by  $\mathscr{T}$ (resp. $\mathscr{P}$).
\end{definition}
We denote the set of $Y_c\in \Mod_2$ 
of class $\mathscr{T}$ (resp. $\mathscr{P}$) by  $\Mod_2(\mathscr{T})$ (resp. $\Mod_2(\mathscr{P})$), and similarly we define $\PMod_2(\mathscr{T})$ and $\PMod_2(\mathscr{P})$. 
 Note that this defines decompositions into  disjoint unions  $\Mod_2 = \Mod_2(\mathscr{T}) \sqcup \Mod_2(\mathscr{P})$ and $\PMod_2 = \PMod_2(\mathscr{T}) \sqcup \PMod_2(\mathscr{P})$ respectively.

We recall the following elementary facts about anticanonical pairs, i.e. pairs $(Y,D)$ with $Y$ a smooth rational surface and $D$ an effective anticanonical cycle. We will use these without further mention. 
\begin{proposition}
Let $(Y,D)$ be an anticanonical pair. Let $C$ be an irreducible curve that is not a component of $D$.
\begin{itemize}
\item[(i)]{If $C^2=-1$,  then $C.D$=1 and $C$ is smooth rational.}
\item[(ii)]{ If $C^2=0$,  then $C.D=2$ or $C.D=0$. In the first case, $C$ is smooth rational, in the second case $p_a(C)=1$. }
\item[(iii)]{If $C^2=-2$, then $C$ is smooth rational and $C.D=0$.}
\end{itemize}
\end{proposition}

\subsection{$(-1)$-Curves}
 In order to control flops, we need to control exceptional curves of the first kind on the components of the central fibres. We show that the set of possible elementary modifications of a surface in $\Mod_{2}$ is quite small. 
 Recall the following terminology.  A \emph{cycle} is a graph whose vertices and edges can be ordered $C_1,\dots C_n$ and $e_1,\dots, e_n$ such that $e_i$ connects $C_i$ and $C_{i+1}$  (where the indices have to be read cyclically). A \emph{tree} is a connected graph not containing a subgraph that is a  cycle.  A vertex $v$  of a graph is called a \emph{fork} if there are at least three edges from $v$.  
 If $G$ is a tree
with a unique fork $v\in G$, the connected components of $G\backslash\{v\}$ are the \emph{branches} of $v$.

 \begin{construction} Let $\Gamma_1^{\mathbf{0}}$ be a tree with a unique fork $v$ such that there are $3$ branches $B_i$, $i=1,2,3$, consisting 
 of $1$, $2$ and $4$ vertices, i.e. the graph underlying the $E_8$ diagram. Let $\mathbf{n}=n_1 \geq 0$. Let $\Gamma_1^{\mathbf{n}}$ be the tree with a unique 
 fork  containing $\Gamma_1^{\mathbf{0}}$   such that $\Gamma_1^{\mathbf{n}}\backslash \{B_1\cup B_2\}$ has a unique connected component given by a 
 chain of length $n_1+4$.
 We label the end of this tree which is not is not a fork of  $\Gamma_1^{\mathbf{n}}$  with $-1$, all other  vertices are labelled with $-2$, see Figure \ref{graph1}.
 (We shall later interpret these labels as intersection numbers.) 
 
 Let $\Gamma_2^{\mathbf{0}}$ be the 
 tree with a unique fork $f$ such that there are $3$ branches $B_i$, $i=1,2,3$, consisting of $1$, $3$ and $3$ 
 vertices. Assume that $B_1$ is the branch that is a singleton. There are exactly two vertices $v_1,v_2$ - the ends of the branches -
 in $\Gamma_2^\mathbf{0}\backslash B_1$  that are connected to a unique vertex $\Gamma_2^\mathbf{0}\backslash\{v_1,v_2\}$. 
 Let $\mathbf{n}=(n_1,n_2)$. Let $\Gamma_2^{\mathbf{n}}$ be the tree  with a unique fork containing $\Gamma_2^{\mathbf{0}}$ such that $\Gamma_2^{\mathbf{n}}\backslash \Gamma_2^{\mathbf{0}}$ has two connected components $C_1,C_2$ such that $C_i$ contains a vertex connected to $v_i$ and $C_i$ has $n_i$ vertices. The ends of the branches $C_i$ are labelled with $-1$, all other vertices are labelled with $-2$, see Figure \ref{graph2}.
 
Finally,  let $\Gamma_4^{\mathbf{0}}$ be the unique tree with $2$  forks and 
 six vertices. 
 Label the vertices that are not forks from $1$ to $4$ such 
 that one fork is connected to the vertices labelled $1$ and $2$.  Let $
 \mathbf{n}=(n_1,n_2,n_3,n_4)$. Let $\Gamma^{\mathbf{n}}_4$ be the 
 tree with $2$ forks, containig $\Gamma^{\mathbf{0}}_4$ such that   $
 \Gamma^{\mathbf{n}}_4$  minus the forks  
 has $4$ connected components $B_i$ $i=1\dots 4$ such that if  $v$ is the unique  labelled vertex 
 contained in $B_i$, $B_i\backslash\{v\}$ contains $n_j$ vertices, with $j$ the label of $v$. Again, we label the ends of the branches with $-1$ and all other vertices with $-2$, see Figure \ref{graph4}.

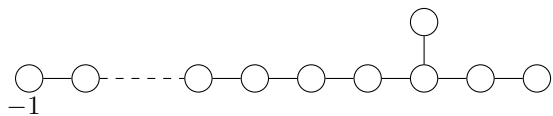
\begin{figure}\centering
\begin{tikzpicture}[scale=0.75]% Intersection graph of blow-up d=2
%\tikzstyle{every node}=[draw,shape=circle];
\node[draw,shape=circle] (1) at (-3,0){}; 
\node(A) at (-3.1,-0.5){\small $-1$};
\node[draw,shape=circle] (2) at (-2,0){}; 
\node[draw,shape=circle] (3) at (0,0){}; 
\node[draw,shape=circle](9) at (1,0){};
\node[draw,shape=circle](10) at (2,0){};
\node[draw,shape=circle] (4) at (3,0){};
\node[draw,shape=circle] (5) at (4,1){};
 
\node[draw,shape=circle] (6) at (4,0){}; 
\node[draw,shape=circle] (8) at (5,0){};
\node[draw,shape=circle] (7) at (6,0){}; 
%\node[fill=black] (8) at (7,0){}; 
\draw (1)--(2)
           (3)--(9)
           (9)--(10)
           (10)--(4)
           (4)--(6)
           (6)--(8)
           (8)--(7)
           (5)--(6);
\draw[dashed] (2)--(3);

\end{tikzpicture}\caption{The intersection graph $\Gamma_1^{\bf n}$.  All vertices with label not displayed are labelled with $-2$.  }
\label{graph1}
\end{figure}

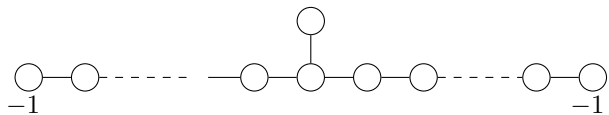
\begin{figure}\centering
\begin{tikzpicture}[scale=0.75]% Intersection graph of blow-up d=2
%\tikzstyle{every node}=[draw,shape=circle];
\node[draw,shape=circle] (1) at (-2,0){}; 
\node(A) at (-2.1, -0.5){\small ${-1}$};
\node[draw,shape=circle] (2) at (-1,0){}; 
\node (2x) at (1,0){}; 

\node[draw,shape=circle] (3) at (2,0){}; 
\node[draw,shape=circle] (4) at (3,0){};
\node[draw,shape=circle] (5) at (3,1){};
 
\node[draw,shape=circle] (6) at (4,0){}; 
\node[draw,shape=circle] (6x) at (5,0){}; 

\node[draw,shape=circle] (7) at (7,0){}; 
\node[draw,shape=circle] (8) at (8,0){}; 
\node(B) at (7.9, -0.5){\small ${-1}$};
\draw (1)--(2)
(2x)--(3)
           (3)--(4)
           (4)--(5)
           (4)--(6)
           (6)--(6x)
           (7)--(8);
\draw[dashed] (2)--(2x)
(6x)--(7);

\end{tikzpicture}\caption{The intersection graph $\Gamma_2^{\bf n}$.   All vertices with label not displayed are labelled with $-2$.   }
\label{graph2}
\end{figure}

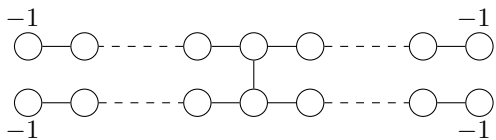
\begin{figure}\centering
\begin{tikzpicture}[scale=0.75]% Intersection graph of blow-up d=4
%\tikzstyle{every node}=[draw,shape=circle];
\node[draw,shape=circle] (1) at (-1,0){}; 
\node (A) at (-1.1,-0.5){\small $-1$};
\node[draw,shape=circle] (2) at (0,0){}; 
\node[draw,shape=circle] (3) at (2,0){}; 
\node[draw,shape=circle] (4) at (3,0){};
\node[draw,shape=circle] (5) at (3,1){};
 
  \node[draw,shape=circle] (11) at (2,1){};
  \node[draw,shape=circle] (12) at (0,1){}; 
  \node[draw,shape=circle] (20) at (-1,1){}; 
  \node (D) at (-1.1,1.5){\small $-1$};
  \node[draw,shape=circle] (13) at (4,1){};
 
  \node[draw,shape=circle] (15) at (6,1){};
   \node[draw,shape=circle] (16) at (7,1){}; 
   \node (B) at (6.9,-0.5){\small $-1$};
  
\node[draw,shape=circle]  (6) at (4,0){}; 
\node[draw,shape=circle] (7) at (6,0){}; 
\node[draw,shape=circle]  (8) at (7,0){}; 
\node (C) at (6.9,1.5){\small $-1$};
\draw (1)--(2)
           (3)--(4)
           (4)--(5)
           (4)--(6)
           (7)--(8)
           (5)--(11)    
              (5)--(13) 
           (12)--(20)
           (15)--(16);
\draw[dashed] (2)--(3)
(11)--(12)
(6)--(7)
  
(13)--(15);       

\end{tikzpicture}\caption{The intersection graph $\Gamma_4^{\bf n}$.   All vertices with label not displayed are labelled with $-2$.   }
\label{graph4}
\end{figure}
\end{construction}
 \begin{lemma}\label{poscoefficient}
 Let $n\in\NN$, $n\geq 2$. Let  $C_i$, $i=1,\dots n$ be a collection of curves on a smooth surface $Y$. Suppose the dual graph is $A_n$, with labelling of the vertices $v_i$ as in Figure \ref{labelvert}.  Assume the labelling is such that $v_i$ corresponds to the curve $C_i$. Moreover, assume $C_i^2=-1$ if $i=1$ and $-2$ if $1<i<n$.  Let $\shH$ be a set of curves on $Y$ such for all $h\in \shH$, $h.c_i=0$ for $i=1,\dots, n-1$. Set $A=\sum_{i=1}^n a_i C_i + \sum_{h\in \shH} \gamma_h h$ for $a_i,\gamma_h\in \QQ$ and assume $A.C_i\geq 0$ for $i=1,\dots, n-1$ and $a_1\geq 0$. Then $a_{i+1}\geq a_{i}$.  If $A.C_i> 0$ for $i=1,\dots, n-1$ and $a_1\geq 0$, it follows $a_{i+1}> a_i$.
 \end{lemma}
\begin{proof}
We have $0\leq A.C_1=-a_1+a_2$, so $a_2\geq a_1$. Suppose $a_i\geq a_{i-1}$ for some $1<i<n-1$. Then   $0\leq A.C_i=a_{i+1}-2a_i+a_{i-1}$, so $a_{i+1}\geq 2a_i-a_{i-1}\geq a_i$.  The claim follows. Replacing weak by strict inequalites shows the second claim.
\end{proof}

\begin{figure}\centering
\begin{tikzpicture}
\tikzstyle{every node}=[draw, shape=circle];
\node (1) at (0,0){$1$}; 
\node (2) at (1,0){$2$};

\node (3) at (2,0){$3$};

\node (7) at (5,0){$n$}; 
\draw[thick] (1)--(2)
(2)--(3);
  \draw[dashed, thick]          (3)--(7);

\end{tikzpicture}
\caption{The numbering of the vertices of the $A_n$ graph in Lemma \ref{poscoefficient}. }
\label{labelvert}
\end{figure}
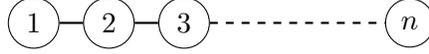

 Let  $\mathfrak{Y}_i$ be as in Construction \ref{constructioncomponents}, with anticanonical cycle $D$,  and let $p$ be a special point. We recall that $p$ is an  \emph{interior} special point if $p$ is a smooth point of $D$.

\begin{proposition}\label{pro:canflops}
 Consider $\mathfrak{Y}_i$ for $i=1,2 $ or $4$, with anticanonical divisor $D$. Let $(p_1,\dots p_i)$ denote the tuple of  interior special points. Let $Y$ denote the ${\mathbf n}=(n_1,\dots n_i)$ blow-up of $\mathfrak{Y}_i$ in $(p_1,\dots p_i)$. 
 
Then there are exactly $i$ exceptional curves  $C_k\subset Y$, $k=1\dots i$ of the first kind such that $C_k$ is not a component of $D$. More precisely, the intersection graph of the negative curves that are not components of $D$ is  given by the 
graph $\Gamma_i^{\mathbf n}$.
\end{proposition}
\begin{proof}
 Let $\psi\colon Y\to \mathfrak{Y}_i$ be the $(n_1,\dots ,n_i)$ blow-up of $\mathfrak{Y}_i$ in the  interior special points $(p_1,\dots, p_i)$. 
 
Let $C(Y)$ be the set consisting of strict transforms of the  $(-1)$ curves $E_i$ from the construction of the $\mathfrak{Y}_i$ in Construction \ref{constructioncomponents}
and  the  $(-2)$-curves  of $\mathfrak{Y}_i$ together with the irreducible components of the exceptional locus of $\psi$, see also Figures \ref{curveconfig} and \ref{curveconfig2}. 
By the results of \cite[Section $2$ ]{Loo}, the set $C(Y)$ defines a $\QQ$-basis of $\Pic(Y)$, with intersection graph $\Gamma_i^\mathbf{n}$ and  $\mathbf{n}=(n_1,\dots n_i)$.
We can assume $i=4$ as it is straightforward to check, using the construction  of the $\mathfrak{Y}_i$, 
that the result for $i=4$ implies the result for  $i=1, 2$.
 
 Let $C_k$, $k=1,\dots,4$ denote the  integral $(-1) $-curves in $\Ex(\psi)$, corresponding to the  vertices labelled with $(-1)$ in the diagrams.  Let $D_1,\dots, D_4$ denote the components of the anticanonical divisor $D$ given as 
 birational transforms of the anticanonical cycle from the construction of $\mathfrak{Y}_4$.  
 Suppose there is an integral curve $C$ with $C^2<0$ not in $\shB:=\{D_1,\dots,D_4\} \cup C(Y)$. We have $C.B\geq 0$ for all curves $B$ in  $\shB$. 
 Write 
\begin{equation}\label{eg:d3}
mC=\sum_{ B\in C(Y)}\beta_B B.
\end{equation}
with  $m, \beta_B \in\ZZ$ and $m> 0$. 
From Lemma \ref{poscoefficient}, it follows that all coefficients  $\beta_B$ are non-negative. 

We find that 
\begin{equation}\label{eg:d4}
0 > mC^2=\sum_{ B\in C(Y)} \beta_B (C.B) \geq 0,
\end{equation}
 where the last inquality follows since $\beta_B \geq 0$ and $C.B\geq 0$.
Hence a curve such as $C$ cannot exist.
\end{proof}

\begin{remark}
The same proof works for $i=3$, the case $i=5$ also holds. We restrict to $i=1,2,4$ as these are the relevant surfaces in the $d=1$ case.
\end{remark}
We have the  the following result.

\begin{corollary}\label{main1} Let $Y_c=Y_1\cup Y_2\cup Y_3$ be in $\Mod_2$. Then the cones of curves $\NE(Y_i^{\nu})$ on the normalisations  of the components are finitely generated. If the Picard rank of $Y_i$
is at least $3$, a generating set is given by the curves $C$ with $C^2<0$. For smaller Picard rank, either $Y_i$ is a Hirzebruch surface or $\PP^2$.
\end{corollary}
\begin{proof} The only thing that remains to show is the statement on the generators of cones of curves. By 
Proposition \ref{pro:canflops}, there are finitely many curves $C$ on $Y_i$ with $C^2<0$, so, 
if the Picard rank is at least $3$, the corollary follows  from \cite{ArLa}.
\end{proof}

%%%%%% %%%%%%%%%%%%%%%%%%%%%%%%%
 \subsection{Curve Structures} Let $Y$ be a component of a surface $Y_c\in \Mod_2$. 
 By definition, 
there is a sequence of type I flops $ Y_\mathscr{G}\dashrightarrow Y_c$ whose inverse connects 
$Y$  to a component of a model $Y_\mathscr{G}\in \DNV_2$ in $(-1)$-form. 
Using this sequence, we will recursively define a set of curves $C(Y)$
on the normalisation $Y^\nu$.
Note that there is an induced sequence of blow-ups 
and blow-downs  $\psi\colon Y^\nu \dashrightarrow \mathfrak{Y}_i$ 
for some $i\in\{1,2,4\}$. 

Let $D$ denote the anticanonical divisor of $\mathfrak{Y}_i$ as discussed in Construction \ref{constructioncomponents}.  
Set 
\[ C(\mathfrak{Y}_i):=\{ C\subset \mathfrak{Y}_i \mid C \text{ is an integral curve with } C^2<0, C\not\subset D \}. \]  To start the induction, we factor 
$\psi$  as $Y^\nu\dashrightarrow W\dashrightarrow \mathfrak{Y}_i$ with $Y^\nu\dashrightarrow W^\nu$ corresponding to 
an elementary modification of type I. 
Suppose we have defined $C(W)$.  If $Y^\nu=\text{Bl}_p(W^\nu)$ for some $p$, we take $C(Y)$ to be the strict transforms of curves in $C(W)$ together with the exceptional curve of $Y^\nu\to W^\nu$. If  $Y^\nu\dashrightarrow W^\nu$ is the inverse of $\pi\colon W^\nu=\text{Bl}_p(Y^{\nu})\to Y^\nu$, 
we set \[C(Y):=\{ C\subset Y^\nu \mid C=\pi(C')  \text{ for } C'\in C(W)\}.\]  
Because the normalisation morphism $Y^\nu\to Y$ identifies  components of the singular locus of $Y_c$, which are part 
of the anticanonical cycle, 
and thus never curves in $C(Y)$, we can also  interpret $C(Y)$ as a collection of curves on $Y$.
We then proceed by induction.
The sequence $Y_\mathscr{G}\dashrightarrow Y_c$ is not uniquely defined, but the set $C(Y)$ and its
intersection graph is independent of the choice such a sequence, as the following lemma shows.
\begin{lemma}
Let $Y$ be a component of $Y_c\in \Mod_2$. The  collection $C(Y)$ and the intersection graph are independent of the sequence $ Y_\mathscr{G}\dashrightarrow Y_c$. 
\end{lemma}
\begin{proof} First, the surface $\mathfrak{Y}_i$ is uniquely determined by the number of components of the double locus $D$ of $Y$. Let $D=D_1+\dots + D_i $ and  $p_1, \dots, p_i$ be
 the interior special points.  
 Let $\bar{D}=\bar{D}_1+\dots+\bar{D}_i$ be the anticanonical cycle on $\mathfrak{Y}_i$. 
 
 We can write the induced sequence $Y^\nu\dashrightarrow\mathfrak{Y}_i$ as
\[
\xymatrix{
&\tilde{Y}\ar@{->}[dl]_\psi\ar@{->}[dr]^\pi&\\
 Y^\nu&&\mathfrak{Y}_i
}
 \] with $\pi$ the blow-up in the interior special points on components $\bar{D}_i$ of $\bar{D}$ such that their 
 strict transform $D_i$ under $Y\dashrightarrow \mathfrak{Y}_i$ has $D_i^2<-1$. Up to trivial 
 modifications, i.e. blowing up a special point and then blowing down the resulting exceptional curve, this 
 morphism is uniquely defined by the self-intersection numbers of  $Y^\nu$. Let $J=\{j_1,\dots, j_k \}$ be the set of indices 
 of components $D_j$ of $D$ with $D_j^2>-1$, where we calculate intersection numbers on the normalisation.
Then  $\psi\colon \tilde{Y}\to Y^\nu$ is the $(n_{j_1},\dots,n_{j{_k}})$-blow-up of $Y^\nu$ in $ (p_{j_1},\dots,p_{j{_k}})$. 
This uniquely defines the curves that are contracted. 
\end{proof}

\begin{definition}\label{def:curvestructure} 
Let $Y$ be as above.
We define $\Gamma_Y$ to be
the dual graph of  $C(Y)$ 
and label its vertices with the self-intersection numbers. We call $\Gamma_Y$
the \emph{curve structure} of $Y$.
We say that a curve structure has \emph{type} $d_i$ if $Y$ maps to $\mathfrak{Y}_i$ under a sequence of type I flops. 
 This is well defined as the number of  components of $D$ is fixed under type I flops.
Note that the chosen anticanonical divisor of $\mathfrak{Y}_i$  has $i$ components.
\end{definition}
\begin{remark}
We will  usually consider $C(Y)$ as a set of curves on $Y^\nu$. In particular, intersection numbers will always be calculated on $Y^\nu$.
\end{remark}
\begin{remark}
Note that for surfaces $Y$ which are $(n,m)$-blow-ups of some $\mathfrak {Y}_i$   this coincides with the set $C(Y)$ which we have used in the proof of Proposition \ref{pro:canflops}.
\end{remark}

We will often interpret  $\Gamma_Y$ as the underlying set of vertices. Also, for any vertex $v\in\Gamma_Y$, we will denote the underlying curve  in $C(Y)$ by $C_v $, but abusing notation, we will often just  write $v$. Note that if $v,w$ are two distinct vertices of $\Gamma_Y$, the number of edges between $v$ and $w$ is either $1$ or $0$ and this is the same as the intersection number $C_w.C_v$.
Two vertices  $v,w$ with  $v.w=1$ for the underlying curves are called \emph{adjacent}.

Given $Y$ with normalisation $Y^\nu$ and  with curve structure  $\Gamma_{Y}$, let $D=\sum D_i$ be the anticanonical divisor of $Y^\nu$. For each $D_i$, append a vertex $v_{D_i}$ to $\Gamma_{Y}$ and for each $v\in \Gamma_Y$ such that $C_v$ is not a component of $D$, add an edge joining $v_{D_i}$ and $v$ if $D_i.C_v  \neq 0$. This defines the \emph{augmented} curve structure $\Gamma_Y^a$. See Figure \ref{fig:aug} for an example.

Note that the existence of an edge between a vertex $v\in \Gamma_Y$  and a vertex corresponding to some $D_i$  does not necessarily imply that the intersection number of the corresponding curves is exactly $1$, it simply means that the intersection is non-empty.  

If $\Gamma_{Y}$ and $\Gamma_{Y'}$ are two curve structures, we then  have an obvious notion of an  isomorphism $\Gamma_Y\to \Gamma_{Y'}$:  a bijection of graphs  
$\Gamma_Y\to \Gamma_{Y'}$ preserving the intersection numbers that extends to a bijection of graphs $\Gamma_Y^a\to \Gamma_{Y'}^a$.

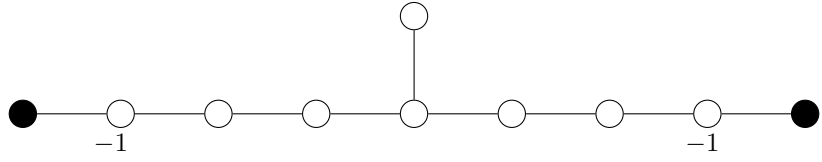
\begin{figure}[]\centering
\begin{tikzpicture}[scale=1.3]

\node[fill=black,draw,shape=circle] (0) at (0,0){};
\node[draw,shape=circle] (1) at (1,0){};
\node (A) at(0.9,-0.3){{\small $-1$}};
\node[draw,shape=circle] (2) at (2,0){};
\node[draw,shape=circle] (3) at (3,0){};
\node[draw,shape=circle] (4) at (4,0){};
 \node[draw,shape=circle] (9) at (4,1){};

  \node[draw,shape=circle] (5) at (5,0){};
   \node[draw,shape=circle] (6) at (6,0){};
    \node[draw,shape=circle] (7) at (7,0){};
\node (B) at (6.95,-0.3) {\small $-1$};

  \node[draw,shape=circle,fill=black](8) at (8,0){};

\draw (0)--(1)--(2)--(3)--(4)--(5)--(6)--(7)--(8)
(4)--(9);

\end{tikzpicture}
\caption{The augmented curve structure of a component $Y$ of $\YP$. The black vertices correspond to the components of the anticanonical divisor.}

\label{fig:aug}
\end{figure}

\begin{definition}Let   $v_e\in \Gamma_Y$  be a vertex  with  $v_e^2=-1$. Let $D_0$ be the unique component of the anticanonical cycle  met by $v_e$.    
The vertex $v_e$  is called  \emph{exceptional} if \[\{ v\in \Gamma_Y \mid v.D_0\neq 0\}=\{v_e\} \text{ and }|\{ v\in \Gamma_Y \mid  v.v_e= 1\}|=1.\]
\end{definition}

Let $Y$ be a component of a surface $Y_c \in \Mod_2$
with curve structure $\Gamma_Y$.  
We will also denote the  preimage of the double locus of $Y_c$ on  $Y$ under the normalisation map $Y^\nu\to Y$ by $D$.
Let $v_e\in \Gamma_Y$ be an exceptional vertex. Starting with this exceptional vertex we now define a subgraph of $\Gamma_Y$ inductively.  Since $v_e$ is exceptional, there is a unique vertex $v_1$ with $v_1.v_e=1$. 
Set $L_1(v_e):=(v_e,v_1).$
Now suppose that we have already defined the ordered tuple of vertices  $L_n(v_e)=(v_e,v_1,\dots, v_n)$.  
If there is a unique vertex $v\in\Gamma_Y\backslash\{v_e,\dots v_n\}$ adjacent to $v_n$ and $v_n.D=0$, set $v_{n+1}:=v$ and $L_{n+1}(v_e)=L_n(v_e)\cup \{v_{n+1}\}$. Else set  $L(v_e):=L_n(v_e)$.  
There is a unique connected subgraph of $\Gamma_{Y}$ whose vertices are given by the vertices in $L(v_e)$.
 By abuse of notation we consider $L(v_e)$ as  this subgraph.

\begin{definition}Let $v_e$ be an exceptional vertex. The graph $L(v_e)$ is the \emph{leg} defined by $v_e$. The unique vertex $v$  of $L(v_e)$ not equal to $v_e$ meeting precisely one vertex of $L(v_e)$ is the \emph{end} of the leg. In this situation we also say that $L(v_e)$ ends in $v$.
\end{definition}

\begin{definition} A curve structure $\Gamma_Y$ is called \emph{degenerate} if  there is no exceptional vertex, or if for some  exceptional vertex $v_e$ the leg $L(v_e)$ ends in a vertex $v$ with $v.D_i=1$ for some  component $D_i\subset D$ such that $D_i$ is in the preimage of a smooth component of the restriction of the double locus of $Y_c$ to  $Y$.

\end{definition}

The   degenerate curve structures  of type $d_2$ 
are displayed in Figure \ref{degcurve}. Curve structures that are not degenerate will be called \emph{non-degenerate}.

\begin{example}\label{exnondeg}
A non-degenerate curve  structure $\Gamma_Y$ has at least one exceptional vertex. If $\Gamma_Y$ is of type $d_2$, then $Y$ is  obtained from $\mathfrak{Y}_2$ by a sequence of blow-ups and blow-downs. It follows from the definitions that $\Gamma_Y$ has two exceptional vertices. More precisely, $\Gamma_Y$ has two legs that both end on the same vertex $v$. Moreover, there is a unique vertex $v'$ that is on none of the legs and also meets $v$.  In particular,  we have 
$D_1^2\leq 1$ and $D_2^2\leq 1$ for the two components $D_1$ and $D_2$ of the anti-canonical divisor.
\end{example}

\begin{figure}\centering

\begin{tikzpicture}

%              
%%leftmost
 \node[draw,shape=circle][fill=black] (D1) at (0,0){};      
  \node[draw,shape=circle] (1) at (1,1){};  
  \node(n1) at (1.5,1){{\small $-2$}};  
\node[draw,shape=circle] (2) at (1,0){}; 
\node(n1) at (1,-0.4){{\small $0$}}; 
 \node[draw,shape=circle][fill=black] (D2) at (2,0){} ;     

\draw[thick]  (1)--(2)
   (D1)--(2)--(D2);

%middle
       \begin{scope}[shift= {(5,0)}]      
      
          \node[draw,shape=circle][fill=black] (D1) at (0,0){};      
  \node[draw,shape=circle] (1) at (1,1){};  
  \node(n1) at (1.5,1){{\small $m$}};  
\node[draw,shape=circle] (2) at (1,0){}; 
\node(n1) at (1,-0.5){{\small $0$}}; 
 \node[draw,shape=circle][fill=black] (D2) at (2,0){} ;     

\draw[thick]  (1)--(2)
(1)--(D1)
   (D1)--(2)--(D2);

         \end{scope}

 \begin{scope}[shift= {(10,0)}]      
      
          \node[draw,shape=circle][fill=black] (D1) at (0,0){};      
  %\node[draw,shape=circle] (1) at (1,1){};  
 % \node(n1) at (1.5,1){{\small $m$}};  
\node[draw,shape=circle] (2) at (1,0){}; 
\node(n1) at (1,-0.5){{\small $1$}}; 
 \node[draw,shape=circle][fill=black] (D2) at (2,0){} ;     

\draw[thick]  
%(1)--(2)
%(1)--(D1)
   (D1)--(2)--(D2);

         \end{scope}

\end{tikzpicture}
\caption{The possible (augmented) degenerate curve structures of type $d_2$ without exceptional vertex. The numbers give the self intersection of the underling curve, with $m\geq -1$. }
\label{degcurve}
\end{figure}
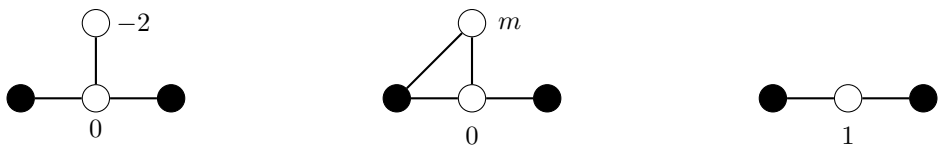

\begin{proposition}\label{qbasis}Let $\Gamma_Y$ be a curve structure of type $d_i$ with $i\in \{1,2,4\}$. Then 
\[ \{C_v\mid v \in \Gamma_Y\}\] is a $\QQ$-basis of $\Pic(Y^\nu)$.
\end{proposition}
\begin{proof}
The statement is true for $\mathfrak{Y}_i$: as mentioned earlier, by the results in   \cite[Section $2$]{Loo}, one obtains a $\QQ$-basis of $
\Pic(\mathfrak{Y}_i)$. In the general case,  the curve structures in question are obtained via blow-ups and blow-downs from the $\mathfrak{Y}_i$, 
$i=i\in \{1,2,4\}$. Arguing inductively, we need to show that given a curve structure $C(Y_0)$, the structure $C(Y)$ obtained on $Y$ under one of these 
operations is also a $\QQ$-basis.
Obviously,  a blow-up $Y^{\nu}\to Y^{\nu}_0$ induces a curve structure  on $Y^{\nu}$ 
that is also a $\QQ$-basis, as it adds the exceptional curve. So let  $g\colon Y^\nu_0\to  Y^\nu$ be the  blow-down in an exceptional curve corresponding to a type I flop, i.e. the contraction of an extremal ray spanned by a curve $E$. We have the exact sequence
\[0\to  \Pic(Y^\nu)\to \Pic(Y^\nu_0)\to \ZZ\to 0\] with first map pullback via $g$ and second map evaluation on $E$. Let $A=\{H_i\}$ denote the collection of elements of $C(Y_0)$ disjoint from $E$. They induce a collection $\{ L_i\}$ with $H_i=g^*L_i$. If $C_j$ is in $C(Y_0)$ 
but not in $A$ and not equal to $E$, then $C_j.E=1$. Hence $C_j+E=g^*(F_j)$ for some $F_j$. 
The collection $\{ L_i,F_j\}Ê$ is linearly independent as $C(Y_0)=\{H_i,C_i+E,E\}$ is a spanning set of $\Pic(Y^\nu_0)_\QQ$ with $\rank \Pic(Y^\nu_0)$ elements, and thus linearly independent. 
Also, note that $g_*H_i=g_*g^*L_i= L_i$ and $g_*(C_j)=g_*(C_j+E)=g_*g^*(F_j)= F_j$.
\end{proof}

We now specialise the discussion to curve structures of type $d_2$. These are the ones appearing on components $Y$ of  surfaces $Y_c$ in $\Mod_2(\mathscr{P})$.  
Note that in this case all components of $Y_c$ are smooth. 
Our goal in the remainder of this section is to express properties of the ample cone of such  $Y$ in terms of curve structures. In the next section, we will see how these properties give criteria for the  projectivity of $Y_c$.  Note that  all such  $Y$ are smooth rational surfaces and the anticanonical divisor $D$ is the restriction of the double locus of $Y_c$ to  $Y$.

\begin{definition}Let $\Gamma_Y$ be a curve structure of type $d_2$.  We say $\Gamma_Y$ is \emph{regular} if  $|\Gamma_Y|>1$ and no leg $L(e)$ of an exceptional vertex $e$ ends in a vertex $v$ with $v^2=0$. \end{definition}

A curve structure that is not regular is called {\em non-regular}.
In particular, a non-degenerate curve structure is regular and a non-regular curve structure is degenerate. Let $\Gamma_Y$ be a non-regular curve structure. Then for each component $D_i$ of the anticanonical divisor $D$ there is a unique vertex
 $v_{D_i}$ of $\Gamma_Y$
 with  $D_i.v_{D_i}>0$.

Let $\Gamma_Y$ be a curve structure. Consider a map  \[f \colon\Gamma_Y\to \ZZ. \]
 By abuse of notation we mean here that  the map is defined on the underlying set of $\Gamma_Y$, i.e. its vertices.
 
This defines a divisor  
\[\Gamma_f=\sum_{v\in\Gamma_Y} f(v)C_v.\]  

Rephrasing Lemma \ref{poscoefficient}, we obtain the following lemma.

\begin{lemma}\label{leg} Let $\Gamma_f.v>0$ for all $v\in\Gamma_Y$. Suppose there is an exceptional vertex $v_0\in\Gamma_Y$  and suppose $f(v_0)>0$. Assume the leg $L(v_0)$ defined by $v_0$ is given by the graph $(v_0,v_1,v_2,\dots, v_n)$. Then $f(v_{i+1})>f(v_i)$ for $i=0,\dots, n-1$.
\end{lemma}

\begin{lemma}\label{lem:critample}
Let $\Gamma_Y$ be of type $d_2$. Let $\Gamma_f$ be a divisor defined by a map $f \colon\Gamma_Y\to \ZZ$. Suppose $\Gamma_f.v>0$ for all $v\in\Gamma_Y$ and assume that  $f(v_e)>0$ for all  exceptional vertices $v_e$  of $\Gamma_Y$ or, if there is no exceptional vertex, $f(c)>0$ for the unique vertex meeting both components of $D$.  
Then $\Gamma_f$ is an ample divisor. 
\end{lemma}
\begin{proof}
If $\Gamma_Y$ does not have an exceptional vertex, it is degenerate.  It follows from the description  of possible degenerate curve structures without exceptional vertices, compare  Figure \ref{degcurve}, that all $f(v)$ are strictly positive. 
The same statement  holds if $Y$ has an exceptional vertex. This follows from Lemma \ref{leg}. 
We next claim that for any curve $C$ not in $C(Y)$ there exists some $C_v, v\in \Gamma_Y$ such that $C.C_v>0$. To see this, let  
$A$ be an ample divisor on $Y$. Then $A.C>0$. As $C(Y)$ is a basis of $\Pic(Y)_\QQ$, we can, replacing $A$ by a  multiple if necessary,  write $A=\sum_{C_v\in C(Y)} \beta_v C_v$ with $\beta_v\in \ZZ$. Then  $C.C_v>0$ for some $v\in C(Y)$, as else $A.C =0$.
 It follows that $\Gamma_f.C>0$ for all integral curves, so $\Gamma_f$ is positive on $\NE(Y)$, as the latter is finitely polyhedral by Corollary \ref{main1}. 
 By Kleiman's criterion, $\Gamma_f$ is ample.
\end{proof}

\begin{proposition}\label{nondegample}Let $Y$ be a component of a surface $Y_c\in \Mod_2(\mathscr{P})$.  Suppose $\Gamma_Y$ is non-degenerate.
 Let $f_0, g_0$  be the exceptional vertices and let $k_1,k_2$ be non-negative integers. 
Then,  there is a positive  integer $\triangle$  such that for any $e>\triangle$, there is a map  $f\colon \Gamma_Y\to \ZZ$ such that $\Gamma_f$ is an ample divisor with $f(f_0)=e+k_1$ and $f(g_0)=e+k_2$.
\end{proposition}
\begin{proof}  Let $c$ be the fork of $\Gamma_Y$ and let
 $(f_0,\dots,f_n,c)$ and $(g_0,\dots,g_m,c)$ denote the legs associated to the exceptional vertices. Let $y$ denote  the unique vertex not appearing on any of the legs. 
 We can assume
 \[k_1+\frac{1}{2}n(n+1)\geq k_2+\frac{1}{2}m(m+1).\]
 Set $\triangle':=k_1+\frac{1}{2}n(n+1)-k_2-\frac{1}{2}m(m+1)+2\max\{n,m\}+2$. 
  Let $e$ be any integer such that $e>\triangle:=2\triangle' +2$.

 Define a map  $f\colon \Gamma_Y\to \ZZ$ by 
 
\begin{align*}
&f(f_i)=e+k_1+\frac{1}{2}i(i+1)\\
&f(g_i)=e+k_2+\frac{1}{2}i(i+1)\\
&f(c)=\max\{f(f_n),f(g_m)\}+\max\{n,m\}+1\\ 
&f(y)= \left \lceil{  \frac{1}{2}e}\right \rceil -1.
%&f(y)=\frac{1}{2}E-1.
\end{align*} Note that by assumption $f(f_n)\geq f(g_m)$. Also, we have $f(f_1)-f(f_0)=1$ 
and \[f(f_{i-1})-2f(f_i)+f(f_{i+1})=1\] for $1\leq 1\leq n-1$. Also, $f(f_{i})-f(f_{i-1})=i$, so
\begin{align*}
-2f(f_n)+f(c)+f(f_{n-1})=\max\{n,m\}+1-n>0.
\end{align*} Analoguously, we find 
$f(g_1)-f(g_0)=1$, $f(g_{i-1})-2f(g_i)+f(g_{i+1})=1$ for $1\leq 1\leq m-1$ and
\begin{align*}
-2f(g_m)+f(c)+f(g_{m-1})\geq\max\{n,m\}+1-m>0.
\end{align*} 
We also have 
\begin{align*}
-2f(c)+f(f_n)+f(g_m)+f(y)&=-2\max\{f_n,g_m\}-2\max\{n,m\}-2+f(f_n)+f(g_m)+f(y)\\
&=-f(f_n)+f(g_m)+f(y)-2\max\{n,m\}-2\\
&=f(y)-\triangle'>0.
\end{align*}
Finally, $ \max\{f(f_n),g(g_m)\}\geq e$, so  $-2f(y)+f(c)=-e+2+\max\{f_n,g_m\}+\max\{n,m\}+1>0.$ 
By Lemma \ref{lem:critample}, $\Gamma_f$ is an ample divisor 
with properties as desired.
\end{proof}

\begin{proposition}\label{nefample}Let $Y$ be a component of a surface $Y_c\in \Mod_2(\mathscr{P})$.  Suppose $\Gamma_Y$ is non-degenerate.  Write $D=D_1+D_2$ for the anticanonical cycle. Assume $D_1^2\leq 0$.  Let $k_1, k_2$ and $\gamma$ be integers with 
$k_1,k_2 \geq 0$
and $\gamma>1$. Then there are positive integers $\triangle$ and $\alpha$ such that for any  $e>\triangle$ divisible by  $\alpha$, % $e=\alpha r>\triangle$ 
there is an ample divisor $A$ with $A.D_1=e+k_1$ and $A.D_2=\gamma e+k_2$.
\end{proposition}
\begin{proof} 
Assume first that $D_1^2<0$. Take an ample divisor $L$ on $Y$ and set $M=-D_1^2 L + (L.D_1)D_1$. Let $C$ be an integral curve on $Y$. If $C$  is distinct form $D_1$ then $M.C= (-D_1^2)L.C +(L.D_1) D_1.C> 0$. 
Also, $M.D_1=-D_1^2 L.D_1 + (L.D_1)D_1^2=0$.  
If  $D_1^2=0$, then set  $M=D_1$. 
In any case there is a  nef divisor  $M$
with  $M.D_2=\alpha>0$. 
and $M.D_1=0$.
Now, by  Proposition \ref{nondegample}, there is an integer $\triangle$ and  such that for any integer  $e>\triangle$ there is an ample divisor $\Gamma_f$ such that $\Gamma_f.D_1=e+k_1$ and $\Gamma_f.D_2=e+k_2$. 
 Fix  $e=\alpha r$ for some integer 
$r>0$ such that  $e>\triangle$.   Then $A:=\Gamma_f+(\gamma-1)rM$ is ample. Also, $A.D_1=e+k_1$ and $A.D_2=k_2+\gamma e$, as desired.
\end{proof}

\begin{example}\label{notample} Proposition \ref{nefample} fails if $D_1^2 > 0$, as  the legs of the curve structure are then too short: 
Let $D'=D_1'+D_2'$ be the anticanonical divisor of $\mathfrak{Y}_2$, $p_i\in D'_i$, $i=1,2$ be the interior special points of $\mathfrak{Y}_2$. Assume that $C$ is the $(-1)$-curve meeting $D'_1$ and $F$ the $(-2)$-curve with $C.F=1$.  For $n\in \NN$  let $Y(n)$ 
be the surface obtained from $\mathfrak{Y}_2$ obtained by blowing up $n$ times in $p_2$ and contracting $C+F$.  
Note that $\Gamma_{Y(n)}$ is non-degenerate of type $d_2$. Let $k$ and $e$ be positive integers.  Suppose  $A$ is an ample divisor with $A.D_1=e$  and $A.D_2=\gamma e+k$ with some integer $\gamma>1$,
and define $f\colon \Gamma_{Y(n)}\to Z$ to be the
map with $A=\Gamma_f$.  Let $v_{e_i}$ be the exceptional vertex meeting $D_i$. We have $L(v_{e_1} )=(e_1,c)$ with $c$  the fork of $\Gamma_{Y(n)}$. Also, write $L(v_{e_2})=(v_{e_2},\dots,w,c)$ and let $y$ be the unique vertex not on any leg. We have $f(c)=\gamma e+\alpha+\beta$ and $f(w)=\gamma e +\beta$ with $\alpha,\beta >0$ by Lemma \ref{leg}. 
By ampleness
$A.c>0$, so $f(e_1)+f(w)+f(y)>2f(c)$ 
and using that $A.D_1=f(e_1)=e$  we obtain
 \[ f(y)>(\gamma-1)e+2\alpha+\beta.\] 
 Intersecting $A$
with $y$ gives $-2f(y)+f(c)>0$. So $0>2f(y)-f(c)>2(\gamma-1)e+4\alpha+2\beta-f(c)$, implying $0>(\gamma-2)e+3\alpha+\beta>0$, a contradiction. 
Hence a divisor such as $A$ cannot exist.
\end{example}

We note the following variant of Proposition \ref{nondegample}.

\begin{proposition}\label{32ample}Let $Y$ be a component of a surface $Y_c$ in 
$\Mod_2(\mathscr{P})$, with $\Gamma_Y$ non-degenerate. Let $D=D_1+D_2$ be the anticanonical divisor. Assume $D^2_1=1$. Let $k \geq 0$ be an integer. Then there is an integer $\triangle>0$ such that for any even integer $e>\triangle$ an ample divisor $A$ on $Y$ exists with $A.D_1=e$ and $A.D_2=\frac{3}{2}e+k$.
\end{proposition}
\begin{proof} 
Let $c$ be the fork of $\Gamma_Y$ and let $v_{e_i}$ denote the exceptional vertex with $D_i.v_{e_i}=1$. Let 
 $L(v_{e_2})=(f_0,\dots,f_n,c)$.  Denote by $y$  the unique vertex not appearing on any of the legs and
  %Set $\triangle:=(n-1)(n+1)+4n+6+k$. 
set $\triangle=n(n+1)+6n+2k+10$. 
  
  Let $e$ be any even integer such that $e>\triangle $. Define a map  $f\colon \Gamma_Y\to \ZZ$ by 
\begin{align*}
&f(v_{e_1})=e\\
&f(f_i)=\frac{3}{2}e+k+\frac{1}{2}i(i+1)\\
&f(c)= f(f_n)+n+1  \\  
&f(y)= \frac{1}{2}e+\frac{1}{2}n(n+1)+2n+3+k.%&f(y)=\frac{1}{2}E-1.
\end{align*}  
Then $-f(v_{e_1})+f(c)=f(f_n)+n+1-e>0.$ As above,  we have $f(f_1)-f(f_0)=1$ 
and \[f(f_{i-1})-2f(f_i)+f(f_{i+1})=1\] for $1\leq i \leq n-1$. Moreover, $f(f_{i})-f(f_{i-1})=i$  for $1\leq i \leq n-1$, so
\begin{align*}
-2f(f_n)+f(c)+f(f_{n-1})=1.
\end{align*}  
We also have 
\begin{align*}
-2f(c)+f(f_n)+f(v_{e_1})+f(y)&=-2f(f_n)-2n-2+f(f_n)+f(v_{e_1})+f(y)\\
&=-f(f_n)+f(v_{e_1})+f(y)-2n-2\\
&=-\frac{3}{2}e-k-\frac{1}{2}n(n+1)-2n-2+e+f(y)\\
&=1.
\end{align*}
Finally 
\begin{align*}
-2f(y)+f(c)&= -e-n(n+1)-4n-6-2k+\frac{3}{2}e+k+\frac{1}{2}n(n+1)+n+1\\
&=\frac{1}{2}e-\frac{1}{2}n(n+1)-3n-5-k\\
&=\frac{1}{2}(e-\triangle )>0.
\end{align*}
Hence $\Gamma_f$ is an ample divisor with properties as desired.
\end{proof}

%%%%%%%%%%%%%%%%%%%%%%%%%%%%%%%%%%%%%%%%%%%%%%%%%
We  will also need to construct ample  divisors on surfaces with degenerate curve structures. For better readability, we will treat the different cases separately. 

\begin{proposition}\label{degample}Let $Y$ be a component of a surface $Y_c\in \Mod_2(\mathscr{P})$.  
Suppose $\Gamma_Y$ is of type $d_2$, regular and degenerate.  Suppose there is an exceptional vertex $v_0$ and that $e$ is an integer, $e>2$.  
Write $D=D_1+D_2$ with  $D_1$ denoting the component of $D$  with $D_1.v_0=1$.
Then there is a map   $f\colon \Gamma_Y\to \ZZ$ such that $\Gamma_f$ is an ample divisor 
 with  degree $e+m$ on $D_2$ and degree $e$ on $D_1$ for some  integer  $m>0$ independent of $e$.
\end{proposition}
\begin{proof}
  Write  $L(v_0)=(v_0,\dots, v_n)$. Let $w$ be the unique element in $\Gamma_Y\backslash L(v_0)$ which exists by regularity.  Let $e\geq 2$. Define $f\colon \Gamma_Y\to \ZZ $ by setting 
 $f(v_i)=e+\frac{1}{2}i(i+1)$ and $f(w)=n+1$. Then $\Gamma_f.v_i=1$ for all $i$ and also $\Gamma_f.w>0$. It follows from  Lemma \ref{lem:critample} that $\Gamma_f$ is an ample divisor with degree $e$ on  $D_1$ and degree  $e+(D_2.w)(n+1)+\frac{1}{2}n(n+1)$ on $D_2$. Since $D_2.w> 0$ we obtain the claim by setting $m=(D_2.w)(n+1)+\frac{1}{2}n(n+1)$.
\end{proof}

Let $\Gamma_Y$ be of type $d_2$, non-regular and with no exceptional vertex. Then $\Gamma_Y$ is a singleton given by a vertex $v$. Conversely, if $Y$ is a component of a surface $Y_c\in \Mod_2(\mathscr{P})$ with $\Gamma_Y$ a singleton, then $\Gamma_Y$ is of type $d_2$, non-regular and with no exceptional vertex.  For later use we note  the following obvious statement.

\begin{proposition}\label{degample2} Let $Y$ be a component of a surface $Y_c\in \Mod_2(\mathscr{P})$.
Suppose $\Gamma_Y$ is a singleton given by a vertex $v$.  Then, after relabelling $D$, we have $D_1.v=1$ and $D_2.v=2$.
 Then, for any $e>0$, $v\mapsto e $ defines an ample divisor with degree $e$ on  $D_1$ and $2e$ on $D_2$ and any ample divisor is of this form.

\end{proposition}

%%%%%%%%%%%%%%%%%%%%%%%%%%%%%%%%%%%%%%%%%%%%%%%%%%%%%%%%%
\begin{proposition}\label{degample3}
Let $Y$ be a component of a surface $Y_c\in \Mod_2(\mathscr{P})$.  Suppose $\Gamma_Y$ is
not regular. 
Suppose there is an exceptional vertex $v_0$ and let $e>2$ be an integer.  Write $D=D_1+D_2$ with  $D_1$ denoting the component of $D$  with $D_1.v_0=1$.
Then there is a map   $f\colon \Gamma_Y\to \ZZ$ such that $\Gamma_f$ is an ample divisor 
 with  degree $2e+m$ on $D_2$ and degree $e$ on $D_1$ for some  integer  $m>0$ independent of $e$.
\end{proposition}
\begin{proof}
As $\Gamma_Y$ is not regular and hence degenerate, $L(v_0)=\Gamma_Y$. Write $L(v_0)=(v_0,v_1,\dots v_n)$  and setting $f(v_i)=e+\frac{1}{2}i(i+1)$ defines an ample divisor (note that $v^2_n=0$) with degree $2(e+\frac{1}{2}n(n+1))$ on $D_2$ and $e$ on $D_1$.
\end{proof}

\section{Projective  Models}\label{sec:projmodel}
In this section we analyse the elements in the set $ \PMod_2$. 

\subsection{Models of class  $\mathscr{T}$}\label{subsec41}
Recall from Example  \ref{DNV2} that $\YT$ has a component 
with normalisation $\mathfrak{Y}_4$. Denote  this component  by $(\YT){_{\omega}}$.
Let $Y_c\in \Mod_2(\mathscr{T})$. There is a sequence of type I flops to $\YT$.  Let $Y_\omega$ be the image of    $(\YT){_{\omega}}$ under the induced birational map. It is independent of the chosen sequence.  We call $Y_\omega$ the \emph{special} component of $Y_c$. 

\begin{proposition}\label{projTcurves}Let $Y_c\in \Mod_2(\mathscr{T})$, $Y_\omega$ the special component. Let $D_\omega$ be the smooth component of the boundary curve of $Y_\omega$. Let $D_1$ and  $D_2$ be the disjoint curves in the preimage of $D_\omega$  under the normalisation $\pi\colon Y^\nu_\omega\to Y_\omega$. Then $Y_c$ is projective if and only if $D_1^2=D_2^2=-1$.
\end{proposition}
\begin{proof}Let $Y_c=Y_1\cup Y_2\cup Y_3$  and suppose $Y_\omega=Y_2$.
We begin by showing that  $Y_c\in  \PMod_2(\mathscr{T})$
if and only if $Y_\omega$ is projective.
Projectivity of $Y_c$ clearly implies projectivity of $Y_w$, so we only need to show the opposite implication.
So suppose $Y_\omega$ is projective.  Because $Y_c$ 
is in $\Mod_2$ and has dual intersection complex $\mathscr{T}$,  the remaining two components of $Y_c$ are obtained by blow-ups and blow-downs of $\mathfrak{Y}_1$ in the interior special point, and thus are projective. Hence there are  ample line bundle  $A_i$ on $Y_i$ for $i=1,2,3$. The components $Y_1$ and $Y_2$  are glued by identifying a curve $D_{12}\subset Y_1$ with  a curve $D_{21}\subset Y_2$. Similarly,  $Y_3$ and $Y_2$  are glued by identifying a curve $D_{32}\subset Y_3$ with  a curve $D_{23}\subset Y_2$.  Replacing the  $A_i$ by suitable multiples, we can assume that $ A_1.D_{12}=A_2.D_{21}$ and $A_2.D_{23}=A_3.D_{32}$.
Because $Y_c$ has trivial Carlson extension,  one obtains an ample bundle on $Y_c$ by Lemma \ref{degree} and \cite[Proposition 1.2.16]{Lazar},
implying that $Y_c$ is projective. 

We now show that projectivity of $Y_w$ is equivalent to the numerical conditon stated in the propositon.
 Suppose first that $D_1^2=D_2^2=-1$.  One can, similar as in the proof of Proposition \ref{max:proj}, 
find an ample divisor on $Y_\omega$. Indeed, take an ample divisor $L'=L'_\omega$ on the normalisation $Y^\nu_\omega$. Suppose without loss of generality $L'.D_1\geq L'.D_2$.  Set $L''=(L'.D_2)L'$.  Set $M=L'+(L'.D_1)D_1$. Then $M.C\geq 0$ for all curves $C$ on $Y^\nu_\omega$ with equality if and only if $C=D_1$.  Write $n=L'.D_1-L'.D_2$.  Note that $D_1.D_2=0$. Then $L_\omega=L''+nM$ is ample by the
Nakai-Moisjezon criterion
and 

\begin{align*}
L_\omega.D_2&=L''.D_2+nM.D_2\\
&=(L'.D_2)(L'.D_2)+(L'.D_1)(L'.D_2)-(L'.D_2)(L'.D_2)\\
&=L''.D_1+nM.D_1\\
&=L_\omega.D_1.
\end{align*}
Hence $L_\omega$ induces an ample divisor on $Y_w$. 

We now show that the condition on the intersection numbers is also necessary. So suppose that $D_1^2<-1$. Note that  $D_2^2= -2-D_1^2$. 
Suppose $L'$ is an ample line bundel on $Y_\omega$.  The normalisation $\nu \colon Y_\omega^\nu \to Y_\omega$ is finite, so $L=\nu^*L'$ is ample.

By Proposition \ref{qbasis},  replacing $L$ by a suitable multiple we can write
 \[L=\sum_{v\in C(Y)}a_vC_v\] with $a_v\in \ZZ$. 
  
 As $L$ is ample, the coefficients of vertices meeting the anticanonical divisor are strictly positive. This implies $a_v>0 $ for  all $v$ by Lemma \ref{leg}, as either $v$ is contained in the  leg defined by  some exceptional vertex or meets a component of the boundary. 
 By the condition on the self-intersection of $D_1$ the curve structure $\Gamma_Y$ is degenerate: there is an exceptional vertex $e$ with $e_1.D_1=1$ and $L(e)$ ends in a vertex $w$  with $w.D_2=1$. 
By Lemma \ref{leg},
 \[  L.D_2>a_w>a_e= L.D_1,\] 
 so $Y_\omega$ cannot be projective.
\end{proof}
In particular, the proposition says that  $Y_c\in \Mod_2$ is in  $  \PMod_2(\mathscr{T})$ if and only if there is a sequence of type I flops \[Y_c\dashrightarrow \dots\dashrightarrow Y_i \dashrightarrow \dots\dashrightarrow \YT,\]  such that any  flopping curve meets the nodal components of the double locus on the  special component of $Y_i$. 

\subsection{Models of class $\mathscr{P}$}

\begin{proposition}\label{nondegproj2}
Let $Y_c\in \Mod_2(\mathscr{P})$ and write $Y_c=Y_1\cup Y_2\cup Y_3$. 

\begin{itemize}
\item[(i)] If $\Gamma_{Y_i}$ is degenerate for all $i=1,2,3$, then $Y_c$ is not projective.
\item[(ii)]Suppose $\Gamma_{Y_i}$ is regular for all $i$. If there  is a component $Y_i$ such that $\Gamma_{Y_i}$ is non-degenerate, then $Y_c$ is projective.
\end{itemize}
\end{proposition}
\begin{proof}
Suppose first 
 that the curve structures $\Gamma_{Y_i}$ are all degenerate. We show that this cannot happen if $Y_c$ is projective.
 So let $A$ 
be an ample divisor on  $Y_c$. Denote the restriction to $Y_i$ by $A_i$.  
 The set of underlying curves of $\Gamma_{Y_i}$  is a basis of $\Pic(Y_i)_\QQ$ and thus after replacing $A$ by a suitable multiple, we write $A_i=\Gamma_{f_i}$ for a suitable function $f_i\colon \Gamma_{Y_i}\to \ZZ$.
 Note that
at least one of the $Y_i$, say $Y_1$, has an exceptional vertex $v_e$. Let the leg of $v_e$ be given by $L(v_e)=(v_e,\dots, v_n)$. 
Suppose the component of the anticanonical divisor met by $v_e$ is $D_{13}$ while $v_n$ meets $D_{12}$.
Necessarily $f(v_e)>0$, and thus we have $f(v_n)>f(v_e)$ by Lemma \ref{leg}. In particular, $A_1.D_{12}>A_1.D_{13}$.
Now consider $Y_2$.  If $\Gamma_{Y_2}$ has an exceptional vertex, we find  $A_2.D_{23}>A_2.D_{21}$ by the same reasoning.
If $\Gamma_{Y_2}$ has no exceptional vertex and 
 $|\Gamma_{Y_2}|>1$, then $\Gamma_{Y_2}$ has two elements, say $v_1$ and $v_2$ with indices chosen such that $v_1^2=0$, $v_1.v_2=1$, $D_{21}.v_2=0$ and $D_{21}.v_1=1$.  Write $A_2=a_1 v_1+a_2v_2$. Then $0<A_2.v_1=a_2$ and  $0<A_2.D_{21}=a_1$. So $A_2.D_{23}\geq a_1+a_2>a_1=A_2.D_{21}$. 
 If $|\Gamma_{Y_2}|=1$, we have  $A_2.D_{23}=2A_2.D_{21}$ by virtue of Proposition \ref{degample2}. 

The same reasoning applied to $Y_3$ yields the chain of inequalites 
 \[A_2.D_{23}\geq A_2.D_{21}=A_1.D_{12}>A_1.D_{13}=A_3.D_{31}\geq A_3.D_{32}=A_2.D_{23},\]
where the equalities come from the gluing condition. Hence $Y_c$ is not projective.

Conversely, suppose  the curve structure $\Gamma_{Y_1}$ is non-degenerate. For large enough $e$, there is an ample divisor $A_2$  on $Y_2$ with  $\deg A_2{_{|D_{21}}}=e$ and  $\deg A_2{_{|D_{23}}}=e+k_1$, with $k_1\geq 0$ and independent of $e$, by Proposition \ref{nondegample} if $\Gamma_{Y_2}$ is non-degenerate or Proposition \ref{degample} in the degenerate case. Similarly, maybe increasing $e$, one finds an ample divisor $A_3$ on $Y_3$ with $ A_3.D_{32}=e+k_1$ 
and  $ A_3.D_{31}=e+k_2$, $k_2\geq 0$. 
Then, maybe again increasing $e$, there is an ample divisor $A_1$  on $Y_1$ with $A_1.D_{13}=e+k_2$ and  $ A_1.D_{12}=e$, as $\Gamma_{Y_1}$ is non-degenerate and hence Proposition \ref{nondegample} applies.
Hence $Y$ carries an ample bundle and thus is projective.
\end{proof}
\begin{remark}
The non-degeneracy condition of Propositon \ref{nondegproj2} implies that $Y_c\in  \PMod_2(\mathscr{P})$
if and only if there is an index $i$ such that  $D^2_{ij}\leq 1$  for all $j\neq i$, with self intersections considered on $Y_i$.
\end{remark}
%%%%%%%%%%%%%%%%%%%%%%%%%%%%%%%%%%%%%%%%%%%%%%%%%%%%%%%%%%%%%%%%%%%%%%%%

Recall that if $\Gamma_{Y_i}$ is a curve structure of type $d_2$ that is not regular and $D_{ij}$ is a component of the anticanonical cycle, then there is a 
unique vertex $v_{D_{ij}}\in \Gamma_{Y_i}$ such that $v_{D_{ij}}.D_{ij}\neq 0$.
Note that if  $Y=Y_1\cup Y_2\cup Y_3$ is of class $\mathscr{P}$ such that   $\Gamma_{Y_1}$ is non-degenerate and  $\Gamma_{Y_i}$ is not regular for $i=2,3$, then there exists an  $i$ such that $v_{D_{i1}}.D_{i1}=2$.

\begin{proposition}\label{nondegproj3}
Let $Y_c\in \Mod_2(\mathscr{P})$ and write  $Y_c=Y_1\cup Y_2\cup Y_3$. Suppose $\Gamma_{Y_1}$ is non-degenerate and $\Gamma_{Y_i}$ is not regular for $i=2,3$. 
\begin{itemize}
\item[(i)] If $v_{D_{i1}}.D_{i1}=2$ for $i=2,3$, then $Y_c$ is projective.
\item[(ii)] If $v_{D_{i1}}.D_{i1}=1$ for exactly one $i\in\{2,3\}$, then $Y_c$ is projective if and only if  $D^2_{1i}\leq 0$.
\end{itemize}
\end{proposition}

\begin{proof}
We begin by  proving the first claim. By Propositions \ref{degample2} and \ref{degample3}, given any sufficiently large $e$, there are ample divisors  $A_2$ and $A_3$ on $Y_2$ and $Y_3$ with $A_2.D_{23}=A_3.D_{32}=e$ and  $A_i.D_{i1}=2e+n_i$, $i=2,3$,  with $n_i\geq 0$ and $n_i$ independent of $e$. 
As $Y_1$ is non-degenerate, after maybe increasing $e$, by Proposition \ref{nondegample},  there is an  ample divisor $A_1$ on $Y_1$ with $A.D_{1i}=2e+n_i$, 
hence the $A_i$'s glue to an ample divisor on $Y_c$.

In the second case, we can assume  $v_{D_{21}}.D_{21}=1 $. If  $D^2_{12}>0$,  then in fact  $D^2_{12}=1 $, as $Y_1$ is non-degenerate.   We want to show that there is no ample bundle on $Y_c$. To get a 
contradiction, 
suppose $A$ is ample on $Y_c$. Write $A_i$ for the restriction to $Y_i$. 
Write  $A_2.D_{21}=e>0$.  Note that $|\Gamma_{Y_2}|>1$. So  by Proposition \ref{degample3} $A_2.D_{23}=2e+n_1$ for some number $n_1>0$.  

As $v_{D_{31}}.D_{31}=2$, by the same results and the gluing condition, we have $A_3.D_{32}
=2e+n_1$ and $A_3.D_{31}=4e+n_2$ for some $n_2>0$. It follows that  $A_1.D_{12}=e$ and $A_1.D_{13}=4e+n_2$ by the gluing condition. But 
then $A_1$ cannot be ample by Example \ref{notample}.

 If  $D^2_{12}\leq0$, we can, by  Lemma \ref{leg} and 
Proposition \ref{degample2},  find an ample divisor on $A_2$  on $Y_2$ with  $A_2.D_{21}=e$ and  $A_2.D_{23}=2e+n_1$
with $n_1$ independent of $e$ for any $e$ sufficiently large.  Similarly, we find an ample $A_3$  on $Y_3$ with  $A_3.D_{32}
=2e+n_1$ and $A_3.D_{31}=4e+n_2$, $n_2$ independent of $e$. By Proposition \ref{nefample}, there is a ample $A_1$ on $Y_1$ with  $A_1.D_{12}=e$ and $A_1.D_{13}=4e+n_2$. By construction, the $A_i$ glue to an ample bundle $A$ on $Y_c$, so $Y_c$ is projective.
\end{proof}

The following proposition is proven by the same reasoning as Proposition \ref{nondegproj3}.

\begin{proposition}\label{nondegproj4}
Let $Y_c\in \Mod_2(\mathscr{P})$ and write  $Y_c=Y_1\cup Y_2\cup Y_3$. Suppose $\Gamma_{Y_1}$ is non-degenerate, $\Gamma_{Y_2}$ is not regular and $\Gamma_{Y_3}$ is degenerate and regular.
\begin{itemize}
\item[(i)] If $v_{D_{21}}.D_{21}=1$ then $Y_c$ is projective if and only if  $D^2_{12}\leq 0$.
\item[(ii)] If $v_{D_{21}}.D_{21}=2$ then $Y_c$ is projective if and only if  $D^2_{13}\leq 0$. 
\end{itemize}
\end{proposition}

\begin{proposition}\label{nondegproj5}
Let $Y_c\in \Mod_2(\mathscr{P})$ and write  $Y_c=Y_1\cup Y_2\cup Y_3$. Suppose $\Gamma_{Y_1}$ is not regular and  $\Gamma_{Y_i}$ is non-degenerate for $i=2,3$. Also, assume $v_{D_{12}}.D_{12}=2$. Then $Y_c$ is projective.
\end{proposition}
\begin{proof}Note that by non-degeneracy, $D^2_{31}\leq 1$ and $D^2_{23}\leq 1$. 
 If $D^2_{31}\leq 0$ or $D^2_{23}\leq 0$ then $Y_c$ is projective: this follows as in the proof above. 
 If $D^2_{31}=1$ and  $D^2_{23}= 1$, the intersection numbers and the conditions on degeneracy and regularity fix the components, so there is only one such surface.  It follows from  Propositon \ref{32ample} that 
 for sufficiently large $e$ 
 there is an ample divisor $A_3$ on $Y_3$ with  $A_3.D_{31}=e$ and  $A_3.D_{32}=\frac{3}{2}e$. 
 
 We  can take $e$ to be divisible by $4$.  Since $D^2_{31}=1$, $\Gamma_{Y_1}$ has an exceptional vertex $v_0$. Let $L(v_0)=(v_0,v_1,\dots,v_n)$ be the associated leg. Then $v_n^2=0$ as $\Gamma_{Y_3}$ is not regular. So $C_{v_n}$ is a nef divisor on $Y_1$. By Proposition \ref{degample3}, there is an ample divisor $A'_1$ on $Y_1$ with $A'_1.D_{13}=e$ and $A'_1.D_{12}=2e+k$, with $k$ independent of $e$. Set $A_1=A'_1+\frac{1}{4}eC_{v_n}$. This is an ample divisor with  $A_1.D_{13}=e$ and $A_1.D_{12}=\frac{9}{4}e+k$. By Proposition \ref{32ample}, there is an ample divisor $A_2$ on $Y_2$ with $A_2.D_{23}=\frac{3}{2}e$ and $A_2.D_{21}= \frac{9}{4}e+k$. Hence the $A_i$ glue to an ample divisor on $Y$.
\end{proof}

\section{Flops, Birational Automorphisms and Orbits}\label{sec:flopauto}

\subsection{Flops}

In this section we analyse possible flops of models of the DNV family of degree $2$. We first recapitulate Lemma 1.29 from \cite{theta15}, whose proof is an elementary but lengthy computation.

\begin{lemma}[\cite{theta15}]\label{exlocus} Let $\shY\to S$ be a  model of the Dolgachev-Nikulin-Voisin family of degree $2d$, and let $\pi\colon \shY\to \shZ$ 
be a birational contraction (over $S$) whose exceptional locus $\Ex(\pi)$ intersects $\shY_c$ in a one dimensional scheme. Let $\nu \colon \shY_c^\nu\to \shY_c$ 
be the normalisation. Then each connected component $C$ of $\Ex(\pi)\cap \shY_c$ is one of the following:
 \begin{itemize}
 \item[(i)] A non-singular irreducible component of the double curve of $\shY_c$. Necessarily $\nu^{-1}(C)$ consists of two copies of $\PP^1$, each with self intersection $-1$ in $\shY^\nu_c$.
 
 \item[(ii)] A tree of rational curves contained in an irreducible component of $\shY_c$. This tree is of the form $\nu(C')$, where $C'$ is a connected tree of rational curves contained in one connected component of $\shY^\nu_c$, intersecting its boundary transversally in one point.
 
 \item[(iii)] A tree of rational curves in an irreducible component of $\shY_c$, disjoint from the double curve.
 
 \item[(iv)] $\nu^{-1}(C)$ consists of two connected components, each intersecting the boundary of the connected component of $\shY^\nu_c$ containing it transversally in one point. 
 These two intersection points are identified under $\nu$.
 \end{itemize}
\end{lemma}

\begin{remark}
In $(\rm{iv})$, the two connected components can lie in one or two components of $\shY_c$. 
\end{remark}

The lemma implies the following description of flopping contractions.

\begin{proposition}\label{irredcomp} Let $\shY\to S$ be a  model of the Dolgachev-Nikulin-Voisin family of degree $2d$, and let $\pi\colon \shY\to \shZ$ be a birational contraction (over $S$) with  $\codim \Ex(\pi)= 2$, 
i.e. a flopping contraction. Let $C$ be a  connected component
 of $\Ex(\pi)$.  
Let $\nu \colon \shY_c^\nu\to \shY_c$ be the normalisation. 
Then $\nu^{-1}(C)$ is an integral $(-1)$-curve if $C$ is not contained in the singular locus of $\shY_c$ and a disjoint union of two integral $(-1)$-curves
 if $C$ is contained in the singular locus.
\end{proposition}
\begin{proof}First, note that $\Ex(\pi) \subset \shY_c$ 
and hence the connected components of $\Ex(\pi) $ are as in Lemma \ref{exlocus}.
In particular, one of the cases applies to $C$.  In case $\rm(i)$ 
 of Lemma \ref{exlocus} there is nothing to show. 
Suppose we are in  case  $\rm(ii)$
of the Lemma.
 Let $Y_i$ be the component 
 of the central fibre containing $C$. Then there is an induced contraction $Y_i^\nu\to Z_i^\nu $ with $Z_i^\nu$ the normalisation of a component of the central fibre of $\shZ$.   The intersection matrix of the curves contracted by $Y_i^\nu\to Z_i^\nu $  is negative definite, see e.g. \cite[Lemma 3.40]{KoMo}.
  In particular, if $C$ is reducible there is an irreducible  component $C_0\subset C$ with $C_0^2=-2$.
 As $C_0$ does not meet the boundary, we can extend  it trivially to a divisor $L_c$ on $\shY_c$ and then find,  by maximality of the DNV family, a divisor $L$ on $\shY$ restricting to $L_c$. 
 Let $F$ be a nef divisor inducing the flopping contraction $\pi$, with restriction $F_c$ to $\shY_c$. As $C$ is contracted, $F_c.L_c=0$. Also, $L_c^2=-2$.  Let $L_\eta$, $F_\eta$ denote the restrictions to the generic fibre.  We note that the intersection numbers on $\shY_c$ can be calculated on the normalisation $\shY^\nu_c$, see eg \cite[Corollary 1.11]{bad13}.
 Because $\shY\to S$ is flat, we have $L_\eta ^2=L_c^2$ and $F_\eta.L_\eta=F_c.L_c$, by constancy of the Euler characteristic, see \cite[Proposition VI.2.9]{kollrational}. By a standard argument for K3 surfaces, either $L_\eta$  or $-L_\eta$ is effective, see \cite[2.1.4]{Huy}. In any case, there is an integral curve on $Y_\eta$ contracted by $F_\eta$. But  $\pi$ is a flopping contraction, so this is a contradiction.  Hence  $C$ is irreducible. As $C$  meets the double locus and and $C^2<0$, it is necessarily a $(-1)$-curve. 
In the situation of $\rm(iii)$  of the lemma, there is again a $(-2)$-curve on some component. In case $\rm(iv)$ one can lift the bundle given by the two irreducible $(-1)$-curves and conclude as above, as the self intersection is again $-2$.
\end{proof}

\begin{remark}\label{dic:change}
In our case $d=1$ the following holds: If $\phi\colon \shY\dashrightarrow\shY^+$ is a flop defined by a flopping contraction $\pi\colon \shY\to \shZ$, then the dual intersection complexes of $\shY$ and $\shY^+$ are the same if and only if  no component
of $\Ex(\phi)$ is contained in the singular locus of $\shY_c$. This follows from the fact that any component of the singular locus meets any other such component in a triple point. 
\end{remark}

\begin{remark}\label{rmkflops}
We include some remarks on terminology and on factorisation of maps into flops. We also fix some notation which we will use throughout the paper.
Let $\phi\colon\shY\dashrightarrow \shY'$ be a birational map between two models of the DNV family that is not an isomorphism. Then it is an isomorphism in codimension $1$. Let $F'$ be an effective ample divisor on $\shY'$ and let 
$F$ be its birational transform on $\shY$. 
Then there is a sequence of $F$- flops factoring $\phi$, i.e there is 
a sequence 
\[ \shY\dashrightarrow\shY_1\dashrightarrow \dots \dashrightarrow\shY_i\dashrightarrow\dots\dashrightarrow \shY_n\cong\shY'.\]  
such that each $\shY_i\dashrightarrow \shY_{i+1}$ is the flop defined by an $\tilde{F}$-flopping contraction where  $\tilde{F}$ is the birational transform of $F$ under $\phi$, 
by Remark \ref{rem:factorizationflops}. 
We shall write $F_k$ for this birational transform of $F$ under $
\shY\dashrightarrow \shY_k$ and  $\phi_k\colon \shY_k\to {\shZ_k}$ 
for the flopping contraction defining $\psi_k\colon\shY_k\dashrightarrow 
\shY_{k+1}$. We also write $C_k$ for the birational transform of $C$ under $\prod_{i=0}^k \phi_i$ if $C$ is not contracted under this map, and similar for components $Y_i$. We further employ a similar notation for the double 
curves $D_{ij}\subset Y_i$, so $(D_{ij})_k$ is the birational transform of $Y_i\cap Y_j$ considered as a curve on $(Y_i)_k$. If we consider several maps at once, we will decorate the notation accordingly, i.e. here we would write $C_k^\phi$, $(D_{jj})_k^\phi$ etc.
Also, note that if  $C$ is a curve in $\shY_{k}$ with $F_k.C>0$ and $C$ is 
disjoint from $\Ex{\phi_k}$, then $F_{k+1}.(\psi_k)_*C>0$.  
Finally, given an $F$-flop
\[
\xymatrix{
\shY \ar@{-->}[rr]\ar[dr]_\pi& & \shY^+\ar[dl]^{\pi^{+}}\\
&\shZ& \\
}
\] 
a curve in $\Ex(\pi)$ will be called a \emph{flopping curve } and a curve in $\Ex(\pi^{+})$ will be called a \emph{flopped curve }. 
\end{remark}

%%%%%%%%%%%%%%%%%%%%%%%%%%%%%%%%%%%%

We want to classify the flopping contractions of a Kulikov model $\shY\to S$ of the  Dolgachev-Nikulin-Voisin family of degree $2$ by their exceptional loci. 
The next proposition shows that an elementary modification connecting central fibers of projective models lifts to a flop of the models.

\begin{proposition}\label{typeflop}Let $\shY $ and $\shY'$ be models of the DNV family of degree $2$, with central fibres $\shY_c$  and $\shY'_c$ respectively. Suppose that there exists a type I or type II elementary modification $Y_c\dashrightarrow Y'_c$ contracting a curve $C$. Then there is a flopping contraction $\pi$ contracting precisely $C$, inducing a flop $\shY\dashrightarrow \shY'$.
\end{proposition}
\begin{proof}
Let $\shX$ and $\shX'$ be the maximal analytic smoothings of $\shY_c$ and $\shY'_c$.  Both are projective because the central fibres are and every line bundle  of $\shY_c$ extends to $\shX$ by maximality, and similarly for $\shY'_c$. 
It follows from \cite[Theorem 6.38]{KoMo}  that there is a flop  $\psi'\colon \shX \dashrightarrow \shX'$ 
given by some divisor $F$, with restriction $F_c$
 on $\shY_c$. 
  The curve $C$  generates an extremal ray of $\shX$. 
Then $C$  also generates an extremal ray of $\shY\to S$ and $F_c$ lifts to a divisor $\shF\in \Pic(\shY/S)$. 
Hence, as in the proof of Proposition \ref{modelflop}, $\shF$ defines a contraction $\pi\colon\shY\to 
\shZ$, auch that the restriction of the exceptional locus  to $\shY_c$ is $C$. It is either divisorial or small. Suppose it is divisorial. Then the exceptional divisor $E$ restricts to a multiple of $C$ with multiplicity $a>0$.
We shall show that the curve  $C$ then leads to a contradiction of the conditions of Lemma \ref{degree}. 
To see this, write 
$\shY_c=Y_1\cup Y_2\cup Y_3$. Assume first $
\shY_c\in \PMod(\mathscr{P})$. Suppose  $C$ is 
the exceptional locus of an elementary 
modification of  type I. Assume the notation is 
such that $C\subset Y_1$ and $C.D_{12}=1$. Let 
Write $r_i$ for the inclusion $Y_i\to \shY$. Then 
\[ r_1^*E.D_{12}=aC.D_{12}=a > 0\text{ and } r_2^*E.D_{21}=0.\]

 If $C$ is the exceptional locus of an elementary modification of type II, assume the normalisation of $C$ is $D_{12}\cup D_{21}$. Then 

\[ r_1^*E.D_{13}=aC.D_{13}=2a > 0 \text{ and } r_3^*E.D_{31}=0.\] 

Next assume that $\shY_c\in \PMod(\mathscr{T})$, and assume further that $Y_1$ is the special component.  The proof is the same if 
$C$ is the exceptional locus of an elementary modification of type I. If $C$ is the exceptional 
locus of a type II modification, necessarily $C$ is the smooth component of the double locus on the special 
component $Y_1$. To see this, suppose this is not the case. Then we can  write $D_{1j}\cup D_{j1}$ for the preimage of $C$ under the normalisation $\nu\colon \shY_c^\nu\to \shY$,  with $j\in \{2,3\}$. By  Proposition \ref{irredcomp},
 $D_{1j}^2=D_{j1}^2=-1$. But  $D^2_{12}+ D^2_{21}=0$ and $D^2_{13}+D^2_{31}=0$ by Definition
\ref{def:-1form} and the definition of $\PMod(\mathscr{T})$. Hence $C$ is indeed the smooth component of the double locus on $Y_1$. Write $C_1\cup C_2$ for the preimage of $C$ under $\nu$.   We have \[ r_1^*E.D_{12}=aC_1.D_{12}+aC_2.D_{12}=2a > 0 \text{ and } r_2^*E.D_{21}=0.\] 
So $r_c^*E$ does not fullfill the degree conditions of Lemma \ref{degree}. It follows that the extremal ray generated by $C$ defines a small contraction.
\end{proof}
 
 For the statement of the next corollary, recall that a surface $Y_c\in \Mod_2(\mathscr{T}) $ has one component $Y_\omega$ that is not smooth, the special component. 
 
\begin{corollary} \label{typeIflopexist}
Let $\shY$ be a model of the DNV family of degree $2$
with $\shY_c\in  \PMod_2(\mathscr{T})$.  
Let $\phi \colon \shY_c\dashrightarrow Y^+_c$ 
be an elementary modification of type I, contracting a curve $C$, such that $C$ does not meet the singular locus of the special component $(\shY_c)_\omega$. Then there is a flopping contraction  $\shY\to \shZ$,  such that for the  induced flop $\shY^+$, we have $\shY^+_c= Y^+_c$.
\end{corollary}
\begin{proof}  Let $Y_\omega$ be the surface obtained by applying $\phi$ to $(\shY_c)_\omega$. Let $D_1\cup D_3$ be the preimage of $\Sing Y_\omega$ under the normalisation $ Y^\nu_\omega \to Y_\omega$. As $C$ does not meet the singular locus of the special component,  $D_1^2=D_3^2=-1$. Hence $Y^+_c\in \PMod_2(\mathscr{T})$  by Proposition \ref{projTcurves}.
\end{proof}

\begin{corollary} \label{typeIIflopexist}
Let $\shY$ be a model of the DNV family of degree $2$, with $\shY_c\in  \PMod_2$.  
Let $\phi \colon \shY_c\dashrightarrow Y^+_c$ 
be an elementary modification of type II, contracting a curve $C$. Then there is a flopping contraction  $\shY\to \shZ$,  such that for the  flop $\shY^+$ we have $\shY^+_c= Y^+_c$. 
\end{corollary}

\begin{proof}
If $\shY$ is of class $\mathscr{P}$ and we apply a type II modification, then  $Y^+_c$ has dual intersection graph $\mathscr{T}$, as follows from the description of elementary modifications in Section 1.1. 
By the definition of elementary modifications of type II, $Y^+_c$ fullfills the condition of Proposition \ref{projTcurves} and hence $Y^+_c$ is projective. The result then follows  from Proposition \ref{typeflop}.

If $\shY$ is of class  $\mathscr{T}$, then $Y^+_c$ has dual intersection graph $\mathscr{P}$. By the definiton of elementary modifications of type II, $C$ is the smooth component of the restriction of the double locus of $\shY_c$ to $(\shY_c)_\omega$. Let $A$ be an ample line bundle on $(\shY_c)^\nu_\omega$ and let $C_1$, $C_2$ be the components of the preimage of $C$ under $(\shY_c)^\nu_\omega\to (\shY_c)_\omega$.  Here, we choose $A$ such that $A.C_1=A.C_2$ which can be done as $(\shY_c)_\omega$ is projective.  Note that $C_i^2=-1$. So $L'=A+(A.C_i)C_i$ has degree $0$ on the $C_i$ and strictly positive degree on all other curves. By the usual arguments,  there is a divisor $L$ on $\shY_c$ that restricts to a positive multiple of  $L'$ on $(\shY_c)_\omega$ and to an ample divisor on the remaining components. It follows that there is a nef divisor $\shL$ on $\shY$ that has degree zero precisely on $C$. 
The birational morphism induced by $\shL$ is a flopping contraction with flopping curve $C$. Hence there is a model $\shY^+$ with properties as required.
\end{proof}

\begin{remark}
We emphasize that in Corollaries \ref{typeIflopexist} and \ref{typeIIflopexist}, the exceptional locus is precisely the curve $C$.
\end{remark}

If $C$ is a curve
 on an component $Y$ of the central fibre of a model $\shY\to S$, we say that $C$ -- considered as a curve on  $\shY$ -- is a $(-1)$-curve if the components of  the preimage $C^\nu$ under the normalisation $Y^\nu\to Y$ are $(-1)$-curves.

\begin{definition}
Let  $f\colon\shY\to\shZ$  be a flopping contraction with flopping curve $C$ an irreducible $(-1)$-curve. 
The flop $\shY\dashrightarrow \shY^+$  defined by  $f$ is  a \emph{ type I flop} if $C$ is not contained in the double locus $D$  of the central fibre $\shY_c$ and a \emph{ type II flop} if $C\subset D$.
\end{definition}

Note that type I and type II flops are simply global versions of type I and type II elementary modifications.

In the degree $2$ case, we  can give a refined version of Lemma \ref{exlocus} in case of flopping contractions on models with central fibre in  $\PMod_2$ .

\begin{proposition}\label{refinedversion}
Let $\shY\to S$ be a model of the DNV family of degree $2$ such that $\shY_c\in  \PMod_2$.
Let $\pi\colon\shY\to\shZ$ be a flopping contraction, given be the contraction of an extremal ray $R$. Then the exceptional locus $\Ex(\pi)$ is an irreducible $(-1)$-curve.  

\end{proposition}
\begin{proof}
The idea of the proof is the following: assume that  $\pi\colon\shY\to\shZ$ has reducible exceptional locus $\Ex(\pi)$. Then we shall show that one can find a contraction that contracts a proper subset of $\Ex(\pi)$. Hence the classes of the 
components of $\Ex(\pi)$ are not contained in an extremal ray $R$, giving a contradiction.

By Proposition \ref{irredcomp} all connected components 
of the exceptional locus of a flopping contraction are $(-1)$-curves. If a component  $C$ of $\Ex(\pi)$ is a $(-1)$-curve  in the double locus, then there is a flopping contraction with exceptional locus $C$, by Corollary \ref{typeIIflopexist}.
 
 So we can assume that all connected components  $C_i$ of  $\Ex(\pi)$  are given  by  interior $(-1)$-curves. As each component of the double locus $\Sing \shY_c$ meets at most one $C_i$, there are at most $3$ such components, $C_1,C_2,C_3$. Also, if $\shY_c\in  \PMod_2(\mathscr{T})$, by Proposition \ref{projTcurves},  none of the $C_i$ meets the smooth component of the restriction of  $\Sing \shY_c$ to $(\shY_c)_\omega$. Hence, by Corollary  \ref{typeIflopexist}, we only need to show the Proposition for  $\shY_c\in  \PMod_2(\mathscr{P})$.
 
   We show that in this case, $\Ex(\pi)$ is irreducible. Suppose there are two components $C_1, C_2$ in $\Ex(\pi)$. Write $\shY_c=Y_1\cup Y_2\cup Y_3$ and assume  $C_1\subset Y_1$.  Let $\xi_i=\sum_{j=1}^3 D_{ij}-D_{ji}$. This defines a divisor in $\Pic(\shY/S)$, by Lemma \ref{degree} and maximality of the model $\shY\to S$.
 Then $\xi_1.C_1=1$ and $\xi_1.C_2=-1$ if $C_2$ is not contained in $Y_1$. If $C_2$ is contained in $Y_1$, then $C_2.D_{1k}=1$ 
 for some $k\in\{2,3\}$. Then $\xi_k.C_1=0$ and $\xi_k.C_2=-1$. 
 Hence $C_1$  is not numerically equivalent to a positive multiple of $C_2$. This is a contradiction to $\pi=\contr_R$ and we can conclude that $\Ex(\pi)=C_1$.
 \end{proof}
 
 For later reference, we record the following immediate corollary.
 \begin{corollary}\label{classflop}
Let $\shY\to S$ be a model of the DNV family of degree $2$.
Let $\pi\colon\shY\to\shZ$ be a flopping contraction. Then the flop $\shY\dashrightarrow \shY^+$  defined by  $\pi$ is a  type I flop  or a  type II flop.
\end{corollary}

 \begin{remark} We can make Remark  \ref{dic:change} more precise:  if $\phi\colon\shY \dashrightarrow \shY^+$ is a flop, then the dual intersection complexes of $\shY$ and $\shY'$ are the same if $\phi$ is of type I and they are distinct if $\phi$ is of type II.
 \end{remark}

%%%%%%%%%%%%%%%%%%%%%%%%%%%%%%%%%%%%%

\begin{lemma}\label{onetypeIIflop} 
Let $\shY$ be a model of the DNV family of degree $2$ with $\shY_c\in  \PMod_2(\mathscr{P})$.
Assume that we have a birational map $\phi\colon \shY\dashrightarrow \shY'$ to another model $\shY'$.  Let $F'$ be an ample divisor on $\shY'$ with birational transform $F$ on  $\shY$. 
Consider a factorization of $\phi$ into $F$-flops:
\[\shY\dashrightarrow \shY_1\dashrightarrow\dots \shY_i\overset{\phi_i}{\dashrightarrow}\shY_{i+1}\dashrightarrow\dots\shY_n\xrightarrow{\sim} \shY'.
\]
Then at most one $\phi_i$ is of type II. In particular, if both $\shY$ and $\shY'$ are of class  $\mathscr{P}$, all $\phi_i$ are type I flops.
\end{lemma}
\begin{proof} Write $\shY_c=Y_1\cup Y_2\cup Y_3$. Without loss of generality we can assume the first flop in the factorisation to be of type II with exceptional locus $Y_1\cap Y_2$. Let $Y=Y_3$.  
 Let $Y'$  be the component of $\shY'_c$ which is the transform of  $Y$. 
 Then $Y'$ is a surface 
 such that the preimage  of the restriction of $\Sing \shY'_c$ 
 to $Y'$ under the normalisation morphism has $4$ components, $2$ of which, say $D_1,D_2$,  are $(-1)$-curves that are identified under the normalisation map, giving a smooth component $D_s$. The images of the curves in $\Gamma_Y$ together with the interior $(-1)$-curves meeting $ D_1$ and $D_2$ define a system of curves that is a basis of $\Pic((Y')^\nu)$, and as in Proposition \ref{projTcurves}, one shows that for all  $\phi_i$, $i\geq 2$,  $\Ex(\phi_i)$ 
cannot meet $D_s$, where we use the notation as in Remark \ref{rmkflops}. Hence $F_k.(D_s)_k>0$ for all $k\geq 2$. 
Thus any flopping curve in the sequence is an interior $(-1)$-curve, so 
all other flops are of type I. 
\end{proof}

%%%%%%%%%%%%%%%%%%%%%%%%%%%%%%%%%%%%%%%%%%%%%%%
\subsection{Birational Automorphisms}

We first fix some language. See Remark \ref{rmkflops} for notation.
\
\begin{definition}\label{flopofC} Let $\shY$ be  model of the DNV family of degree $2$  with $\shY_c\in  \PMod_2(\mathscr{P})$. 
Let $F\in\Mov(\shY/S)$. Let $C$ be a curve with $F.C<0$  that generates an extremal ray and let   $\phi_1\colon \shY\dashrightarrow \shY_1$ be the flop given by contracting $C$.  
Further, let  $\phi'=\prod_{i=2}^n \phi_i\colon\shY_1\dashrightarrow\shY'$ be a sequence of $(\phi_1)_*F$-flops and assume all flops $\phi_i$ in the sequence $\phi'$ are of type I. Set $\phi=\phi'\circ\phi_1$.
\begin{itemize}
\item[(i)]Let $C^+$ be the flopped curve of $\phi_1$. Any curve $C'$ that is the birational transform of $C^+$ under a subsequence of flops from $\phi'$ is called a \emph{$\phi$-flop of $C$}.
\item[(ii)] Let $C'$ be a $\phi$-flop of $C$ that is itself  a flopping curve. Any curve birational to the flopped curve $(C')^+$ under a map $\prod_{i=k}^m \phi_i $, $k,m\leq n$, 
is also called a $\phi$-flop  of $C$.
\end{itemize}
\end{definition}

\begin{remark}If $\phi$ is the sequence   \[\shY\dashrightarrow \shY_1\dashrightarrow\dots \shY_i\overset{\phi_i}{\dashrightarrow}\shY_{i+1}\dashrightarrow\dots\shY_n\xrightarrow{\sim} \shY'\]
and  $E$ is a curve on some $\shY_k$ in the factorisation, we have a sequence  $\phi'$ of flops defined by the tail 
 \[\shY_k\dashrightarrow \shY_{k+1}\dashrightarrow\dots \shY_i\overset{\phi_i}{\dashrightarrow}\shY_{i+1}\dashrightarrow\dots\shY_n\xrightarrow{\sim} \shY'.\] In this situation, if $E'$ is a  $\phi'$-flop of $E$ we will 
 also call $E'$ a $\phi$-flop of $E$.
\end{remark}

\begin{definition}\label{def:aloneannex}
Let $\shY\to S$ be a model of the DNV family in degree $2$ with $\shY_c\in  \PMod_2(\mathscr{P})$. Let $Y\subset \shY_c$  be a component 
of the central fibre and let $D_1$ be a component of the anticanoncial cycle $D=D_1+D_2$  of $Y$.   Let $C$ be an interior $(-1)$-curve on  $Y$ meeting $D_1$ 
\begin{itemize}
\item[(i)]  $C$  is called \emph{alone} if  there is no $H$ in $\Gamma_Y$ with $H.D_1=H.D_2=1$  and also \[\{E\in \Gamma_{Y}\mid E.D_1=1 \text{ and }E^2=-1\}=\{C\}.\]  
\item[(ii)]  $C$ is  called an \emph{annex} if  $C$ is not alone  and meets a unique $v\in\Gamma_Y$ or is the unique interior $(-1)$-curve of $Y$ meeting $D_1$.
\item[(iii)] If $C$ is not alone, the curve $C'\in \Gamma_{Y}\backslash{C}$ meeting $D_1$ is the \emph{companion} of $C$.  
\end{itemize}
\end{definition}

Some illustrations of this definition can be found in Figure \ref{ex:alone}.

\begin{figure}\centering

\begin{tikzpicture}

      \begin{scope}[shift= {(0,0)}]      
      
          \node[draw,shape=circle][fill=black] (D1) at (0,0){};      
  \node[draw,shape=circle] (1) at (1,1){};  
    \node(n1) at (0,-0.5){{\small $D_1$}};  
  \node(n1) at (1.5,1){{\small $-1$}};  
\node[draw,shape=circle] (2) at (1,0){}; 
\node[draw,shape=circle] (3) at (2,0){}; 

\node(n1) at (1,-0.5){{\small $-1$}}; 
\node(n2) at (2,-0.5){{\small $-1$}};

 \node[draw,shape=circle][fill=black] (D2) at (3,0){} ;     

\draw[thick]  (1)--(2)--(3)
(1)--(D1)
   (D1)--(2)
   (3)--(D2);

     \end{scope}

  %middle
       \begin{scope}[shift= {(4.5,0)}]      
      
          \node[draw,shape=circle][fill=black] (D1) at (0,0){};      
            \node(n1) at (0,-0.5){{\small $D_1$}};  
  \node[draw,shape=circle] (1) at (1,1){};  
  \node(n1) at (1.5,1){{\small $-1$}};  
\node[draw,shape=circle] (2) at (1,0){}; 
\node[draw,shape=circle] (3) at (2,0){}; 
\node[draw,shape=circle] (4) at (3,0){}; 
\node(n1) at (1,-0.5){{\small $-1$}}; 
\node(n2) at (2,-0.5){{\small $-2$}}; 
\node(n3) at (3,-0.5){{\small $-1$}}; 

 \node[draw,shape=circle][fill=black] (D2) at (4,0){} ;     

\draw[thick]  (1)--(2)--(3)--(4)
(1)--(D1)
   (D1)--(2)
   (4)--(D2);

         \end{scope}

        \begin{scope}[shift= {(10,0)}]      
      
          \node[draw,shape=circle][fill=black] (D1) at (0,0){};     
            \node(n1) at (0,-0.5){{\small $D_1$}};   
  \node[draw,shape=circle] (1) at (1,1){};  
  \node(n1) at (1.5,1){{\small $-1$}};  
\node[draw,shape=circle] (2) at (1,0){}; 

\node(n1) at (1,-0.5){{\small $0$}}; 

 \node[draw,shape=circle][fill=black] (D2) at (2,0){} ;     

\draw[thick]  (1)--(2)
(1)--(D1)
   (D1)--(2)
   (2)--(D2);

         \end{scope}

\end{tikzpicture}
\caption{Some examples of (augmented) curve structures where none of the $(-1)$-curves meeting $D_1$ is alone. In all examples, the $( -1)$ curve $v$  meeting $D_1$ and exactly one curve in $\Gamma_Y\backslash\{v\}$ 
is an annex. Also, both curves meeting $D_1$ are companions (of each other). }
\label{ex:alone}
\end{figure}
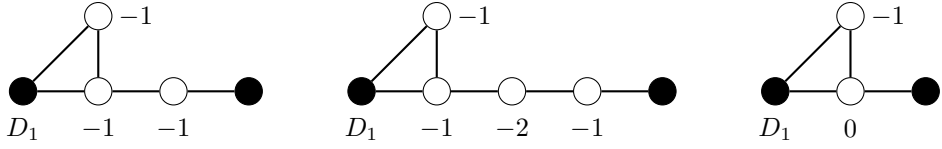

%%%%%%%%%%%%%%%%%%%%%%%%%%%%%%%%%%%%%%%%%%%%%%%%%%%%%%%%%%%%%%%%
\begin{definition}\label{def:completeetc}
Let $\phi$ be a sequence of $F$-flops \[\shY\dashrightarrow \shY_1\dashrightarrow\dots \shY_i\overset{\phi_i}{\dashrightarrow}\shY_{i+1}\dashrightarrow\dots\shY_n\xrightarrow{\sim} \shY'
\]  and $F$ be a divisor on $\shY$. Assume all flops $\phi_i$ are of type I. Let $I=\{0, \ldots , n\}$
be the index set, where we assign the index $0$ to $\shY$. Let $E$ be a curve on some $\shY_k$ that is the flopped curve of $\phi_k$.
\begin{itemize}
 \item[(i)]An \emph{$E$-sequence} is a
collection of indices $K^\phi_E\subset I$ such that for all $k\in K^\phi_E$, the flopping curve of $\phi_k\colon\shY_k\dashrightarrow\shY_{k+1}$ is either $E$ or a $\phi$-flop of  $E$. 

\item[(ii)] Set 
$m_1(K_E^\phi)=\min_{K_E^\phi} k $ 
and  $m_2(K_E^\phi)=\max_{K_E^\phi} k$. An $E$-sequence $K^\phi_E$ is \emph{complete}  if for any $i$ such that the flopping curve of $\phi_i$ is $E$ or a $\phi$-flop of $E$, $i\notin K_E^\phi$   implies $i\notin [m_1(K_E^\phi),m_2(K_E^\phi)]$.

\item[(iii) ]For any $k\in K_E^\phi$, let $E(k)$ denote the flopped curve of $\phi_k$.
A complete $E$-sequence  $K_E^\phi$ is \emph{directed} if either $(D_{ij})_k.E(k)=1$ implies $ i<j$ for all $k\in K_E^\phi$ or $(D_{ij})_k.E(k)=1$ implies $ i>j$ for all $k\in K_E^\phi$.

\item[(iv)]A directed $E$-sequence $K_E^\phi$ is called \emph{initial} if for any $i$ such that the flopping curve of $\phi_i$ is $E$ 
or a $\phi$-flop of $E$, $i\notin K_E^\phi$   implies $i>m_2(K_E^\phi)$.

\item[(v)]An initial $E$-sequence $K_E^\phi$ is called \emph{exhaustive} if for any $i$ such that the flopping curve of $\phi_i$ is $E$ or a $\phi$-flop of $E$, $i\in K_E^\phi$.
\end{itemize}
\end{definition}

%%%%%%%%%%%%%%%%%%%%%%%%%%%%%%%%%%%%%%%%%%%%%%%%%%%%%%%%%%%%%%
\begin{definition}Let $\shY$ be  a model of the DNV family of degree $2$ with $\shY_c\in  \PMod_2(\mathscr{P})$. Let $\shY_c=Y_1\cup Y_2\cup Y_3$ be  the central fibre. 
Let $F\in\Mov(\shY)$.
 Let $C$ be an interior $(-1)$-curve on $Y_1$ with $F.C<0$  that generates an extremal ray and let   $\phi_1\colon \shY\dashrightarrow \shY_1$ be the flop given by contracting $C$. Assume the numbering is such that  $D =D_{12}$ 
 is the component of the anticanonical divisor met by $C$.   Assume $D^2\neq -5$.
  Let $C^+$ be the flopped curve of $\phi_1$. 
  $C$ is \emph{good} for $D$  or \emph{$D$-good} 
  if for any sequence of $F$-flops $\phi=\phi'\circ\phi_1$ as in Definition \ref{flopofC}, the following holds:
 \begin{itemize}
 \item [(i)]  if $C'$ is the flopped curve of $\shY_k\dashrightarrow \shY_{k+1}$ and if
 there is an  initial $C$-sequence $K_C^\phi$ with $k\in K_C^\phi$, then $C'$ is alone.
  \item[(ii)]  Let $E$ be the curve in $\Gamma_{(Y_2)_1}$  
with $E^2<0$ and $E.C^+=1$.  Suppose there is a birational transform $E_r$ of $E$ that is the flopping curve of some $\shY_r\dashrightarrow\shY_{r+1}$. 
If  $E'$ is the flopped curve of $\shY_k\dashrightarrow \shY_{k+1}$ and there is an  initial $E_r$-sequence $K_E^\phi$, with $k\in K_{E_r}^\phi$, then $E'$ is alone.
 \end{itemize}
\end{definition}

The following example is a model for our future considerations.

\begin{example}\label{trailingcurveexample} We give an example of $D$-good curves and an exhaustive sequence.
Let $\shY$ be a model with central fibre $\shY_c\in  \PMod_2(\mathscr{P})$. 
Write $\shY_c=Y_1\cup Y_2\cup Y_3$ 
and suppose that $\Gamma_{Y_1}$ is degenerate, regular and has an exceptional vertex $v_0$. In particular, $v_0^2=-1$. Assume moreover that for the unique vertex $v_w\in \Gamma_{Y_1}\backslash\{L(v_0)\}$, $v_w^2=-1$.  
Let $v_n$ be the vertex on which $L(v_0)$ ends. We have $v_n^2=-1$. 
Let $D_{12}$ be the component of the anticanonical divisor met by $v_w$, see Figure \ref{examplecurves}.
\begin{figure}[]\centering
\begin{tikzpicture}[scale=1]
          \node[draw,shape=circle][fill=black] (D1) at (0,0){};      
            \node(n1) at (0,-0.5){{\small $D_{12}$}};  
  \node[draw,shape=circle] (1) at (1,1){};  
  \node(n1) at (1.5,1){{\small $v_w$}};  
\node[draw,shape=circle] (2) at (1,0){}; 
\node[draw,shape=circle] (3) at (2,0){}; 
\node[draw,shape=circle] (4) at (3,0){}; 
\node(n1) at (1,-0.5){{\small $v_n$}}; 
\node(n2) at (2,-0.5){{\small $-2$}}; 
\node(n3) at (3,-0.5){{\small $v_0$}}; 

 \node[draw,shape=circle][fill=black] (D2) at (4,0){} ;     
\node(n1) at (4,-0.5){{\small $D_{13}$}};  

\draw[thick]  (1)--(2)

(1)--(D1)
   (D1)--(2)
   (4)--(D2);

\draw[dashed](2)--(3)--(4);

\end{tikzpicture}
\caption{The curve structure in Example \ref{trailingcurveexample}}

\label{examplecurves}
\end{figure}
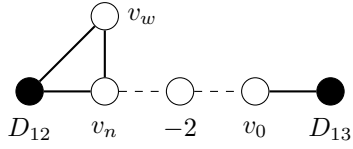

 We show that $C_{v_w}$ is $D_{12}$-good. Note that $D_{12}^2= 3$. Let $\phi$ be a sequence of  $F$-flops, such that $\shY\dashrightarrow \shY_1$ has flopping curve $C:=C_{v_w}$ and assume all flops $\phi_i$ in $\phi$ to be of type I.  Let $E$ be the unique curve of $\Gamma_{(Y_2)_1}$ met by $C^+$. We have  $(C_{v_n})_1^2=0$. 
 Assume  no birational transform  $E_r$ of $E$ is a flopping curve of 
 $\phi_r$. Then no birational transform of  $(C_{v_w})$ meets a 
 birational transform of $D_{23}$. 
 We first assume that there exists some $p\geq 2$ such that $\Ex(\phi_p)$ meets $(D_{12})_p$ or $(D_{21})_p$. Choose the minimal such $p$.
It then follows by minimality that   $\Ex(\phi_p).
 (D_{12})_p=1$.  But then $\Ex(\phi_p)=(C_{v_n})_p$, which is 
 impossible as $(C_{v_n})_1^2\geq0$.  Therefore $C_{v_w}$ is  $D_{12}$-good.

Now, assume  $E_t$ is the flopping curve of  $\phi_t$, where $t$ is chosen minimal. Necessarily, $\Ex(\phi_t).(D_{23})_t=1$.  Let $E^+$ be the flopped curve of $\phi_t$.  It is straightforward to calculate $(D_{32})_{t+1}^2=-12$: $(Y_2)_t$ is obtained from $\YP$ by blowing up one interior special point $4$ times and blowing down $10$ times in the other interior special point. Flopping $E_t$ corresponds to a further blow down, giving  $(D_{32})_{t+1}^2=-12$.
  So $E^+$ is the only curve in $\Gamma_{(Y_3)_{t+1}}$ meeting $(D_{32})_{t+1}$. In particular, $E^+$ is alone. 
 No birational transform of $E^+$ can be a flopping curve: such a curve $E^+_s$ would be $F_s$-negative, implying that there is a minimal index $k$ such that $E_k^+$ meets $(D_{23})_k$. But this implies that 
 all curve structures of the components of  the central fibre of $\shY_k$ are degenerate, a contradiction to  Proposition \ref{nondegproj2}.
 Hence $C$ is $D_{12}$-good and there is an exhaustive $C$-sequence.  
 
 Similar  arguments show that $C_{v_n}$ is $D_{12}$-good. 
 \end{example}

In order to make the statemens of the following lemmas lighter, we introduce the following setup.

\begin{setting}\label{setting}
Let $\shY$ and $\shY'$  be  models of the DNV family of degree $2$, with $\shY_c,\shY_c'\in  \PMod_2(\mathscr{P})$. Write $\shY_c=Y_1\cup Y_2 \cup Y_3$.
Let $\phi\colon\shY\dashrightarrow \shY'$ be a birational map  and let $C$ be an interior  $(-1)$-curve on the component $Y_1$ meeting 
$Y_1\cap Y_2$. Let $F'$ be  ample on $\shY'$ and denote its pullback under $\phi$ by $F$. Let \[\shY\dashrightarrow \shY_1\dashrightarrow\dots \shY_i\overset{\phi_i}{\dashrightarrow}\shY_{i+1}\dashrightarrow\dots\shY_n\xrightarrow{\sim} \shY'
\] be a factorisation into $F$-flops. 
\end{setting}
 
We assume this setting in Lemmas \ref{meetboundary}, \ref{mainrr},  \ref{neverback}, \ref{genmainrr},\ref{noreflection}, \ref{dgoodnorr}, \ref{floppingback}
and in  Corollary \ref{isalone}.

\begin{lemma}\label{meetboundary} 
Suppose that the curve  $C$ is not an annex and $F.C\geq 0$.  If there is a $p$ such that $\Ex(\phi_p)=C_p$  and $C_p.(D_{12})_p=1$, then there exists $q<p$ with $C_q.(D_{13})_q=1$.
\end{lemma}
\begin{proof}
Suppose there is a $p$ such that $\Ex{\phi_p}=C_p$ but $C_q.(D_{13})_q=0$ for all $q\leq p$. We have $F_p.C_p<0$, hence there must be some minimal $k<p$ such that $C_k$ meets the flopping curve $H$ or the flopped curve $H^+$ of $\phi_k$. As $C_q.(D_{13})_q=0$ for all $q<p$, in the first case, $H.(D_{12})_k=1$ and in the second case, $H^+.(D_{12})_{k+1}=1$. 
The  first case can only occur if 
 $C$ is not alone  and $H=C'_k$, $C' $ the annex meeting $C$. So in this case,  $C^2_{k+1}=0$. The curve $H$ is $D_{12}$-good by the argument in Example \ref{trailingcurveexample}. There, it is also shown that there is an exhaustive $H$-sequence, so it follows that  $C^2_{s}\geq0$ for $s\geq {k+1}$. Thus, $C_s$ is not contracted for any $s\geq {k+1}$, a contradiction to $\Ex{\phi_p}=C_p$.
 Consider the second case.
 We have   $F.H^+>0$ and by our assumptions, there is an $h$ such that $\Ex(\phi_h)=H^+_h$ and then arguing as before, there is $j<h$ such that $H^+_k$ meets the flopping curve $G$ or the flopped curve $(G)^+$ of $\phi_j$. 
  As $H^+$ is not an annex, as above, we conclude then that either $(H^+_i)^2=0$ for $i>k$ or  $G^+.H^+_k=1$ and $H^+_k$ is disjoint from the boundary.
 As there are only finitely many flops in the sequence, this procedure stops with  a curve $H'$ such that either there is an $l<p$ such that $F_i.H'_i>0$ and $H'_i.(D_{12})_i=1$ for $i>l$ or $(H'_i)^2=0$ and $H'_i.(D_{12})_i=1$ for $i>l$, a contradiction. 
\end{proof}

\begin{definition}\label{reflection}

Assume Setting \ref{setting}.
\begin{itemize}
\item[(i)]
Assume $\shY\dashrightarrow \shY_1$ has flopping curve $C$. Let $C^+\subset (Y_2)_1$ be the flopped curve and assume $C^+$ is not alone. Let $E$ be the  
$(-1)$-curve  meeting $(D_{21})_1=(Y_2)_1\cap (Y_1)_1$ with $E.C^+=1$.
Set $D_R:=D_{21}$.  Assume there is an index $q$ such that  $E_q$  meets $(D_R)_q$ and is the flopping curve of $\shY_q\dashrightarrow\shY_{q+1}$. This $q$ is unique as $E_q$
is $(D_R)_q$-good. In this case we say that $C$ is \emph{replaced} by $E$ and say $E$ is the \emph{replacement} of $C$.
\item[(ii)] Suppose $F.C\geq 0$.  
Assume $V:=\{i\in \NN\mid  C_i.(D_{13})_i=1\}\neq\varnothing$  
and let $p$ be its minimal element.   Let $E$ be the flopping curve of $\phi_{p-1}$.  Also, suppose that if $s$ is the minimal integer $s>p$ such that $C_s$ is the flopped curve of $\phi_s$, then $s\notin V$.  If there is such an $s$, we say that $C$ is \emph{reflected}. If $C$ is reflected, there is a minimal $q$,
 $s>q>p$ such that $q\notin V$. We set $D_R:=D_{23}$.  
 The flopped curve  $C^R$ of $\phi_q$ is the \emph{reflection} of $C$.  
\end{itemize}
 The curve $D_R$ is the \emph{$R$-locus} of the replacement or the reflection.   The index $q$ such that the flopped curve of $\phi_q$ is the replacement or the reflection of $C$ is called the \emph{index} of the 
 replacement resp. the reflection. 
  \end{definition}

\begin{remark}
If a $\phi$-flop  $C'$ of $C$ is replaced or reflected we will also call the replacement $E$ of $C'$ a replacement  or a reflection of $C$.
\end{remark}

\begin{lemma}\label{mainrr} 
 Assume that $C$ is alone and $F.C\geq 0$.  
Suppose $C$ is reflected, with reflection $C^R$. 
Suppose $p$ is the minimal element of $\{i\in \NN\mid  C_i.(D_{13})_i=1\}$, which exists by definition.  Let $E^+$ be the flopped curve of $\phi_{p-1}$.
Let $q$ be the index of the reflection. 
 If $E^+$ is not an annex, then  $E^+$ is not alone and $\Ex{\phi_q}$ is  a
 birational transform of the annex  $E^a$ meeting  $E^+$.
 \end{lemma}

\begin{proof}
 By definition of reflection,  there is a $b$ with $C_b.(D_{12})_b=1$ and $C_b$ the flopping curve of $\phi_b$.   
 There is a $p$, $p<b$ with $C_p.(D_{13})_p=1$. Assume $p$ minimal with this property.
 As in the statement, let $E^+\subset  (Y_3)_p$ be the flopped curve  of $\phi_{p-1}$.  By assumption,   $E^+$ is not an annex.
 Note that if $C_p$ is not alone,  it is necessarily not an annex. Also, the annex $C^a_p$ is  $D_{13}$-good and hence is not a flopped curve for $i\geq p$, as $C$ is reflected.
  
If $E^+$ is not alone, let $E^{a}$ be the annex and suppose 
$C^R$ is not a $\phi$-flop of $E^a$.  In any event ($ E^+$ alone or not)
as $C$ is reflected, there is a $b'$, $b'<b$ with $E^+_{b'}$ the flopping curve of $\phi_{b'}$ and $E^+_{b'}.(D_{31})_{b'}=1.$
By Lemma \ref{meetboundary}, there is $p'$, $p'<b'$, such that $E^+_{p'}.(D_{32})_{p'}=1$, which we again assume minimal. 
 
 Let $G^+\subset (Y_2)_{p'}$ be the flopped curve of $\phi_{p'-1}$. Suppose $G^+$ is alone. Again, it follows that there is  $b''$ with $p'<b''<b'$  such that  for the flopping curve $G^+_{b''}$ 
 of $\phi_{b''}$ one has $G^+_{b''}.(D_{23})_{b''}=1$. 
 By Lemma \ref{meetboundary}, there is $p''$, $p'<p''<b''$ such that $G^+_{p''}.(D_{21})_{p''}=1$, which we again assume minimal. So $\Gamma_{(Y_2)_{p''}}$ and $\Gamma_{(Y_3)_{p''}}$ are degenerate.  If $G^+$ is not alone, it follows that  $\Gamma_{(Y_2)_{p'}}$ and $\Gamma_{(Y_3)_{p'}}$ are degenerate.

In both cases, by projectivity, it follows from  Proposition \ref{nondegproj2} that  $\Gamma_{(Y_1)_{j}}$ for $j\in\{ p',p''\}$ is non-degenerate.  
By the assumption that $C$ is reflected, we have $C_s.(D_{12})_s=1$ for the  minimal $s>p$ such that $C_s$ is the flopped curve of $\phi_s$. Hence it follows from Lemma \ref{meetboundary} (applied to the exceptional vertex of $\Gamma_{(Y_1)_{j}}$ distinct from $C$)  that $(G^+)_i.(D_{21})_i=1$ for all $i\geq j$.  But then $C$ is not reflected, a contradiction. 
Hence we conclude that  $E^+$ is not alone
and  $C^R$ is a $\phi$-flop of $E^{a}$. Now, it follows by the argument in  Example \ref{trailingcurveexample} that $\Ex(\phi_q)$ is indeed a birational transform of $E^{a}$.
\end{proof}

\begin{lemma} \label{neverback}
  Suppose $C$ is the flopping curve of $\phi_0$.
If $C$ is  $D$-good, then there is an exhaustive $C$-sequence $K_C^\phi$.

\end{lemma}

\begin{proof} 
Let   $K_C^\phi$ be a maximal initial $C$-sequence. To get a 
contradiction, suppose $K_C^\phi$ is not exhaustive. We can 
assume that $K_C^\phi=\{0\}$.
Let $C^+$ be the flopped curve of $\phi_0$. Being a $\phi$-flop of $C$,  $C^+$ is alone and $F_1.C^+>0$.
As $K_C^\phi$ is not exhaustive, by definition, $C^+ $ is reflected. 

 So by assumption, there is some $b$ with $C^+_b.(D_{12})_b=1$ and $C_b$ the flopping curve of $\phi_b$.   
 There is a $p$, $p<b$ with $C_p.(D_{13})_p=1$. Assume $p$ minimal with this property.
 Let $E^+\subset  (Y_3)_p$ be the flopped curve  of $\phi_{p-1}$.  As $C$ is $D_{12}$-good, $E^+$ is alone. In particular, $E^+$ is not an annex.  Being alone, $E^+$ does not have a companion. From Lemma \ref{mainrr} we conclude that $C^+$ is not reflected and hence $K_C^\phi$ is exhaustive. 
\end{proof}

\begin{corollary}\label{isalone}
Suppose $F.C\geq 0$ and 
assume that there is an integer $p$ such that $\Ex{\phi_p}=C_p$ and $C_p.(D_{12})_p=1$.
Then $C$ is alone.  
\end{corollary}
\begin{proof}
Suppose $C$ is not alone. As $F.C\geq 0$, there must be some minimal $k$, $k<p$, such that $C_k$ meets the flopping curve $H$ or the flopped curve $H^+$ of $\phi_k$ .  Suppose first that $C$ is an annex.
 Let $C'$ be the companion of $C$.
 Suppose we are in the first case, i.e. $C_k$ meets the flopping curve $H$. Then $H=C'_k$ as $C_k.D_{12}=1$.
  By the argument in Example \ref{trailingcurveexample}, $C'$ is $D_{12}$-good, and hence from Lemma \ref{neverback}, there is an  exhaustive $C'$-sequence, and hence $C_i^2\geq 0$ for all $i>k$  contradicting our assumptions, as in particular $C_p^2<0$.

 Hence we are in the second case, i.e.  $C_k$ meets the flopped curve $H^+$ of $\phi_k$. Then $F_{k+1}.H^+>0$, $H^+$ is not an annex and  there is an $h$ such that $\Ex(\phi_h)=H^+_h$, with $H^+_h.(D_{12})_h=1$.    
 By Lemma \ref{meetboundary}, we have a minimal  $l$, $l<h$ with $H^+_l.(D_{13})_l=1$. 
 Note that the flopping curve of $\phi_{l-1}$ is  $C'_{l-1}$ which is necessarily $D_{13}$-good and there is an  exhaustive $C'_{l-1}$-sequence $\{l-1,\dots, v\}$, by Lemma \ref{neverback}, implying that $C_{i}$ is disjoint from  $(D_{12})_i$ for $i>l$, a contradiction.
Now suppose $C$ is not alone and not the annex. Let $E$ be the curve in $\Gamma_{Y_1}$ with $E.C=1$ and $E.D_{12}=0$. By a straightforward extension of the argument in the proof of Example \ref{trailingcurveexample}, one checks that $E$  and $C$ are $D_{13}$-good. 
By Lemma \ref{meetboundary}, $C_q.(D_{13})_q=1$.  Hence for some $r<q$, $E_r=\Ex(\phi_r)$ and $E_r.(D_{13})_r=1$.   By Lemma \ref{neverback}, there is an exhaustive  $E$-sequence $K^\phi_E$. As $C$ is $D_{13}$-good, 
our assumptions imply  $K^\phi_E=\{r\}$. Then $C_i.(D_{13})_i=1$ 
for all $i>r$, a contradiction.
\end{proof}

\begin{lemma}\label{genmainrr} 
 Suppose $F.C\geq 0$ and 
assume $C$ is reflected, with reflection $C^R$. Let $E^+$ be the flopped curve of $\phi_{p-1}$, with $p$ the minimal element of $\{i\in \NN:\mid  C_i.(D_{13})_i=1\}$. Let $q$ be the index of the reflection. 
 \begin{itemize}
 \item[(i)] If $E^+$ is not an annex, then  $E^+$ is not alone and $\Ex{\phi_q}$ is  a birational transform of the annex  $E^a$ meeting  $E^+$.
 \item[(ii)] If $E^+$ is an  annex, then  $\Ex{\phi_q}$ is  a birational transform of the companion of $E^+$.
 \end{itemize}
\end{lemma}
\begin{proof}It follows from Corollary \ref{isalone} that $C$ is alone. Hence
item $\rm{(i)}$ is nothing but Lemma \ref{mainrr}. If $E^+$ is an annex, it follows from Corollay \ref{isalone}
 that $\Ex{\phi_q}$ cannot be  a birational transform of $E^+$.
 Hence it is necessarily a  birational transform of the companion of $E^+$.
\end{proof}

\begin{remark} 
In particular, it follows that for the index  $q$ of a replacement or reflection
exactly one of $\Gamma_{(Y_R)_q}$ or $\Gamma_{(Y_R)_{q+1}}$ is a  regular curve structure, where $Y_R$ is the component containing $D_R$.
\end{remark}

\begin{lemma}\label{noreflection}
Assume that $C$ is not alone, and let $W$ be the  companion of $C$. 
\begin{itemize}
\item[(i)]  Assume $\shY\dashrightarrow \shY_1$ has flopping curve $C$.
Suppose $C$ and none of its $\phi$-flops is  replaced or reflected. Then there is an exhaustive $C$-sequence $K_C^\phi$.
\item[(ii)] Suppose $C$ is not an annex. Suppose there is no index $i$ sucht that  
$\Ex(\phi_i).(D_{12})_i=1$ and  $\Ex(\phi_i)$ a  $\phi$-flop of $C$  or $W$ or of a replacement of $C$. 
Then
 $D_{12}^2\geq(D_{12})_i^2$ for all $i>0$.
\end{itemize}
\end{lemma}
\begin{proof}The first part is immediate: take a maximal initial  $C$-sequence $K_C^\phi=\{1,\dots, p\}$. We can assume $K_C^\phi=\{1\}.$ If the flopped curve  $C^+$ of $\phi_p$ is an annex, we are done by Corollary \ref{isalone}, as in particular an annex is not alone. Else, we conclude from Lemma \ref{meetboundary} and  the assumption that there is no reflection.

For the second part, to obtain a contradiction, suppose the statement on intersection numbers is not true.
 Let $\shF(D_{12})$ be the collection of indices in $\{1,\dots , n\}$
such that $i\in \shF(D_{12})$ implies  $\Ex(\phi_i).(D_{12})_i=1$ or  
$\Ex(\phi_i).(D_{21})_i=1$.  The assumption that the statement on intersection numbers is not true implies that there is a  $p\in\shF(D_{12})$  with $\Ex(\phi_p).(D_{12})_p=1$.  Assume that $p$ is minimal with this property. Then there 
is $q\in\shF(D_{12})
$  with $q<p$ such that $\Ex(\phi_q).(D_{21})_q=1$, 
as else 
$\Ex(\phi_p)=C_p$  (or $W_p$), contrary to our assumptions.  
Suppose $q$ is chosen maximal. Let $H\subset (Y_1)_{q+1}$ be the flopped curve of $
\phi_q$. In particular, $F_{q+1}.H>0$. Also, $H$ is not an annex. 

We 
have $\Ex(\phi_p)=H_p$, as the remaining possibility is that $H$ is 
not alone and $\Ex(\phi_p)=H^ {a}_p$ with $H^{a}$ the annex 
meeting $H$. In that case, by the argument in Example \ref{trailingcurveexample}, $H^{a}_p$ is 
$D_{12}$-good, so $(H_i)^2\geq 0$ for all $i>p$ and $D_{12}
^2\geq(D_{12})^2_i$ for all $i>p$. By definition of $p$ we get $D_{12}
^2\geq(D_{12})^2_i$ for all $i>0$. 

Hence, by item $\rm(i)$, $H$ is reflected. Thus, there is $r$ with $r<p$ such that $H_r.(D_{13})_r=1$. Also, by constrution, $H_r.(D_{12})_r=1$. Hence  $\Gamma_{(Y_3)_r}$ and $
\Gamma_{(Y_1)_r}$ are degenerate. So, by projectivity, it follows from Proposition \ref{nondegproj2} that  $\Gamma_{(Y_2)_r}$ is non-degenerate. 
This implies that $\Ex(\phi_q)$ is $D_{21}$-good, so there is an exhaustive $\Ex(\phi_q)$-sequence by Lemma \ref{neverback},
contradicting $\Ex(\phi_p)=H_p$.
\end{proof}

\begin{lemma}\label{dgoodnorr}
Assume $\shY\dashrightarrow \shY_1$ has flopping curve $C$. 

\begin{itemize}
\item[(i)] If $C$ is $D_{12}$-good, then $C$ and none of its $\phi$-flops is replaced or reflected.
\item[(ii)] If $C$ is an annex, then $\Gamma_{(Y_1)_i}$ is not regular for any $i>0$.
\item[(iii)] If $C$ is not alone and not an annex,  then  $\Gamma_{(Y_1)_i}$ is  regular for any $i>0$.
\end{itemize}
\end{lemma}
\begin{proof}
Item $\rm{(i)}$ is immediate: by definition of $D$-goodness, it follows that $C$ or its $\phi$-flops are not replaced. They are not reflected as there is an exhaustive $C$-sequence.
For the remaining items, note  that in both cases $C$ is $D_{12}$-good, implying the claim.
\end{proof}

We are now able to prove a statement which is crucial for us (and which is specific to the degree $2$ case).

\begin{corollary}\label{cor:seqtypeI}
Let $\shY\to S$ be a model of the DNV family of 
degree $2$, of class $\mathscr{G}$ for    $
\mathscr{G} \in\{\mathscr{P}, \mathscr{T}\}$. 
Then there is a sequence of type I flops 
\[  \shY_{\mathscr{G}} \dashrightarrow \dots \dashrightarrow \shY.\]
\end{corollary}
\begin{proof}
We start with a birational map $\phi\colon \YYP\dashrightarrow \shY$ be a birational map.   Let 
\[\YYP\dashrightarrow \shY_1\dashrightarrow \shY_2\dashrightarrow\dots \shY_i\overset{\phi_i}{\dashrightarrow}\shY_{i+1}\dashrightarrow\dots\shY_n\xrightarrow{\sim} \shY \] 
be a factorisation of $\phi$ into flops. By Proposition \ref{irredcomp}, each flop $\phi_i$ is  either a type II flop or has exceptional locus given by disjoint interior $(-1)$-curves. Hence, 
if the set $\{ i\in\NN \mid \phi_i \text{ is of type II } \}$ is empty, we are done. 
Otherwise it follows from Lemma \ref{onetypeIIflop} that this set consists of a single element $\{p\}$.
Let $D_{ij}=Y_i\cap Y_j$ 
be a  double curve with $(D_{ij})_p$ contained in  the exceptional locus of $\phi_p$. Let $F$ be the bundle defining the factorisation. If $C$  is an interior  $(-1)$-curve on $Y_i$ meeting $D_{ij}$, then $C$ is $D_{ij}$-good by definition.
It thus follows from Lemma \ref{neverback} that $D_{ij}$ is disjoint from the flopped curve of $\phi_i$ for $i<p$. The same reasoning applies to $D_{ji}$.  Hence $F.D_{ij}=F.D_{ji}<0$.
So we can assume $p=0$.  
 Then, all that remains to show is that the model $\shY_1$  obtained from $\YYP$ by a flop defined by  a flopping contraction with exceptional locus a curve in the singular locus of $\YP$  can be obtained from $\YYT$ via  type I flops. This is immediate: $\shY_1$ has two components that are weak  del Pezzo surfaces of degree $3$. For each such component, there is a tree of curves of length $2$, i.e. a $(-1)$-curve and a $(-2)$-curve, meeting the interior special point. Flopping these trees
 gives the desired result.
\end{proof}
We immediately obtain that the set of all models of the Dolgachev-Nikulin-Voisin family of degree $2$ is indeed bijective to $ \PMod_2$:

\begin{corollary}\label{cor:typeItypeII}
Let $\shY\to S$ be a model of the $DNV$ family of degree $2$. Then $\shY_c\in  \PMod_2$.
\end{corollary}

In particular,  all models of the DNV family can be linked using only type I or type II flops. In particular all flops between
 \emph{any} two models of the DNV family have irreducible exceptional loci.

\begin{corollary}\label{factor} 
Let $\shY\to S$ and $\shY'\to S$ be models of the DNV family of degree $2$. Any birational $S$-map $\shY\dashrightarrow \shY'$  factors into a sequence of type $I$ and type $II$ flops. 
\end{corollary}

\begin{lemma}\label{floppingback}
Let $K^\phi_C$ be a  maximal initial $C$-sequence that is not exhaustive.  Let $p\notin K_C^\phi$ be the minimal index such that $C'=\Ex(\phi_p)$ is a $\phi$-flop of $C$. Let $\phi'$ be the composition
\[\shY_p\dashrightarrow \shY_1\dashrightarrow\dots \shY_i\overset{\phi_i}{\dashrightarrow}\shY_{i+1}\dashrightarrow\dots\shY_n %
\] and let $K_{C'}^{\phi'}$ be a maximal initial $C'$-sequence.  If $\Ex(\phi_i).(D_{21})_i=1$ for some $i\in K_{C'}^{\phi'}$, then $|K^\phi_C|=1$.
\end{lemma}

\begin{proof}By assumption, $C$ is reflected. Let $q$ be the index, $D_R$ the $R$-locus of the reflection.  Suppose the flopping curve of $\phi_q$ is not an annex. Then by the description in Lemma \ref{genmainrr} there is a unique $r$, $r<q$, with   $\Ex(\phi_r).D_{23}=1$ and no $i$, $i<q$  with with   $Ex(\phi_i).D_{32}=1$.  Hence the curve structure $\Gamma_{Y_2}$ is either regular with $|\Gamma_{Y_2}|=2$ or not regular with $|\Gamma_{Y_2}|=1$ and $\Gamma_{Y_3}$ is not regular. Thus $|K^\phi_C|=1$.

By Lemma \ref{genmainrr}, the remaining possibility is that  the flopping curve of $\phi_q$ is an annex. Suppose $D_R\subset Y_R$.
Then $\Gamma_{(Y_R)_{q+1}}$ is not regular. Let $E\subset Y'$ be the reflection of $C$.  
Let $D'$ 
be the component of the double locus on $Y'$ met by $E$. Then $E_i.D'_i=1$ for all $i>q$. Let $c$ be the maximal element of $K^\phi_C$ and let   $C'$ be the flopped curve of $\phi_c$.  By definition  $C'_{q+1}.E=1$.  Also, if $C'_t$  meets the boundary for some $t\geq c$, it is alone. It follows that   $K_{C'}^{\phi'}$ is a singleton set, and   as $\Ex(\phi_i).(D_{21})_i=1$ for some $i\in K_{C'}^{\phi'}$, the same holds for $K^\phi_C$.
\end{proof}

After these preliminaries, we can proof the main result of this section.

\begin{proposition}\label{negativity}
Let $\shY\to S$ and  $\shY'\to S$ be  models of the DNV family of degree $2$, with $\shY_c,\shY_c'\in  \PMod_2(\mathscr{P})$. Let $\psi\colon \shY\dashrightarrow \shY'$ and $\phi\colon\shY\dashrightarrow \shY'$ be two birational maps.
Write $\shY_c=\cup Y_i$ and $\shY'_c=\cup Y'_i$. 
Suppose 
\begin{equation}\label{cond} \psi(Y_i)\subset Y'_j \Leftrightarrow \phi(Y_i)\subset Y'_j.\end{equation}
Let $A'$ be an ample divisor on $\shY'$.  Let $A_\phi$ and $A_\psi$ be the birational transforms of $A'$  under $\phi$ and $\psi$. 
Let $C$ be a  $(-1)$-curve on $\shY$ that generates an extremal ray. Then 
\[ A_\phi.C<0 \Leftrightarrow A_\psi.C<0.\]
\end{proposition}
\begin{proof} Assume $A_\phi.C < 0$. Let  \[\shY\dashrightarrow \shY^\phi_1\dashrightarrow\dots \shY^\phi_i\overset{\phi_i}{\dashrightarrow}\shY_{i+1}^\phi\dashrightarrow\dots\shY^\phi_n\xrightarrow{\sim} \shY'
\] and  \[\shY\dashrightarrow \shY^\psi_1\dashrightarrow\dots \shY^\psi_i\overset{\psi_i}{\dashrightarrow}\shY^\psi_{i+1}\dashrightarrow\dots\shY^\psi_m\xrightarrow{\sim} \shY'
\]  be factorisations of $\phi$ and $\psi$, such that $C$ is the flopping curve of $\shY\dashrightarrow \shY^\phi_1$. 
  We show $A_\psi.C < 0$. 
By Proposition \ref{refinedversion}, $C$ is a $(-1)$-curve on a component, say $Y_1$, meeting $D_{12}$. Let $C^+$ be the flopped curve.  By Lemma \ref{onetypeIIflop}, 
any flopping contraction in the factorisation of $\phi$  contracts an interior $(-1)$-curve, so $C$ is not contained in $(\Sing \shY_c)\cap Y_1$. 

Suppose $C_k$ is not flopped by any $\shY_k^\psi\dashrightarrow \shY^\psi_{k+1}$ for any index $k$, with $(D_{12})_k.C_k=1$ and that the same is true for any $\psi$-flop of $C$ or any replacement. Suppose first that $C$ is an annex. Then $C$ is $D_{12}$-good and $\Gamma_{(Y_1)^\phi_1}$ is not regular. 
By Lemma \ref{dgoodnorr}, $\Gamma_{(Y_1)_n^\phi}$ is not regular. By hypothesis, $\Gamma_{(Y)^\psi_m}$ is regular.  But from Condition (\ref{cond}), we obtain $(Y_1)^\psi_m\xrightarrow{\sim}(Y_1)^\phi_n$, a contradiction.
Hence, $C$ is not an annex. Thus we can apply Lemma \ref{noreflection} and the equality $(\phi_*D_{12})^2=(\psi_*D_{12})^2$, which follows from Condition (\ref{cond}), to conclude that  either 
\begin{itemize}
\item[(i)]$C$ is not alone and a birational transform $C^{a}_k$  of the annex $C^{a}$ is the flopping curve of some $\shY_k^\psi\dashrightarrow \shY^\psi_{k+1}$, or
\item[(ii) ] $C$ is not alone and there is a flop $C_f$ of  $C$  that is  replaced or reflected under $\phi$ or
\item[(iii)] $C$ is alone and there is a flop $C_f$ of  $C$  that is  replaced or reflected under $\phi$.
\end{itemize}

 In cases $\rm{(i)}$ and $\rm{(ii)}$, $C$  is 
  $D_{12}$-good.  
So by Lemma \ref{neverback}, the case $\rm{(ii)}
$ is impossible. In case $\rm{(i)}$, note that 
$C^{a}$ is $D_{21}$-good, and by the arguments in 
the proof of Example \ref{trailingcurveexample}, any birational transform of $C^{a}$ meeting $(D_{2j})_k$ will be $D_{2j}$-good for $j=1,3$. By the same argument as above, we obtain a  contradiction to Conditon(\ref{cond}) using 
 Lemma \ref{dgoodnorr}. 
Hence we are in case $\rm{(iii)}$, namely  $C$ is alone and there  is a flop $C_f$ of  $C$  that is  replaced or reflected under $\phi$.

Let $D_R$
be the  $R$-locus and let $ Y_R$ be the component containing it. Let $q$ be the index of the replacement or reflection.  Let $C^R$ be the replacement (or reflection) of $C_f$.  The curve $C^R$ is $(D_R)_q$-good. This implies that  exactly one of $\Gamma_{(Y_R)_n^\phi}$ and $\Gamma_{(Y_R)_n^\psi}$ is regular, again a contradiction to (\ref{cond}). Indeed, by Lemma \ref{floppingback}, the $R$-locus is given by $D_{32}$. The implied curve structures are depicted in  Figure \ref{curvessetup}.  It then follows that in order  to change the regularity of $Y_R$ under a flop $\psi_i$, there has to be some $k$ such that $\psi_{k}$ has flopped curve $C_k$ or a $\psi$-flop of $C_k$, which we assumed is not the case.  So this case also leads to a contradiction. 
\begin{figure}[]\centering
\begin{tikzpicture}[scale=1.3]
\node[fill=black,draw,shape=circle] (11) at (0.5,0){};
\node[fill=black,draw,shape=circle] (1) at (0,0){};
\node[draw,shape=circle] (2) at (1.5,0){};
\node[draw,shape=circle] (3) at (-1,0){};
  \node[draw,shape=circle] (4) at (-1,1){};
  \node[draw,shape=circle,fill=black](5) at (-2,0){};
  \node[draw,shape=circle,fill=black](51) at (-2.5,0){};
   \node[draw,shape=circle] (6) at (-3.5,0){};

\draw (11)--(2)
(3)--(4)
(1)--(3)
(1)--(4)
(3)--(5)
(51)--(6);
  \draw[double, double distance between line centers=0.2em](5)--(51);
  \draw[double, double distance between line centers=0.2em](1)--(11);

  \node(A) at (1.5,-0.3){{\small $C$}};
  \node(B) at (-2.5,-0.3){{\small${D_R}$}};
\end{tikzpicture}
\caption{A subgraph of the augmented curve structures of $Y_1$, $Y_2$ and $Y_3$. The black vertices correspond to curves in the double locus of $\shY_c$. Here, the double lines between the black vertices indicates that the underlying curves are identified under the normalisation map. }

\label{curvessetup}
\end{figure}
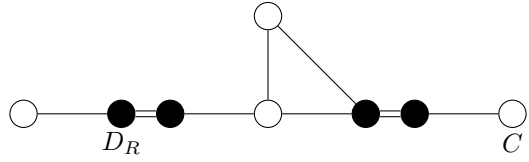

We conclude that there is an integer $p$ such that 
$\Ex(\phi_p)$ is a $\psi$-flop of $C$ 
with $(D_{12})_p.\Ex(\phi_p)=1$.
We show that $p=0$. To get a contradiction, suppose $A_\psi.C\geq 0$. In particular, $C$ is not an annex by Corollary \ref{isalone}.
Suppose $p$ is chosen minimal.  Lemma \ref{meetboundary} implies that $C$ is replaced or reflected. 
It also follows that $C$ is $D_{12}$-good. Suppose there is no $i$ with $\Ex(\psi_i)=C_i$, $i<p$. Then the $R$-locus is $D_{31}$.
Let $q$ be the index of the reflection/replacement. As before, we find that   $\Gamma_{(Y_3)_{q}^\psi}$ and $\Gamma_{(Y_3)_{q+1}^\psi}$ are not both regular or not regular. By 
Lemma \ref{dgoodnorr}, $\Gamma_{(Y_3)_m^\psi}$ has the same regularity as $\Gamma_{(Y_3)_{q+1}^\psi}$. Note that there is a curve $E\in \Gamma_{Y_1}$  meeting $C$ and not meeting $D_{12}$ such that $E_l=\Ex(\psi_l)$ where $l+1$ is the minimal index such that $C_{l+1}.(D_{13})=1$. 
By $D_{12}$-goodness of $C$, there is no $i$ such that $E_i=\Ex(\phi_i)$ and $E_i.(D_{13})_i=1$.
It follows that $\Gamma_{(Y_3)^\phi_n}$ has opposite regularity to  $\Gamma_{(Y_3)^\psi_m}$. As above, this gives a contradiction. 
If there is $l<p$  with $\Ex(\psi_l)=C_i$, $i<p$, by Lemma \ref{floppingback}, 
the $R$-locus is $D_{23}$ and the same reasoning as before applies.
\end{proof}

\begin{corollary}\label{orderBir}Let $\shY\to S$ be a model of the DNV family of degree $2$, with $\shY_c=\cup_i Y_i$. Suppose  $\shY_c\in \PMod_2(\mathscr{P})$. Let  $\phi$ and $\psi$ be birational automorphisms of $\shY\to S$. Suppose \[ \psi(Y_i)\subset Y_j \Leftrightarrow \phi(Y_i)\subset Y_j.\]
Then there is a regular  automorphism $\gamma$ of $\shY$ such that \[  \psi=\gamma\circ \phi.\]
 
\end{corollary}
\begin{proof} We use the same notation as in the previous propositon. Let $A_\phi$ be negative on a curve $C_1$ generating an extremal ray. Then there is an $A_\phi$ flop $\phi_1\colon\shY\dashrightarrow \shY_1$ defined by $\contr_{C_1}$.  
This is also the  $A_\psi$-flop defined by $\contr_{C_1}$, see e.g. \cite[6.10]{KoMo}. 
Continuing this way, one obtains a sequence of flops factoring $\psi$ and $\phi$ up to
isomorphism, i.e. a birational map $\psi_n\colon\shY\dashrightarrow\shY_n$ and isomorphisms $\alpha\colon \shY_n\to \shY$ and $\beta\colon \shY_n\to \shY$ such that 
$ \psi=\alpha\circ \psi_n $ and $\phi=\beta\circ \psi_n$ and thus setting
$\gamma=\alpha\circ \beta^{-1}$
we have $ \psi=\gamma\circ \phi$. 
\end{proof}

\begin{corollary}\label{autocoro} Let $\shY\to S$ be a model of the DNV family of degree $2$, with $\shY_c=\cup_i Y_i$. Suppose  $\shY_c\in \PMod(\mathscr{P})$. Let  $\phi\in\Bir(\shY/S)$.  Suppose $\phi(Y_i)\subset Y_i$ for all $i$. Then $\phi\in\Aut(\shY/S)$. In particular, $\Bir(\YYP/S)=\Aut(\YYP/S)$.
\end{corollary}
\begin{proof}The first part  follows  by setting $\psi$ to be  identity. 
 The second part follows as any interior $(-1)$-curves  $C$ on a component of $\YP$ is $D$-good for the component $D_{ij}
 $ of the double locus met  by $C$. Indeed, let  $\prod_{i=1}^n\phi_i$ be a factorisation of $\phi$ into flops and suppose there is  an index $k$ such that $C_k=\Ex(\phi_k)$ for the birational transform of an interior $(-1)$ curve $C$ of a component of $\YP$.  Obviously, we can assume  $k=1$.  By $D$-goodness, it follows from Lemma \ref{noreflection} that $(D^2_{ji})_t\leq -2$ for all  $t\geq 1$. This contradicts $(D_{ij})_n^2=-1$.
\end{proof}

\begin{corollary}\label{s6}
The automorphism group $\Aut(\YYP/S)$ of $\YYP/S$ contains a subgroup $S_3(\YYP)$  that is isomorphic to the symmetric group $S_3$ and  acts faithfully by  permutations on the set of components of $\YP$.  
\end{corollary}
\begin{proof}
Let $C$ be an interior $(-1)$-curve on a component ${\YP}_{,1}$ of $\YP$, meeting ${\YP}_{,1}\cap{\YP}_{,2}$. It defines an elementary 
modification of type I, $\YP\dasharrow Y$. All components of $Y$ have non-degenerate curve structure, so there is a smoothing $\shY$ of $Y$ 
and  a type I flop $\phi\colon \YYP\dashrightarrow\shY$  given by a flopping contraction contracting precisely $C$, by Proposition \ref{typeflop}. 
Note that all components $Y_i$ of $Y$ have pairwise distinct curve structure. Assume that ${\YP}_{,1}$ is mapped to $Y_1$. In particular, they 
are not isomorphic. Let $C'$ be an interior $(-1)$  different from $C$, on say ${\YP}_{,k}$, meeting ${\YP}_{,k}\cap{\YP}_{,i}$ with $i\in\{1,2,3\}\backslash\{k\}$. Note that as $C'$  is different from $C$, $(k,i)\neq (1,2)$. 
 By the same argument as before, we obtain a type I flop $\phi_{(k,i)}\colon 
\YYP\dashrightarrow\shY'$  given by a flopping contraction contracting precisely $C'$. Let $Y'_k$ be the the component which is  the image 
of ${\YP}_{,k}$. Now, $Y'=\shY'_c$ and $Y$ are isomorphic and hence by uniqueness of the DNV family there is an isomorphism $
\gamma_{(k,i)}\colon\shY\to \shY'$. Because the curve structures are distinct, necessarily $\gamma(Y_1)=Y'_k$.  It follows that there is a map $\psi_{(k,i)}$ and a commutative diagram of 
birational maps
\[\xymatrix{\YYP\ar@{-->}[r]^\phi\ar@{-->}[d]^{\psi_{(k,i)}} & \shY\ar[d]^{\gamma_{(k,i)}} \\
\YYP\ar@{-->}[r]_{\phi_{(k,i)}}&\shY'.
}\] By Corollary \ref{autocoro}, $\psi_{(k,i)}$ is a morphism.
By looking at curve structures, we find that $\psi_{(k,i)}$ maps ${\YP}_{,1}$ to ${\YP}_{,k}$ and ${\YP}_{,2}$ to ${\YP}_{,i}$.  One finds that the possible combinations $(k,i)$ are $(1,3),(2,1),(2,3),(3,1)$, and $(3,2)$. Hence the  permutations on the set of components of $\YP$ that are induced by automorphisms  $\psi_{(k,i)}$  are $(1,3,2), (2,1,3),(2,3,1), (3,1,2)$ and $(3,2,1)$. Hence there is indeed a subgroup
as claimed.
\end{proof}

 \begin{definition}\label{def:symmetric}
A model $Y\in  \PMod_2$ is \emph{symmetric} if there are distinct components $Y_1,Y_3\subset Y$ and an automorphism $\psi\in\Aut(Y)$ such that
\[\psi_{|Y_1}\colon Y_1\xrightarrow{\sim} Y_3.\]
If there $Y=\shY_c$ for a model $\shY$ of the DNV family, then $\shY$ is \emph{symmetric}.
\end{definition}

Note that this implies that $D^2_{13}=-1$, as $\psi(D_{13})=D_{31}$. For models with symmetric central fibre, we have the following  statement.

\begin{proposition}\label{symmbirat} Let $\shY\to S$ be a model of the DNV family of degree $2$. Assume $\shY_c\in  \PMod_2(\mathscr{P})$ is symmetric, but  $\shY\neq\YYP$. Suppose $\phi, \psi$ are two birational automorphisms
of $\shY$. Write 
$\shY_c=Y_1\cup Y_2\cup Y_3$ with notation such that $D_{13}^2=-1$. If
$\phi (D_{13})^2=\psi(D_{13})^2$, then there is an automorphism  $\gamma\in \Aut(\shY/S)$ such that \[ \psi=\gamma\circ \phi.\]
\end{proposition}

\begin{proof}
 As $\shY\neq\YYP$ it follows from  $D_{13}^2=-1$ and  symmetry  that  $Y_1\xrightarrow{\sim} Y_3$. 
It is straightforward to see, using $D_{13}$-goodness of interior $(-1)$ curves meeting $D_{13}$ or $D_{31}$,
that $\phi(D_{13})^2=-1$ implies  that $\phi$ is regular, as then no curves can be flopped. Hence we can assume $\phi(D_{13})^2\neq-1$.  Also,  note that any interior $(-1)$-curve $C$ meeting $D_{13}$ is $D_{13}$-good. More generally, it is easy to see that if \[\shY\dashrightarrow \shY_1\dashrightarrow\dots \shY_i\overset{\phi_i}{\dashrightarrow}\shY_{i+1}\dashrightarrow\dots\shY_n\xrightarrow{\sim}\shY\] 
is a factorisation of a birational automorphism of $\shY$  and $C$ is the flopping curve of $\shY_k\dashrightarrow\shY_{k+1}$, with $C.(D_{13})_k=1$, then $C$ is $(D_{13})_k$-good.

We first show the proposition assuming  $\Gamma_{Y_1}$ is regular.  Note that then no flopping curve $C\subset (Y_i)_k$  in the factorisation can change the regularity  of $\Gamma_{(Y_i)_k}$, as such a change would be irreversible by Lemma \ref{dgoodnorr}.  This 
implies the proposition. Indeed,  there is neither replacement nor reflection as there are no regularity changing flops, so if $C$ is the flopping curve of some $\phi_r$, with  $C.(D_{ij})_r=1$, we conclude from Lemma \ref{noreflection} and the condition on intersection numbers that for some $p$, $C_p$ is the flopping curve
of some $\psi_p$, with $C_p.(D_{ij})_p=1$.
  As in the proof of Proposition \ref{negativity}, it follows that $p=0$. Arguing as in Corollary \ref{orderBir}, the claim follows. 
%%%%%%%%%%%%END FIRST PART OF PROOF %%%%%%%%%%%%%%%%%%%%%%%%%%%%%%
%%%%%%%%%%%%%%%%%%%%%%%%%%%%%%%%%%%%%%%%%%%%%%%%%%%%%%%%%%%%%%%%%%%%%%%

Now, suppose $\Gamma_{Y_1}$ is not regular.  Then $\shY_c$ is uniquely defined up to isomomorphism: we have $D^2_{12}=D^2_{32}=4$, $D^2_{13}=D^2_{31}=-1$, 
$D^2_{21}=D^2_{23}=-6$ and also $\Gamma_{Y_3}$ is not regular. Also, both $\Gamma_{Y_1}$ and $\Gamma_{Y_3}$ have three vertices while $\Gamma_{Y_2}$ has $18$ vertices.   
If $\phi$ is a birational automorphism such that  $\phi(D_{13})^2\neq -1$, then either 
$ \phi(D_{13})^2=4,  \phi(D_{12})^2=-1 $ and $\phi(D_{23})^2=4 $ or $\phi(D_{13})^2=-6,  \phi(D_{12})^2=-6 $ and $  \phi(D_{23})^2=-1 $.
We first show that if \[\shY\dashrightarrow \shY_1\dashrightarrow\dots \shY_i\overset{\phi_i}{\dashrightarrow}\shY_{i+1}\dashrightarrow\dots\shY_n\xrightarrow{\sim}\shY\] 
is a factorisation of $\phi$, and $C$ the flopping curve of some $\phi_k$, then there is an exhaustive $C$-sequence $K_C^\phi$. By symmetry, it is enough to show this for the case $\phi(D_{13})^2=4$.
So let $C$ be the flopping curve of some $\phi_k$. Then $( D_{ij})_k.C=1$ for some $D_{ij}$. If $D_{ij}=D_{13}$ or $D_{ij}=D_{31}$, then, as remarked above,  $C$ is  $( D_{ij})_k$-good. If $D_{ij}=D_{21}$, then it may be that the flopped curve  $C^+$ of $\phi_k$ is not alone. 
In that case however, both curves $C^+$ and $C'$  meeting $(D_{12})_{k+1}$ are $(D_{12})_{k+1}$-good by the same argument as in  Example \ref{trailingcurveexample},
 so from the condition on the intersection numbers, there is no $\phi_s$ with flopping curve a birational transform of $C^+$ or $C'$.   The same argument  applies
to $D_{23}$, impliying that $D_{ij}=D_{21}$ and $D_{ij}=D_{23}$ cannot happen. One concludes that there is an exhaustive $C$-sequence $K_C^\phi $.
%%%%%%%%%%%%%%%%%%%%%%%%%%%%%%%%%%%%%%%%%%%%%%%%%%%%%%%%%%%%%%%%%%%%%

To prove the proposition,  assume there  are  factorisations of $\phi$ and $\psi$ that  agree up to term $l$.  As in Proposition \ref{negativity} we  let $A_\phi, A_\psi$ denote the bundles defining the sequences of flops.
If $C$ is the flopping curve of  $\phi_l$, with  $C.(D_{ij})_l=1$, because of the existence of an exhaustive $C$ sequence $K^\phi_C$, the condition on intersection numbers implies that there is a flop $\psi_s$ with flopping curve $C^\psi_s$: if $C$ is alone, this is immediate. If $C$ is not alone, then flopping the companion of $C$ implies that $\Gamma_{(Y_i)^\psi_m}$ and $\Gamma_{(Y_i)^\phi_m}$ have distinct regularity, a contradiction. Hence indeed there is a $\psi_s$ with the claimed properties.  
We show $s=l$.

Suppose first that $C$ is alone.  Assume  $(A_\psi)_l.C_l\geq  0$.
We can deduce from  Corollary \ref{isalone} that $C$ is alone and from  Lemma \ref{meetboundary} we get a minimal $q$, $s>q>l$ such that $C_s^\phi.(D_{ik})_q^\phi=1$, $k\neq \{i, j\}$. Write $E=\Ex(\phi_k)$ for the flopping curve and $E^+$ for the flopped curve. If $E$ is not replaced with 
$R$-locus $D_{ki}$ we have $C_i.(D_{ik})_s^\phi=1$ for all $i\geq 
q$, contradicting $C=\Ex(\psi_s)$.    Hence $E^+$ is not alone and, letting  $E^c$ be the  companion of $E^+$, there is a $t$, $s>t>q$ such 
that $(E^c)_t=\Ex(\psi_t)$ and $E^c_t.(D_{ki})_t=1$. 

 If $E^+$ is an annex,  it follows that the curve structure $\Gamma_{(Y_i)_l}$ consists of only two vertices and there is no $(-1)$-curve meeting $(D_{ki})_m^\phi$ for any $m\geq s $. This implies that $\Gamma_{(Y_i)_m}$ is a singleton for $m\geq l+1$, a contradiction.
 If $E^+$ is not an annex, then $\Gamma_{(Y_k)^\phi_m}$ is regular for all $m\geq l$ while  $\Gamma_{(Y_k)^\psi_m}$ is not regular for $m > t$, a contradiction. Hence $(A_\psi)_s.C<0$ and it follows inductively that $\phi$  and $\psi$ agree up to automorphism.
\end{proof}

\subsection{Orbits}

Let $\shY\to S$ a be model of the DNV family of degree $2$. We recall the action of $\Bir(\shY/S)$ on the Mori fan $\Morifan(\shY/S)$ in somewhat more detail. Let $(\shY',f)$
 be a marked minimal model of the DNV family, i.e. a model $\shY'$ 
 of the DNV family  together with a birational map
 $f \colon\shY\dashrightarrow \shY'$. 
This determines a maximal cone $C(f)$ of $\Morifan(\shY/S)$, defined as the pullback under $f$ of the Nef cone of $\shY'$.

The group $\Bir(\shY/S)$ acts on the cones of $\Morifan(\shY/S)$.  Suppose $g\colon\shY\dashrightarrow \shZ$ is another marked minimal model with $\shZ$ isomorphic to $\shY'$ via $h\colon \shZ\to \shY'$.  We can replace $g$ by $h\circ g$ and assume $\shZ=\shY'$, as the corresponding cones are identical. Then $\gamma=f^{-1}\circ  g$ is a birational $S$-automorphism of $\shY$ mapping $C(f)$ to $C(g)$. Hence the orbit of $C(f)$ under the action of $\Bir(\shY/S)$ is parameterised by the set of  marked minimal models $(\shZ,g)$ with $\shZ\cong \shY'$. 
 If $f={\operatorname{id}}_{\shY/S}$ is the identity on $\shY$, this trivially defines a model 
 $(\shY,{\operatorname{id}}_{\shY/S})$ 
 and the corresponding cone is simply $\Nef (\shY/S)$.

\begin{definition}\label{def:asscone}
Let $\shY\to S$ be a model of the Dolgachev-Nikulin-Voisin family of degree $2d$. If $\sigma$ is  a maximal 
cone in the orbit of $\Nef (\shY/S)$ under $\Bir(\shY/S)$ we say that $\sigma$ is a cone \emph{associated} to $\shY\to S$. 
\end{definition}

\begin{proposition}\label{orbitmain} Let $\shY\to S$ be a model of the DNV family of degree $2$.  Let $\sigma\in\Morifan(\shY/S)$ 
be an  associated maximal cone and
$\Bir(\shY/S).\sigma$ be the orbit of $\sigma$ under $\Bir(\shY/S)$.
Then 
\[|\Bir(\shY/S).\sigma|=
\begin{cases}
1 & \text { if }\shY=\YYP, \\
3 & \text{ if } \shY_c \text{ is symmetric}, \shY\neq\YYP, \\
6 &\text{ else.}
\end{cases}
\]
\end{proposition}
\begin{proof} We first assume  $\shY_c\in  \PMod_2(\mathscr{P})$. If $\shY=\YYP$, the result follows from Corollary \ref{autocoro} together with Corollary \ref{s6}.
In general,
by Corollary
\ref{orderBir}, there are at most $5$ birational automorphism that are not regular  (up to composition with an automorphism). 
This follows since the fibre has $3$ components and hence there are only $5$ permutations of the components which are not the identity.
Let $\shY$ be a model different from $\YYP$ and let $\phi\colon\shY\dashrightarrow \YYP$ be a birational map. 
For each $\sigma\in S_3$, let $g_\sigma\in \Aut(\YYP)$ be an automorphism permuting the components of the central fibre as in Corollary \ref{s6}.
One obtains $6$ birational maps $\alpha_\sigma=\phi^{-1}\circ g_\sigma\circ\phi$.
We show that if $\shY_c$ is not symmetric, $5$ of the maps $\alpha_\sigma$ are  not regular and  distinct in the sense that for any pair $\alpha_\sigma,\alpha_{\sigma'}$  with $\sigma\neq \sigma'$, there does not exist an automorphism  $\beta$ of $\shY$ with $\alpha_\sigma=\beta\circ\alpha_{\sigma'}$.  Suppose  there is an $\alpha:=\alpha_\sigma$ that is regular and does not map each component of $\shY_c$ into itself. Regularity implies $\alpha$ is an isomorphism.  We obtain that $\shY_c$ is symmetric. Indeed, $\alpha$ induces an automorphism of $\shY_c$ and by assumption, there is a component $(\shY_c)_1$ mapped to a component $(\shY_c)_2$, so $\shY_c$ is symmetric. Hence none of the $\alpha_\sigma$ is a regular map.
 Also, all the $\alpha_\sigma$ are distinct: as $\shY_c$ is not symmetric, all automorphisms fix the components of $\shY_c$, so by construction, no two  $\alpha_\sigma$ can be related by an automorphism.

Now assume $\shY_c=Y_1\cup Y_2\cup Y_3$ is symmetric.  By symmetry,  there are two components, say $Y_1$ and $Y_3$, such that $Y_1\xrightarrow{\sim} Y_3$ and $D^2_{13}=D_{31}^2=-1$. 
By Proposition \ref{symmbirat}, up to automorphism, there are at most $2$ birational  automorphisms of $\shY$. Arguing as in the preceeding case, there are precisely $2$ birational automorphishms up to automorphism, 
say $\phi$ and  $\psi$,   where we define $\phi$ as the map such that $\phi(D_{13})^2=D_{21}^2$ and $\psi$ by requiring $\psi(D_{13})^2=D_{12}^2$. These do not agree up to automorphism as the existence of a $\gamma\in \Aut(\shY)$ with $\gamma\circ\psi=\phi$ implies $D^2_{12}=D^2_{21}=-1$ and thus $\shY=\YYP$.

%%%%%%%%%%%%%%%%%%%%%%%%%%%%%%%%%%%%%%%%%%%%%%%%%%%%%%%%%%%%
Now suppose  $\shY_c\in  \PMod_2(\mathscr{T})$. 
Suppose first that the model
is obtained by a single type II flop from $\YYP$. Then $\shY$ 
is symmetric and the orbit of the associated cone has length $3$.  
Second, if the type II flop  of $\shY$ yields a model $\shY^+$ not isomorphic to $\YYP$, still   $\shY^+_c\in \PMod(\mathscr{P})$. Note that 
models obtained from $\shY^+$ by applying different  single type II flops  to $\shY^+$ are non-isomorphic.  
Hence,  if $\sigma$  is an associated cone of $\shY$, there is a unique cone $\sigma^+$ associated to $\shY^+$ meeting $\sigma$ in codimension $1$ and vice versa. It follows that 
if $\sigma$ and $\sigma^+$ are associated cones of $\shY$ and $\shY^+$, the orbits have the same length, i.e. we have 
\[|\Bir(\shY/S).\sigma|=|\Bir(\shY/S).\sigma^+|.\] 
Also, $\shY$ is symmetric if and only if $\shY^+$ is symmetric, as is easily  checked by a direct calculation. 
\end{proof}

\begin{remark}\label{stabil}
Suppose $\gamma\in \Bir(\shY/S)$ fixes the cone $C(f)$ (as a cone), i.e. $\gamma(C(f))=C(f)$. Then $(\shY',f\circ \gamma)$ is a marked minimal model that is isomorphic to $(\shY',f)$. 
In particular, there is $\beta$ in $\Aut(\shY'/S)$ with $f\circ\gamma=\beta\circ f$, see e.g. \cite[Lemma 1.5]{Kaw97}.
Conversely, setting $\gamma=f^{-1}\circ\beta\circ f$ for any $\beta\in \Aut(\shY')$ defines an element of the stabilizer $\Bir(\shY/S)_{C(f)}$ of $C(f)$. 

Hence
\[ \gamma \mapsto f\circ\gamma\circ f^{-1} \] defines an isomomorphism 
\[ \Aut(\shY'/S) \to \Bir(\shY/S)_{C(f)}.\]
\end{remark}

We next construct automorphisms of symmetric models of the DNV family that permute the smooth components of the central fibre. 

\begin{proposition}\label{prop:syminvolution}
Let $\shY$ be a symmetric model of the DNV family of degree $2$. Suppose $\shY_c=Y_1\cup Y_2\cup Y_3$  and let $\phi\in \Aut(\shY_c)$ be an automorphism with $\phi(Y_2)=Y_3$. 
 Then there is an automorphism $\psi\in \Aut(\shY/S)$ of the total space  such that  
$\psi(Y_2)=Y_3$ and $\psi(Y_3)=Y_2$. 
\end{proposition}

\begin{proof}If $\shY\xrightarrow{\sim}\YYP$ the proposition follows from  Corollary \ref{s6}. 
Hence suppose that $\shY$ is not isomorphic to $\YYP$. 
We will first construct birational isomorphisms $\gamma_i$, $i=1,2,3$ of $\shY$ such that the orbit of $\Nef(\shY/S)$ in $\Morifan(\shY/S)$ is $\{C(\gamma_1), C(\gamma_2), C(\gamma_3) \}$. 
For this we fix a birational map $\phi\colon\shY\dashrightarrow \YYP$.  
Write $\YP=Y'_1\cup Y'_2 \cup Y'_3$   with indices chosen such that $\phi(Y_i)\subset Y'_i$.  Let $g_1$ be the identity on $\YYP$ and choose elements  $g_i \in \Aut(\YYP/S), i=2,3$ with $g_i(Y'_1)=Y'_i$. These exist 
by Corollary \ref{s6}. 
We then define the maps $\gamma_i$ via the commutative diagrams

\[\xymatrix{\shY\ar@{-->}[r]^\phi\ar@{-->}[d]^{\gamma_i} & \YYP\ar[d]^{g_i} \\
\shY\ar@{-->}[r]_\phi & \YYP
}\]
as $\gamma_i=\phi^{-1} \circ g_i \circ \phi$.
In particular, $\gamma_1$ is the identity on $\shY$. 
Now, note that if $\shY_c\in \PMod(\mathscr{T})$ then any automorphism of $\shY$  maps the special component to itself as it permutes the smooth components. With our assumptions, the special component is $Y_1$. 
Similarly, if $\shY_c\in \PMod(\mathscr{P})$, 
any 
automorphism of $\shY$ maps $Y_1$ to itself, as $D^2_{23}=-1$ and $D^2_{12}=D^2_{21}\neq -1$ by symmetry.
So in any event, since  $\gamma_i$ maps  $Y_1$ to $Y_i$ the maps $g_2$ and $g_3$ cannot be automorphisms of $\shY$.

We claim that the cones $C(\gamma_i)$ are pairwise distinct. Indeed, if $C(\gamma_i)=C(\gamma_j)$, there exists, by the above Remark \ref{stabil} an automorphism
 $h\in \Aut(\shY/S)$ with  $\gamma_i=h\circ \gamma_j$.  
 We get \[h\circ \gamma_j(Y_1)\subset Y_j \text{ and } \gamma_i(Y_1)\subset Y_i\] and hence $i=j$. 
 By Proposition \ref{orbitmain},  the orbit of $\Nef(\shY/S)$ in $\Morifan(\shY/S)$ consists of three cones and is thus the set $\{C(\gamma_1), C(\gamma_2), C(\gamma_3) \}$. We now construct the desired automorphism $\psi$. Again by  Corollary \ref{s6}  there is  an automorphism $g_0 \in \Aut(\YYP/S)$  
such that 
\[ g_0(Y'_1)=Y'_1,\quad g_0(Y'_2)=Y'_3 \text{ and } g_0(Y'_3)=Y'_2.\] 
Consider the map $\psi=\phi^{-1}\circ g_0 \circ \phi $.  Because $g_0$ fixes $Y'_1$ it follows that also $\psi$ fixes  $Y_1$. 
So we cannot have $C(\psi)=C(\gamma_i)$ for $i=2,3$ by the same reasoning as before. So necessarily $C(\psi)=C(\gamma_1)$.  It follows now from  Remark \ref{stabil} that $\psi\in \Aut(\shY/S)$ and by construction $\psi$
permutes the components of $\shY$ as claimed.
\end{proof}

\section{Counting models}\label{sec:counting}

In this section we will count the elements of  $\PMod_2$.  We will show that this can be done by counting triples of curve structures.

\subsection{Automorphisms }\label{automorph}
We will need certain automorphisms of components of surfaces in $\PMod_2$. We first calculate the automorphism group of the central fibre $\YP$ of the model $\YYP$. 
 
 Write $Y=\mathfrak{Y}_2$, where $\mathfrak{Y}_2$ is the weak del Pezzo surface of degree $2$ defined in Construction \ref{constructioncomponents}.
 Observe that any automorphism $\gamma$ of $Y$ fixes the set of interior special points $\{p_1,p_2\}$ as it fixes the set of $(-1)$-curves. 
Hence $\gamma$ lifts to an 
automorphism of the $(2,2)$-blow-up
$Y_{(2,2)}$ of $Y$ in $(p_1,p_2)$. Also, any automorphism of $Y_{(2,2)}$ fixes the set of exceptional curves of the blow-up $Y_{(2,2)}\to Y$, as these contain all  $(-1)$-curves, and thus descends to $Y$. 
Applying \cite[I.5.4]{Loo} and \cite[Remark 5.2]{GHK15} to $Y_{(2,2)}$, it follows that $\Aut(Y)$ is a subgroup of the dihedral group $\ZZ_2\times\ZZ_2$.
Indeed, we will show that $\Aut(Y)=\ZZ_2\times\ZZ_2$.

First, we give an alternative construction of $Y$. Let $Q=\PP^1\times \PP^1$. Let $\bar{D}=\bar{D}_1+\bar{D}_2+\bar{D}_3+\bar{D}_4$ be its toric boundary, ordered cyclically. Let $F_1$ be a fibre of the ruling with fibre $\bar{D}_1$ 
that meets $\bar{D}$ in smooth points of $\bar{D}$. Let $p_2$ be the point in $F_1\cap \bar{D}_2 $ and  $p_4$ be the point in $F_1\cap \bar{D}_4$. Similarly,  let $F_2$ be a fibre of the ruling with fibre $\bar{D}_2$ that meets $\bar{D}$ in smooth points of $\bar{D}$ with $p_1$ the point in $F_2\cap \bar{D}_1$ and  $p_3$ the point in $F_2\cap \bar{D}_3$, see Figure \ref{blowupauto}.
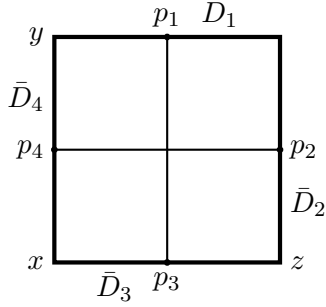
\begin{figure}[]\centering
\begin{tikzpicture}[scale=1]

\begin{scope}[shift={(2,0)}]
%right rectangle%%%%%%%%%%%%%%%%
\draw[ultra thick] (5,0)--(8,0)--(8,3)--(5,3)-- cycle; 
\draw[thick] (6.5,0) -- (6.5,3); % rulings
\draw[thick] (5,1.5) -- (8,1.5);
%points right
\draw[fill=black,thick] (6.5,0) circle (0.03cm);
\draw[fill=black,thick] (6.5,3) circle (0.03cm);
\draw[fill=black,thick] (5,1.5) circle (0.03cm);
\draw[fill=black,thick] (8,1.5) circle (0.03cm);
%labeling
\node[above] (p1) at (6.5,3) {$\small{p_1}$};
%\node[above] (p1) at (6.5,3) {$\small{q_1}$};
\node[above] (p31) at (7.2,3) {$\small{\bar{D}_1}$};
\node[right] (p2) at (8,1.5) {$\small{p_2}$};
\node[right] (p21) at (8,0.8) {$\small{\bar{D}_2}$};

\node[below] (p3) at (6.5,0) {$\small{p_3}$};
\node[below] (p31) at (5.8,0) {$\small{\bar{D}_3}$};

\node[left] (p4) at (5,1.5) {$\small{p_4}$};
\node[left] (p21) at (5,2.2) {$\small{\bar{D}_4}$};

\node[left] (y) at (5 , 3) {$\small{y}$};
\node[left] (x) at (5 , 0) {$\small{x}$};
\node[right] (z) at (8 , 0) {$\small{z}$};
\end{scope}

\end{tikzpicture}
\caption{ $\PP^1\times\PP^1$ with the points $q_1, q_2, p_1, p_2$.}

\label{blowupauto}
\end{figure}

Let $\tilde{\pi}\colon\tilde{Y}\to Q$ be the $(1,3,1,3)$-blow-up of $Q$ in $(p_1,p_2,p_3,p_4)$.  The birational transforms $\tilde{D}_1$ and $\tilde{D}_3$ of $\bar{D}_1$ and $\bar{D}_3$ respectively 
under $\tilde{\pi}^{-1}$ are $(-1)$-curves. Let $\tilde{Y}\to Y$ be the induced contraction. 
The surface ${Y}$  is a weak del Pezzo surface of degree $2$ with an $E_6$ root system of effective curves and thus isomorphic to $\mathfrak{Y}_2$,  see e.g. the global Torelli theorem in \cite{GHK15}. 

 Using this construction, it is straightforward to produce automorphisms on $Y$. First, being a product of two copies of $\PP^1$, every pair $(\psi_1,\psi_2)$ of automorphisms of $\PP^1$ induces an automorphism of $\PP^1\times\PP^1$. 
Let  $y=\bar{D}_1\cap \bar{D}_4$ and $x=\bar{D}_3\cap \bar{D}_4$. Let $\bar{\phi}=(\phi_1,\id_{\PP^1})$ be the automorphism  of $Q$ that is given on $\bar{D}_4$
 by the automorphism $\phi_1$ of $\PP^1$ defined by
\begin{align*}x&\mapsto y\\
y&\mapsto x\\
p_4&\mapsto p_4,
\end{align*} and by the identity on the second ruling. 
By the universal property of blow-ups, $\bar{\phi}$ lifts to an automorphism $\tilde{\phi}$ of $\tilde{Y}$. By construction, $\tilde{\phi}$ maps  $\tilde{D}_1$ to  $\tilde{D}_3$ and vice versa. Thus $\tilde{\phi}$ descends to an automorphism $\phi$ of $Y$.
Note that $\phi$ acts as involution on each of  the  components of the anticanonical divisor $D=D_2+D_4$
 of $Y$.

We can do the same construction for an automorphism $\bar{\psi}=(\id_{\PP^1},\psi_2)$ of $Q$ with $\psi_2$ the automorphism of $\bar{D}_3$ with \begin{align*}x&\mapsto z\\
z&\mapsto x\\
p_3&\mapsto p_3,
\end{align*} 
where $z$ is the point in $\bar{D}_2\cap\bar{D}_3$. This yields an automorphism $\psi$ of $Y$ interchanging the components $D_2$ and $D_4$. 
As we have seen that $\Aut(Y)$ is a subgroup of $\ZZ_2\times \ZZ_2$, it follows that in fact $\phi$ and $\psi$  generate  $\Aut(Y)=\ZZ_2\times \ZZ_2$.

As a consequence, we have the following straightforward lemma.

\begin{lemma}\label{swap} Let $Y$ be a contraction of an $(n,m)$- blow-up of $\mathfrak{Y}_2$ in the interior special points. Let $D$ be the strict transform of the anticanoncial divisor of $\mathfrak{Y}_2$.  
\begin{itemize}
\item[(i)] Let $D_0$ be a component  of $D$. 
Then there is an involution $\phi^Y\colon Y\to Y$ which restricts to an involution on  $D_0$  fixing the interior special point.
\item[(ii)]Suppose $D=D_2+D_4$ and $D_2^2=D_4^2$. Then there is an involution $\psi^Y\colon Y\to Y$ interchanging $D_2$ and $D_4$ while fixing the points in $D_2\cap D_4$.
\end{itemize}
\end{lemma}

\begin{proof}The morphism $\phi$ lifts to any  $(n,m)$- blow-up of $\mathfrak{Y}_2$ in the interior special points and maps any interior $(-1)$ curve to itself, so it descends to an automorphism of $Y$. Similarly, $\psi$ induces an automorphism $\psi^Y$ if $D$ has two components $D_2,D_4$ with $D_2^2=D_4^2$.
\end{proof}

\begin{lemma}\label{isocurves}Let $Y,Y'$ be components of surfaces  in $\PMod_2(\mathscr{P})$. Let $D=D_1+D_2$ and $D'=D'_1+D'_2$ be the double loci. Let  $\alpha\colon Y\xrightarrow{\sim}Y'$ be an isomorphism, mapping $D$ to $D'$. 
Then $\Gamma_Y\xrightarrow{\sim}\Gamma_{Y'}$.
\end{lemma}
\begin{proof} Let $p_1$ and $p_2$ be the interior special points of $Y$ and $q_1, q_2$ be those of $Y'$, assuming $q_i\in \alpha(D_i)$.
The assumptions imply that  $\Gamma_Y$ is regular (degenerate) if and only if  $\Gamma_{Y'}$  is regular (degenerate). 
Then the result follows as  $D_i^2=\alpha(D_i)^2$.
\end{proof}

\begin{proposition}\label{isomcomp}Let $Y$, $Y'$ be components
 of surfaces in $\PMod_2$.  
 Suppose the curve structures $\Gamma_Y, \Gamma_Y'$ are isomorphic and of the same type (see Definition \ref{def:curvestructure}).
 Let $D$, $D'$ be the double curves of $Y$ and $Y'$ respectively. Then there is an isomorphism $Y\xrightarrow{\sim} Y'$ which identifies $D$ and $D'$.
\end{proposition}
\begin{proof} It is enough to  proof the proposition in the type $d_2$
 case, as the proof also implies the $d_1$ and the  $d_4$ cases, these being blow-ups or contractions in nodes of the double curves. 
Let $D=D_1+D_2$ and $D'=D'_1+D'_2$. Write $D_i^2=n_i$ and $(D'_i)^2=n'_i$. As $\Gamma_Y\cong \Gamma_{Y'}$, we can choose indices such that  $n_i=n'_i$. Let $p_i$ and $p'_i$ be the interior special points. If $n_i \geq 0$, 
blow up $Y$ in $p_i$ until $n_i=-1$, and the same for $Y'$. Write $\tilde{Y}$ and $\tilde{Y'}$ for the blow-ups.  These are $(n,m)$-blow-ups
of $\mathfrak{Y}_2$ and thus  we obtain a diagram 
\[
\xymatrix{
&\tilde{Y}\ar[ld]_{\pi}\ar[d]\ar[r]^{\tilde{\alpha}} &\tilde{Y'}\ar[d]\ar[rd]^{\pi'}&\\
Y&\mathfrak{Y}_2\ar[r]^{\sim}&\mathfrak{Y}_2&Y'
}
\]
with $\tilde{\alpha}$ an isomorphism lifting the isomorphism $\mathfrak{Y}_2\xrightarrow{\sim}\mathfrak{Y}_2$, where we can assume that $\tilde{D}_i$ is mapped to $\tilde{D}'_i$. 
\
Let $\Gamma$ and $\Gamma'$ be the intersection graph of integral curves $C$  of $Y$ and $Y'$ with $C^2<0$,  and $C$ not a component of the double curve. 

Then $\tilde{\alpha}$ identifies $\Gamma$ and $\Gamma'$.  Write $\pi''= \pi'\circ\tilde{\alpha}$. Then the extremal cones of $\pi$ and $\pi''$ agree and hence
there is an isomorphism $\alpha\colon Y\to  Y'$ mapping $D_i$ to $D'_i$ and inducing an isomorphism $\Gamma_{Y}\to \Gamma_{Y'}$, see e.g. \cite[Proposition 1.14]{Deb13}.
\end{proof}

For the following, recall our convention that for components $Y_i, Y_j $ of a semistable K3 surface,  the self-intersection number of $Y_i\cap Y_j$ is calculated on $Y_i$.

\begin{lemma}\label{gluealpha} Let $\shY_c,\shY_c'$ be in $\PMod_2(\mathscr{P})$.   
Write $ \shY_c=Y_1 \cup  Y_2 \cup Y_3$ and $\shY'_c=Y'_1 \cup  Y'_2 \cup Y'_3$. Assume $D^2_{21}\neq D_{13}^2$. Suppose that there exist a permutation 
 $\sigma\in \shS_3$ and  isomorphisms of curve structures $\alpha^\Gamma_i\colon \Gamma_{Y_i}\xrightarrow{\sim} \Gamma_{Y'_{\sigma(i)}}$ 
 such that $(Y'_{\sigma(1)}\cap Y'_{\sigma(2)})^2=D^2_{12}$.
Then there is an isomorphism $\alpha\colon \shY_c \xrightarrow{\sim} \shY_c'$.

\end{lemma}

\begin{proof}By Proposition \ref{isomcomp}, for each $\alpha^\Gamma_i$ there is an induced  isomorphism $\alpha_i\colon Y_i \to Y'_{\sigma(i)}$ .
We write $t_1$, $t_2$ for the triple points of $Y$ .  
The interior special point contained in a component $D_{ij}\subset Y_i$ is denoted by $p_{ij}$. We use $t'_1$, $t'_2$  and $p'_{ij}$ for the triple points and interior 
special points on $Y'$, where we assume $\alpha_1(t_i)=t'_i$.
We 
have $\alpha_1(p_{12})=p'_{12}$ and  as  $(Y'_{\sigma(1)}\cap 
Y'_{\sigma(2)})^2=D^2_{12}$, using Lemma \ref{swap}, we can find an isomorphism $\alpha_2\colon 
Y_2\to Y'_{\sigma(2)}$  with  $\alpha_2(p_{21})=p'_{21}$,
 $\alpha_2(p_{23})=p'_{23} $,  $
\alpha_2(t_{1})=t'_{1} $ and $\alpha_2(t_{2})=t'_{2} $, by the (proof of) Proposition \ref{isomcomp}. Similarly, there is  an isomorphism $\alpha_3\colon 
Y_3\to Y'_{\sigma(3)}$  with  $\alpha_3(p_{31})=p'_{31} $,  $\alpha_3(p_{32})=p'_{32} $, $
\alpha_3(t_{1})=t'_{1} $, and $\alpha_3(t_{2})=t'_{2} $. By construction, the $\alpha_i$ glue to an isomorphism  $\alpha\colon Y\xrightarrow{\sim} Y'$.
\end{proof}

The following proposition says that in order to count models, it will be enough to count curve structures. We will first formulate this for models of class $\mathscr{P}$ and then extend it to models of class $\mathscr{T}$.

\begin{proposition}\label{isomorphismprop}
Let $\shY$, $\shY'$ be two models of the DNV family of degree $2$
with $\shY_c, \shY'_c$ in $\PMod_2(\mathscr{P})$ and components $Y_i$ and $Y'_i$ respectively. Then $\shY$ and $\shY'$  are isomorphic if and only if   there are isomorphisms of curve structures $\Gamma_{Y_i}\to \Gamma_{Y'_{\sigma(i)}}$
for some permutation $\sigma \in \shS_3$. 
\end{proposition}
\begin{proof}
It is enough to show that the central fibres are isomorphic, as then the models are isomorphic by uniqueness of the DNV family.

 Either both $\shY$ and $\shY'$ are  isomorphic to $\YYP$, or there is a component $Y_1\subset \shY_c$ such that 
  for the components of the double curve $D=D_{12}+D_{13}$ of $Y_1$, $D^2_{12}\neq D^2_{13}$. There is a component $Y'_{\sigma(1)}$ with $\Gamma_{Y_1}\xrightarrow{\sim} \Gamma_{Y'_{\sigma(1)}}$ and by Proposition \ref{isomcomp} an isomophism $\alpha_1\colon Y_1\xrightarrow{\sim} Y'_{\sigma(1)}$. 

  Let $Y_2$ be the component with $Y_1\cap 
 Y_2=D_{12}$.  
 We also have a component $Y'_{\sigma(2)}$ with $\Gamma_{Y_2}\xrightarrow{\sim} \Gamma_{Y'_{\sigma(2)}}$ and an isomophism $\alpha_2\colon Y_2\xrightarrow{\sim} Y'_{\sigma(2)}$.
 
 Consider first the case $(Y'_{\sigma(1)}\cap Y'_{\sigma(2)})^2=D_{13}^2$. 
 Then, as $D_{12}^2\neq D_{13}^2$ and we have isomorphisms of curve structures $\Gamma_{Y_i}\to \Gamma_{Y'_{\sigma(i)}}$, we have 
   $D^2_{31}=D^2_{23}$ and also $D^2_{23}=D_{13}^2$. 
 Also, there is an isomorphism $\alpha_3\colon Y_3\xrightarrow{\sim} Y'_{\sigma(3)} $ and $(Y'_{\sigma(1)}\cap Y'_{\sigma(3)})^2=D_{12}^2$.
  By the same 
 token, we obtain $D_{32}^2=D_{21}^2$ and $D_{23}^2=D_{12}^2$. We deduce that $\shY_c$ is given by three components $Y_1,Y_2$ and $Y_3$ 
 with $D^2_{21}=D^2_{13}=D^2_{32}=a$ and   $D^2_{12}=D^2_{31}=D^2_{23}=-a-2$ for some integer $a$.  As $\Gamma_{Y_1}$ is non-degenerate, $a\leq 1$ and $-a-2\leq 1$.  We deduce that all $\Gamma_{Y_i}$ are isomorphic:  all self intersection numbers of the $D_{ij}$ are less than $1$, hence the curve structures are regular, hence they are completly determined by the self intersection numbers $D_{ij}^2$. 
 It follows from Proposition \ref{isomcomp} that 
 $Y_1\xrightarrow{\sim}Y_2\xrightarrow{\sim}Y_3$ and the same is true for $\shY'$. 
 
 It is then immediate to choose an isomorphism $Y_i \to Y'_i$ that induces 
 an isomorphism $\shY_c\xrightarrow{\sim} \shY_c'$. Hence $\shY\cong\shY'$ by uniqueness of the DNV family.
 
 In the remaining case,  $(Y'_{\sigma(1)}\cap Y'_{\sigma(2)})^2=D_{12}^2$. Then
 the result follows from  Lemma \ref{gluealpha}. The other direction follows from Lemma \ref{isocurves}. 
  \end{proof}
We now consider models of class  $\mathscr{T}$.

%%%%%%%%%%%%%%%%%%%%%%%%%%%%%%%%%%%%%

\begin{proposition}\label{cstructureT}
Let $\shY$, $\shY' \in \PMod_2(\mathscr{T})$ be two models of the DNV family of degree $2$ with components $Y_i$ and $Y'_i$ respectively.
Then $\shY$ and $\shY'$  are isomorphic if and only if  there are  isomorphisms of curve structures $\Gamma_{Y_i}\to \Gamma_{Y'_{\sigma(i)}} $
for some permutation $\sigma\in \shS_3$.
\end{proposition}

\begin{proof} We first assume that there are isomorphisms of curve structures $\Gamma_{Y_i}\to \Gamma_{Y'_{\sigma(i)}} $
for some permutation $\sigma\in \shS_3$. We first note that there are models  $\hat{\shY}$, $\hat{\shY}'$ of class $\mathscr{P}$ and type II flops   $\hat{\shY} \dashrightarrow \shY$ and 
$\hat{\shY}'\dashrightarrow \shY'$. Write $\hat{\shY}_c=\hat{Y}_1\cup \hat{Y}_2\cup \hat{Y}_3$ 
and ${\hat{\shY}'}_c=\hat{Y}'_1\cup \hat{Y}'_2\cup \hat{Y}'_3$ 
for the components of the central fibres. The curve structures of the  surfaces $\hat{Y}_i$ and $\hat{Y}'_i$ are determined  by the curve structures of  $Y_i$ and $Y'_i$. This follows as any sequence of type I flops 
$\YYT\dashrightarrow \shY$ induces a sequence $\YYP \dashrightarrow \hat{\shY}$ and similar for $\shY'$ and $\hat{\shY}'$.
Hence there exist
a permutation $\tau\in \shS_3$ and  isomorphisms of curve structures
 $\Gamma_{\hat{Y}_i}\to \Gamma_{\hat{Y}'_{\tau(i)}}$.
Then,  Proposition \ref{isomorphismprop} implies that $\hat{\shY}\cong \hat{\shY'}$.  So $\shY$ and $\shY'$ are both flops of the extremal contraction $\hat{\shY}\to \shZ$ defining the type II flop, and hence isomorphic.

We now prove the remaining direction. So assume  $\shY\xrightarrow{\sim}\shY'$ via some isomorphism $h$.  Any isomorphism preservers the number of $(-1)$-curves meeting a component of the double locus $D_{ij}$ and also the self-intersection numbers of the $D_{ij}$. But this datum uniquely determines  curve structures of type $d_4$ and $d_2$. Hence  we can find a permutation  $\sigma\in\shS_3$ and  isomorphisms of 
curve structures $\Gamma_{Y_i}\to \Gamma_{Y'_{\sigma(i)}}$. 
\end{proof}

\subsection{Counting models of class $\mathscr{T}$}
We now count the number of models of class $\mathscr{T}$. Note that by  Remark \ref{finiteMF}, this is a finite problem. 
 Our approach is as follows: 
 we can count the number of models of class $\mathscr{T}$ by starting with $\YT$ 
 and applying all possible elementary modifications of type I such that the flopped curve does not meet the singular locus of the special component.
 By Proposition \ref{projTcurves}, any surface obtained in this way is the central fibre of a model of the DNV family of degree $2$. By Corollary \ref{cor:typeItypeII}, all models can be obtained in this way, up to isomorphism. By Proposition \ref{cstructureT}, the isomorphism classes are uniquely determined by curve structures. Hence, to count models we will count the distinct curve structures that one can obtain from $\YYT$ by type I elementary modifications.
 Note that by  Remark \ref{finiteMF}, this is a finite problem. 
This counting of flops is then essentially reduced to countig possible combinations of  self-intersection numbers of the nodal components of the double locus. 

So let $Y\in \PMod_2(\mathscr{T})$. 
Let $(n_1,n_2)$ be the self intersection numbers of the preimages of the  nodal components of the double curve  - now written as  $D_1,D_2$ - of  the special component $Y_\omega$ of $Y$ on its normalisation.  Let $Y_i$ be the smooth component glued to $D_i$.
We will break up the analysis into several cases, namely  $\rm({i})$ $n_1 \geq -1, n_2 \leq -1$, $\rm({ii})$ $ n_1,n_2 \geq 0$ and 
$\rm({iii})$ $n_1,n_2 \leq -2$.  By the symmetry of $\YT$, these are all cases.

Consider first the case $n_1\geq -1$ and $n_2\leq -1$. The possible numerical combinations are as follows.
\begin{itemize}
\item
{ $n_1\in \{-1,0,1\}$ and $n_2\in \{-1,\dots,-9\}$. 
Note that there are two models with  with $n_2=-8$ for any $n_1$. This is because after flopping $7$ curves from a smooth component  $Y_i$ of $\YT$, there are $2$ interior $(-1)$ that can be flopped, cf. Figure \ref{fig:augYT}.
This gives $30$ models. 

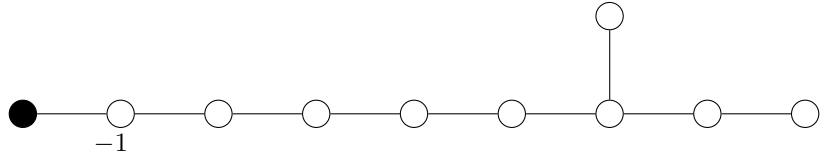
\begin{figure}[]\centering
\begin{tikzpicture}[scale=1.3]

\node[fill=black,draw,shape=circle] (0) at (0,0){};
\node[draw,shape=circle] (1) at (1,0){};
\node (A) at(0.9,-0.3){{\small $-1$}};
\node[draw,shape=circle] (2) at (2,0){};
\node[draw,shape=circle] (3) at (3,0){};
\node[draw,shape=circle] (4) at (4,0){};
 \node[draw,shape=circle] (9) at (6,1){};

  \node[draw,shape=circle] (5) at (5,0){};
   \node[draw,shape=circle] (6) at (6,0){};
    \node[draw,shape=circle] (7) at (7,0){};
%\node (B) at (6.95,-0.3) {\small $-1$};

  \node[draw,shape=circle](8) at (8,0){};

\draw (0)--(1)--(2)--(3)--(4)--(5)--(6) --(7)--(8)
(6)--(9);

%(1)--(3)
%(1)--(4)
%(3)--(5)
%%(5)--(6);
%  
%  \node(A) at (1,-0.3){{\small $C$}};
%  \node(B) at (-2,-0.3){{\small${D_R}$}};
\end{tikzpicture}
\caption{The augmented curve structure of a smooth component $Y$ of $\YT$. The black vertices correspond to  the anticanonical divisor. All  unlabelled curves of $\Gamma_Y$ have self intersection $-2$.}

\label{fig:augYT}
\end{figure}
}
 \item{There are two different ways to obtain  $n_1=2$, as after contracting $2$ curves to obtain $n_1=1$, one can either contract the transform of a fork or a transform of a $(-1)$-curve on $Y_\omega$ meeting $D_2$. In the first  case one can  flop $r_1\in\{0,\dots,9 \}$ curves from $Y_2$, with $2$ choices for flopping $7$ curves as above. In the second case one has to flop one curve from the smooth component $Y_1$
 in order to obtain $n_1=2$. So one can flop $r_2\in\{0,\dots,8 \}$ curves, with $2$ possibilities for $r_2=6$.
This gives $21=11+10$ models.}
 \item{ $n_1\in\{3,\dots 9\}$ . To obtain $n_1$, one needs to  flop $n_1-1$ curves from $Y_2$ to $Y_\omega$ and then flop $n_1+1$ curves to $Y_1$.  One can also flop more curves from $Y_2$ to $Y_\omega$:   for given $n_1$, one can flop $r\in\{0,\dots,9-n_1\}$ additional curves. If $n_1<9$, one value of $n_2$ can be obtained in two different ways, as above. This yields $34$ models.  }
 \end{itemize}
 
 Hence we find $30+21+34=85$ models.  Now, if both $n_1$ and $n_2$ are $\geq 0$, there is only the possibility $(n_1,n_2)=(0,0)$.

 If both $n_1$ and $n_2$ are $\leq -2$, all models are obtained by flopping curves from the smooth components to the special component. These operations are independent on each $D_i$, so the models are given by intersection numbers \[(n_1,n_2)\in \{-2,\dots,-9\}^2,\] with two possibilities to obtain a $-8$ for each entry. This gives $81$ models. More precisely, these models are paremetrised by the set  \[(n_1,n_2)\in \{-2,\dots,-8,-8' ,-9\}^2,\]
 Among these $81$ there are the  $9$ models $(x,x)$ with $x\in \{ -2,\dots, -7,-8,-8',-9\}$. Taking these out, 
the remaining set is given by $(n_1,n_2)$ with $n_1\neq n_2$, giving $72$ models. 
 Taking out the ordering of the components, we obtain $36$ non-isomorphic models in this set. 
 Adding up, we have $131=85+1+9+36$  projective models of class $\mathscr{T}$. 
  
\begin{theorem}\label{modelsT}  There are $131$ surfaces in $\PMod_2(\mathscr{T})$.
\end{theorem}

Note that the number of surfaces in the theorem is the number of orbits of maximal cones in the Mori fan under the action of the birational goup.

\subsection{Counting models of class $\mathscr{P}$}\label{sec:countP}

We first count models such that at least one component of the central fibre has non-regular curve structure. 

\begin{theorem}\label{modelsP}
There are $104$  surfaces $Y_c=Y_1\cup Y_2 \cup Y_3 $ 
in  $\PMod_2(\mathscr{P})$ sucht that there is a component $Y_i$ with non-regular curve structure. Explicitly, 
these are given as follows:
\begin{itemize}
\item[(i)] Surfaces $Y_c$ such that $\Gamma_{Y_1}$ is not regular, $\Gamma_{Y_2}$ is non-degenerate and $\Gamma_{Y_3}$ is degenerate  and regular. There are $71$ such models.
\item[(ii)] Surfaces $Y_c$ such that $\Gamma_{Y_1}$ is not regular, $\Gamma_{Y_2}$ is non-degenerate and $\Gamma_{Y_3}$ is not regular. There are $8$ such models.
\item[(iii)] Surfaces $Y_c$ such that $\Gamma_{Y_1}$ is not regular, $\Gamma_{Y_2}$ is non-degenerate and $\Gamma_{Y_3}$ is non-degenerate. There are $25$ such models.
\end{itemize}
\end{theorem}
\begin{proof} 
First, note that the three cases items in the theorem exhaust all possible configurations of curve structures with at least one of the structure not regular: by assumption, one of the curve structures is not regular, so a second one has to be non-degenerate and the third one is as in the statement of the theorem.

Now, recall that there is a unique vertex $v_{D_{12}}$ meeting $D_{12}$, by regularity.
If $Y_c$ is a model as in $\rm{(i)}$, then $v_{D_{12}}.D_{12}=1$
or $v_{D_{12}}.D_{12}=2$. In the first case,  
 the conditions on the curve structures 
impy that
$D^2_{13}=4$,  $D^2_{21}\in [-3,1]$ and  
$D^2_{32}\in [2, 10]$. Projectivity of $Y_c$ further implies $D^2_{21}<1$, by Proposition \ref{nondegproj4}.  
Conversely, each  triple $(D^2_{13},D^2_{21}, D^2_{32})$ with  
$D^2_{13}=4$, $D^2_{21}\in [-3,0]$ and $D^2_{32}\in [2, 10]$ 
determines a  projective model, and these are pairwise non-isomorphic 
as the  intersection numbers of the double curves are distinct. Hence 
there are 36 models with this 
specification of curve structures  and $v_{D{12}}.D_{12}=1$.

If $v_{D{12}}.D_{12}=2$, we have $D^2_{12}=4$. Assume first that $D^2_{31}\geq 2$.  By non-degeneracy of $Y_2$, we have $D^2_{23}\in [-4,1]$ and counting possible flops, $D^2_{31}\in [ 2, 6+D^2_{23}]$. Projectivity requires  $D^2_{23}<1$ 
by  Proposition \ref{nondegproj4}, so one obtains $15$ distinct models.
Now suppose  $D^2_{31}< 2$. Then $D^2_{13}\in [-3,1]$ and $D_{32}^2\in [ 2, 6+D_{13}^2]$, giving $20$ models.

If $Y_c$ is a model as in $\rm{(ii)}$, we can assume $v_{D_{12}}.D_{12}=2$. Then either   $v_{D_{31}}.D_{31}=2$ or $v_{D_{31}}.D_{31}=1$.
In the first case,  $D^2_{12}=D^2_{32}=4$. One finds that up to isomorphism, we have $D^2_{13}\in [-3,0]$, so there are $4$  models in this case.
In the second case, $D^2_{31}=4$, $D^2_{12}=4$ and $D^2_{23}\in [-3,0]$ by non-degeneracy and projectivity.  Hence there are again $4$ such models.

If $Y_c$ is a model as in $\rm{(iii)}$, we can  assume $v_{D_{23}}.D_{23}=2$. Then $D_{13}=4$, $D_{12}\in[-3,1]$ and $D_{23}\in [-3,1]$. These models are pairwise non-isomorphic as the curve structure of $\Gamma_{Y_3}$ always has an exceptional vertex $v$ with $|L(v)|=9$. This never occurs for $\Gamma_{Y_{2}}$. One finds $25$ models.
\end{proof}

%%%%%%%%%%%%%%%%%%%%%%%%%%
We now count models $Y\in  \PMod_2(\mathscr{P})$ such that $\Gamma_{Y_i}$ is regular for all components $Y_i\subset Y$.
We use the following shorthand notation.  Fix a numbering of the components of $\YP$.
There is a
sequence of type I modifications  $\YP\to Y$. Write $Y=\cup_i Y_i$, assuming that the $i$-th component of $\YP$ maps to $Y_i$. Let $D_i=Y_i\cap Y_{i+1}$, with  $i=1,2,3$, indices considered modulo $4$,  considered as curve on $Y_i$. We have $D_i^2=-1+n_i$ for some $n_i\in \ZZ$. 

Note that if $\Gamma_{Y_i}$ is regular for $i=1,2,3$,  the triple $(n_1,n_2,n_3)$ determines $Y$ uniquely, 
by Proposition \ref{isomorphismprop}. The meaning of $(n_1,n_2,n_3)$ is that $n_i$ curves  are flopped from  $Y_i$ to $Y_{i+1}$ (or from $Y_{i+1}$ to $Y_i$, if $n_i<0$).

In order to simplify the arguments below, we now define certain operations on triples of integers. 

\begin{definition}Let  $(x,y,z)$ be a triple. We shall call the triples $(z,x,y)$ and $(y,z,x)$ the \emph{shifts} of $(x,y,z)$. The triple  $(-y,-x,-z)$ will be called  the \emph{involution} of $(x,y,z)$. 
We write $(x,y,z)\sim (u,v,w)$ if the triples $(x,y,z)$ and $(u,v,w)$ are related by a sequence of shifts and involutions and call $(x,y,z)$ and $(u,v,w)$ \emph{equivalent}.
\end{definition}

\begin{remark}
We let $s$ be the shift operator $(x,y,z)\mapsto (z,y,x)$ and $\i$ be the involution operator $(x,y,z)\mapsto (-y,-x,-z)$. 
 Then $\i\circ s\circ \i(x,y,z)=s^2(x,y,z)$ and $\i\circ s^2\circ \i(x,y,z)=s(x,y,z)$. Also, $s\circ \i\circ s=\i$ so any sequence of shifts and involutions reduces to $\i^a\circ s^b$ or $s^b\circ \i^a$ with $a\in\{0,1\}$ and $b\in\{0,1,2\}$. In particular, the group generated by $\i$ and $s$ is isomorphic to $S_3$.
\end{remark}

\begin{lemma}\label{shift} Let $Y$, $Y'$ be two models in $\PMod_2(\mathscr{P})$  such that all curve structures $\Gamma_{Y_i} $ and $\Gamma_{Y'_i}$ are regular, with triples $(n_1,n_2,n_3)$ and $(n'_1,n_2',n_3')$ (relative to a chosen numbering of
the components of  $\YP$). Then $Y\cong Y'$ if and only if $(n_1,n_2,n_3)\sim(n'_1,n_2',n_3')$.
\end{lemma}
\begin{proof} Write $Y=\cup Y_i $ and $Y'=\cup Y'_i$ for the central fibres. If $Y$ and $Y'$ are isomorphic, after possible renumbering, we can assume   $Y_1\cong Y'_1$. 
Let $D'=D'_1+D'_2$ be the anticanonical cycle on $Y'_1$ with numbering chosen such that $D_1\cong D'_1$, with $D_1=Y_1\cap Y_2$ and write $(D'_1)^2=-1+x$ and $(D'_2)^2=-1+y$. Then either the tupel  $(-y,x)$ or the tuple $(-x,y)$ is contained in a shift of $(n_1',n_2',n_3')$. Let $Y'_2$ be the component of $Y'$ 
meeting $Y_1'$ in $D'_2$ and let $D''$ denote the component of the anticanonical cycle  on $Y'_2$ that is not glued to $Y'_1$. Write $(D'')^2=-1+z$. It follows that after applying shifts, either  $(n_1',n_2',n_3')\sim (-z,-y,x)$ or $(n_1',n_2',n_3')\sim (-x,y,z) $. Because $Y_1\cong Y'_1$, we have  $x=-n_3, y=n_1$ and $z=n_2$ thus  $(n_1',n_2',n_3')\sim (-n_2,-n_1,-n_3)$ 
or $(n_1',n_2',n_3')\sim (n_3,n_1,n_2) $. So indeed $(n_1,n_2,n_3)\sim(n'_1,n_2',n_3')$.\\
If conversely the triples are equivalent, the corresponding models are isomorphic: let $(x,y,z)$  be a triple defining a model $Y=Y_1\cup 
Y_2\cup Y_3$. By our convention, this is the model where $x$ curves are flopped from $Y_1$ to $Y_2$, $y$ curves are flopped from $Y_2$ to 
$Y_3$ and $z$ curves are flopped from $Y_3$ to $Y_1$. The shift $(z,x,y)$ also defines a model, say $Y'=Y_1'\cup Y_2'\cup Y'_3$. Then 
$Y_1$ and $Y_2'$ are obtained from $\mathfrak{Y}_2$  by the same type I modifications and thus there is an isomorphism $Y_1\to Y_2'$. Similarly there are isomorphisms $Y_2\to Y'_3$ and $Y_3\to Y'_1$.  Thus $Y\xrightarrow{\sim} Y'$ by Proposition \ref{isomorphismprop}. The case of an involution is similar.
\end{proof}

\begin{theorem}\label{modelsPreg} There are $27+103+225=353$  surfaces $Y=Y_1\cup Y_2\cup Y_3$ in $\PMod_2(\mathscr{P})$ such that  the  curve structures $\Gamma_{Y_i}$ are all regular. Explicitly, these are, up to  equivalence, given as follows:
\begin{itemize}
\item[(i)] Surfaces with all $\Gamma_{Y_i}$ non-degenerate. These are given by the triples
\begin{align*}
&(0,1,-1),(0,1,2),(0,1,-2),(0,2,1),(0,2,-2),(0,-1,2),\\ &(0,-1,1),(0,-2,2),
(1,2,-1),(1,2,-2),(1,-1,2),(1,-2,2),\end{align*}  surfaces $(x,y,y)$ with $x\in \{1,2\}$ and $y\in  \{-2,-1,0,1,2\}\backslash \{x\}$, the surfaces $(0,1,1)$ and $(0,2,2)$ and surfaces $(x,x,x)$ with $x\in\{-2,-1,0,1,2\}$. These are $27$ surfaces.
\item[(ii)] Surfaces with one $\Gamma_{Y_i}$ degenerate: triples $(3,y,-3)$ with $0\leq y\leq 2$ and triples $(x,y,z)$ with  $x,y\in \{-2,\dots, 2\}$, $z\in \{y-6,\dots, -3\}$.  These are $103$ surfaces.
\item[(iii)] Surfaces with two $\Gamma_{Y_i}$ degenerate. These are given by the sets  
\begin{align*}
&K=\{(x,-3,-3)\mid  3\leq x\leq 9\}\\
&M(0)=\{ (x,0,z) \mid -3\geq x \geq -6, 6\geq z\geq 3\}\\
&M(-1)=\{ (x,-1,z) \mid -3\geq x \geq -7, 5\geq z\geq 3\}\\ 
&M(-2)=\{ (x,-2,z) \mid -3\geq x \geq -8, 4\geq z\geq 3\}\\
&N(-2)=\{ (x,y,-2) \mid -3\geq y\geq -8 , -3\geq x\geq y-6\}\\
&N(-1)=\{ (x,y,-1) \mid -3\geq y\geq -7, -3\geq x\geq y-6\}\\
&N(0)=\{ (x,y,0) \mid: -3\geq y\geq -6, -3\geq x\geq y-6\}\\
&N(1)=\{ (x,y,1) \mid -3\geq y\geq -5, -3\geq x\geq y-6\}\\
&N(2)=\{ (x,y,2) \mid -3\geq y\geq -4, -3\geq x\geq y-6\}.
\end{align*}
Adding up, these are  $7+16+15+12+57+45+34+24+15= 225$ surfaces.
\end{itemize}
\end{theorem}
\begin{proof}
We consider the competely non-degenerate case $\rm{(i)}$ first. Non-degeneracy of all curve structures implies $-2\leq n_i\leq 2$ for all $i$, compare Example \ref{exnondeg}. This defines a set of $125$ candidate triples $(x_1,x_2,x_3)$. 
There are $5$ triples $(x,x,x)$ that only appear once, all other triples appear with multiplicity $3$ via the shift relation above. Hence there are at most  $45$ distinct surfaces. These include the triples $(x,x,x)$, and  $20$ triples of the form $(x,y,y)$ with with $x\in \{-2,-1,0,1,2\}$ and $y\in  \{-2,-1,0,1,2\}\backslash \{x\}$. Here $(x,y,y)$ is equivalent 
 to $(-x,-y,-y)$ modulo a shift and an involution, so the distinct surfaces in this subset are given by the conditions $x\in \{1,2\}$ and $y\in  \{-2,-1,0,1,2\}\backslash\{x\}$ and the surfaces $(0,1,1,)$ and $(0,2,2)$. This gives $12$ non-equivalent triples.  Having enumerated these cases, there remain $20$ triples to discuss.
 These can be recovered by applying the involution to the list of triples $(x,y,z)$ in the statement of the proposition that have pairwise distinct entries. By Lemma \ref{shift}, the triples with pairwise distinct entries
 tupels define non-isomophic surfaces.  This finishes the completely non-degenerate case.
 
We turn to the case $\rm(ii)$ of precisely one degenerate curve structure. 
Let the component of $Y$ with regular degenerate curve structure be denoted by $Y_1$.  By the degeneracy assumption, after maybe  applying an involution, we can assume $n_3\leq -3$ and $ n_1\leq 3$. Assume first that both conditions hold sharp, i.e.  
$n_1=3, n_3=-3$. From the non-degeneracy assumption on $Y_2$ and $Y_3$ it follows that  $-2\leq n_2\leq 2$  
and any such choice gives a projective surface $Y$. The surfaces defined by  $(3,n_2,-3)$ are equivalent to  surfaces defined by $(3,-n_2,-3)$, so the 
subset $0\leq n_2\leq 2$ already gives all equivalence classes.  
Now  assume $n_3\leq -3$ and $n_1<3$. 
It follows from non-degeneracy  of $Y_2$ that $ -2\leq n_1\leq 2$ and then one obtains $n_1-6\leq n_3\leq -3$ by counting the total number of curves in the curve structures of the $Y_i$. 
Together with the condition $-2\leq n_2\leq 2$,  we obtain the set $M$ of triples $(x,y,z)$ with  $x,y\in \{-2,\dots, 2\}$, $z\in \{y-6,\dots, -3\}$. None of the triples in $M$ are equivalent: suppose $(x,y,z)\sim (x',y',z')$, say under a sequence $T$ of shifts and involutions:
assume we  have $T=\i^a\circ s^b$. Then  $|z| > |x|,|y|$ shows $b=0$ and $z<0$ then shows $a=0$. The case $T=s^a\circ \i^b$ is done the same way.

Now consider case $(\rm{iii})$, i.e. assume that two of the curve structures are degenerate. 
Let $Y_1$ be the component with  $\Gamma_{Y_1}$ non-degenerate, and $Y_2$,$Y_3$ be the components with regular but degenerate curve structures. Then degeneracy on $Y_2$ implies  $n_1\leq -3$ or $n_2\geq 3$ and 
from $Y_3$ we get $n_2\leq -3$ or $n_3\geq 3$. Suppose first $n_1\leq -3$ and $n_2\geq 3$. Then $n_1=-3$ and $n_2=3$ by counting curves on $Y_2$. It follows that $3\leq n_3\leq 9$, giving the set  $K$ 
of triples $(x,-3,3)$ with $3\leq x\leq 9$. If  $n_2\leq -3$ and $n_3\geq 3$ one obtains equivalent  tupels  thus isomorphic surfaces. 

Now assume $n_1\leq -3$ and $n_2\leq 2$.  
By counting curves, we have 
\begin{align*}
-3&\geq n_1\geq n_2-6\\
2&\geq n_2\geq n_3-6\\
-2&\leq n_3\leq 6+n_2.
\end{align*}
We shall first consider the case $n_3\geq 3$.  Then $2\geq n_2\geq -3$. This gives the following sets:
\begin{align*}
&M(2)=\{ (x,2,z) \mid -3\geq x \geq -4, 8\geq z\geq 3\}\\  
&M(1)=\{ (x,1,z) \mid -3\geq x \geq -5, 7\geq z\geq 3\}\\ 
&M(0)=\{ (x,0,z) \mid -3\geq x \geq -6, 6\geq z\geq 3\}\\
&M(-1)=\{ (x,-1,z) \mid -3\geq x \geq -7, 5\geq z\geq 3\}\\ 
&M(-2)=\{ (x,-2,z) \mid -3\geq x \geq -8, 4\geq z\geq 3\}\\ 
&M(-3)=\{ (x,-3,3) \mid -3\geq x \geq -9 \}.
\end{align*}

Applying the involution and a shift we see that the sets $M(i)$ and $M(-i)$ describe equivalent triples and  hence the same isomorphism classes of surfaces for $i=1,2$, as do $M(-3)$ and $M_0$.

Now suppose again $n_1\leq -3$ and $n_2\leq 2$, but assume $n_3\leq 2$, so necessarily $n_2\leq -3$. Then by non-degeneracy of $Y_1$ we must have $2\geq n_3 \geq -2$. We obtain the following sets:
\begin{align*}
&N(-2)=\{ (x,y,-2) \mid -3\geq y\geq -8 , -3\geq x\geq y-6\}\\
&N(-1)=\{ (x,y,-1) \mid -3\geq y\geq -7, -3\geq x\geq y-6\}\\
&N(0)=\{ (x,y,0) \mid -3\geq y\geq -6, -3\geq x\geq y-6\}\\
&N(1)=\{ (x,y,1) \mid -3\geq y\geq -5, -3\geq x\geq y-6\}\\
&N(2)=\{ (x,y,2) \mid -3\geq y\geq -4, -3\geq x\geq y-6\}.
\end{align*}
 We  observe that the sets $M(i)$ and the $N(i)$ contain non-equivalent triples: this follows, as 
the number $i$ only appears as an entry  in $M(i)$ and then 
$x,y$ have to have alternating signs, so $(x,y,i)$ is not contained in any of the $N(i)$. 
This shows that models defined by the  $N(i)$ are not isomorphic to models defined by triples 
in the sets $M(i)$. The set $N(i)$ also parameterises models not isomorphic to any model in any $N(k)$ if $i\neq k$: let $(x,y,i)\in N(i)$. Assume there is $(x',y',i')\in N(i')$ with $(x,y,i)\sim (x',y',i'), i'\neq i$, i.e. there exists a composition $T$
 of shifts and involutions  with $T(x,y,i)=(x',y',i')$. 
 Since the absolute value of the middle entry of these triples is smaller than the absolute value of the other entries, this is only possible if $i'=-i$ and  $T=\i$. But now the claim follows from the fact that the third entry is always negative.   
Thus, the $N(i)$ parameterize
distinct surfaces. 
The remaining cases, $n_2\geq 3, n_1\geq -2$ and $n_3\geq 3$ or $n_3\leq 2$ desribe the same models: 
$(x,y,z)\sim (z,x,y)\sim (-x,-z,-y)=:(x', y',z'),$ and thus 
from $x\leq -3, y\leq 2, $ it follows $x'\geq 3, z'\leq -2$.  Then $z\geq 3$ gives $y'\leq -3$  and $z\leq 2$ gives $y'\geq -2$. 
This concludes the proof.
\end{proof}

This concludes the count: if $Y\in \PMod(\mathscr{P})$, then either all curve structures of components of $Y$ are regular or not, In the first case, there are either one, two or three components of $Y$ with non-degenerate curve structure, in the second case, there is a component with non-regular curve structure. These cases are covered by Theorem  \ref{modelsPreg} and  Theorem \ref{modelsP}.
Hence we have the following:

\begin{theorem}\label{modelsPtotal}  There are $457=104+353$ surfaces in $\PMod_2(\mathscr{P})$.
\end{theorem}
We point out that the number of these surfaces is the number of orbits of maximal cones of the Mori fan under action of the birational group.

\subsubsection{Mori fan}
We can now count the number of maximal cones of the Mori fan of the DNV family of degree 2. For this we will have to take the action of the birational automorphism group 
and its action on the maximal cones of the Mori fan into account, which we discussed at the end of Section \ref{sec:flopauto}.
We will first count the number of symmetric models.

\begin{proposition}\label{numbsym}There are $22$ symmetric models of the DNV family of degree $2$. More precisely, the symmetric models $\shY$ of the DNV family of degree $2$ of class $\mathscr{P}$ are 
\begin{itemize}
\item[(i)] the models $\{ (0,n,-n) \mid n\in \{-3,\dots 6\}\}$, with notation as in Section \ref{sec:countP} and 
\item[(ii)] the model such that  $(a)$ $\Gamma_{Y_i}$  is not regular for $i=1,3$ and $(b)$ the  intersection numbers are  $D^2_{12}=D^2_{32}=4$, $D^2_{13}=D^2_{31}=-1$,
$D^2_{21}=D^2_{23}=-6$.
\end{itemize}
The symmetric models of class $\mathscr{T}$ are given by $(x,x)$ with $x\in \{0,-1,-2,-3,\dots,-8,-8',-9\}$, with notation as in the proof of Theorem \ref{modelsT}.
\end{proposition}
\begin{proof}
We first count symmetric models of class $\mathscr{P}$. A model $\shY$ is symmetric if and only if there is an automorphism of $\shY_c$  that identifies two of the components, say $Y_1$ and $Y_3$. We have mentioned that this implies $D_{13}^2=D_{31}^2=-1$. As $Y_1\cong Y_3$, our model is thus completely specified by the curve sructure of $Y_1$.  The curve structure  $\Gamma_{Y_1}$ is obtained  from $\mathfrak{Y}_2$. The condition $D^2_{13}=-1$ implies that in fact,  $Y_1$ is obtained from blow-ups or blow-downs in the interior special point $p\in D_{12}$. If $\Gamma_{Y_1}$ is regular, it follows that, in the teminology of  Section \ref{sec:countP}, $n_1\leq 6$. Because  of the symmetry, $D^2_{12}=D^2_{32}$ and hence $n_1\geq -3$. Conversly, each such choice of $n_1$ implies a unique symmetric model.
If $\Gamma_{Y_1}$ is not regular, then the model is uniquely determined, compare Proposition \ref{symmbirat}: we have $D^2_{12}=D^2_{32}=4$, $D^2_{13}=D^2_{31}=-1$, 
$D^2_{21}=D^2_{23}=-6$ and also $\Gamma_{Y_3}$ is not regular. Also, both $\Gamma_{Y_1}$ and $\Gamma_{Y_3}$ have three vertices while $\Gamma_{Y_2}$ has $18$ vertices.   

For $\shY_c\in \PMod_2(\mathscr{T})$, 
being symmetric is the same as having isomorphic smooth components: indeed, suppose $Y_1,Y_3$ are the smooth components, $Y_\omega$ is the special component  and suppose there is an isomorphism $\gamma\colon Y_1\to Y_3$.  By Proposition \ref{isomcomp}, we can assume $\gamma$ maps $Y_1\cap Y_\omega$  to $Y_2\cap Y_\omega$.  Then one can show -- using the morphism $\bar{\psi}$ defined in Section \ref{automorph}
 --   
that there is an automorphism  $\psi_\omega$ on $Y_\omega$  
that exchanges the nodal components of the anticanonical cylce of $Y_\omega$.  Using $\psi_\omega$ and $\gamma$ (and maybe an involution on the smooth components), from the universal property of pushouts, 
one gets an automorphism  of $\shY_c$ that maps $Y_1$ to $Y_3$.
A similar reasoning as above gives the set of symmetric models of class $\mathscr{T}$. 
Alternatively, it is easy to see that these are precisely the models that can be obtained by a type II flop from a symmetric model of class $\mathscr{P}$.
\end{proof}

We can now count all maximal cones of the Mori fan.

\begin{theorem}\label{theo:countingcones}
Let $\shY\to S$ be a model of the Dolgachev-Nikulin-Voisin family. Then $\Morifan(\shY/S)$ has $3460$ maximal cones. Of these $753$ are associated to a model of class $\mathscr{T}$ and $2707$ are associated to a model of class $\mathscr{P}$. 
\end{theorem}

\begin{proof}  By Proposition \ref{numbsym}, there are $11$ symmetric models in $\PMod(\mathscr{T})$, 
out of $131$  isomorphism classes of models of class  $\mathscr{T}$ in total (Theorem \ref{modelsT}). It now follows from Proposition \ref{orbitmain} that there are $120$ models having $6$ associated cones and $11$ models having $3$ associated cones, giving us $120 \times 6 + 11 \times 3=753$ maximal cones.

Again by Proposition \ref{numbsym}, there are $11$ symmetric models  $\PMod(\mathscr{P})$.  By Theorem \ref{modelsPtotal}, there are $457$ 
 projective models of class $\mathscr{P}$ in total, hence,  by Proposition \ref{orbitmain},
  there are $457-11=446$ models having $6$ associated cones, $10$ models having $3$ associated cones and the model $\YYP$ defining a unique cone in the Mori fan. This defines $446 \times 6 + 10 \times 3 + 1=2676+30+1=2707$ 
 maximal cones. 
 Altogether, we obtain  $753+2707=3460$ maximal cones.
\end{proof}

\section{The Secondary fan}\label{sec:secondaryfan}
In this section, as an application of our results, we give a description of the secondary fan of the DNV family of degree $2$, as introduced in \cite{HKY}.
This fan is obtained by coarsening the Mori fan of $\shY$. Roughly, it is obtained by deleting all facets that correspond to flops that do not change the dual intersection graph.

\subsection{Preliminaries}
For a cone $C$ in a vector space we denote by  $\Int C$ its interior and by $\Relint C$ the relative interior of $C$. 
Let $\shY\to S$ be a model of the DNV family of degree $2$.
The set of maximal cones of $\Morifan(\shY/S)$ will be denoted by $\Morifan_{\operatorname{max}}(\shY/S)$. 
We recall that $\Morifan(\shY/S)$ contains only finitely many cones and that these are all rational polyhedral.

Recall that any interior facet $\tau$ of $\Morifan(\shY/S)$ corresponds to a flop $f_\tau\colon \shY'\dashrightarrow \shY''$ for $\shY',\shY''$ models of the DNV family, see Proposition \ref{prop:intconesflops}. By Corollary \ref{classflop}, $f_\tau$ is a type I or type II flop and we will 
correspondingly call $\tau$ of type I or type II. We denote by $\shF$ the set of all interior facets that correspond to type II flops and 
set $\shM=|\Morifan(\shY/S)|\backslash\cup_{\tau\in\shF}|\tau|$. Let $\shC$ be a connected component of $\shM$ and let $C(f)$ be a cone of $\Morifan(\shY/S)$ such that $\Int C(f)\cap \shC\neq \varnothing$, 
 defined by a marked model $(\shY_f,f)$, 
i.e. a map $f\colon\shY\dashrightarrow \shY_f$. 
Let $I(f)$ be the set of all maximal cones of $\Morifan(\shY/S)$ consisting of cones $C(g)$  corresponding to  models $(\shY_g,g)$, such that there is a sequence of type I flops $\phi\colon\shY_f\dashrightarrow \shY'$ and an isomorphism $\gamma\colon \shY'\to \shY_g$ giving a commutative diagram
\[
\xymatrix{
& \shY\ar@{-->}[dl]_f\ar@{-->}[dr]^g& \\
\shY_f\ar@{-->}[r]^{\phi}& \shY'\ar[r]^\gamma & \shY_g.
}
\]
By the construction of $\shM$ and $\shC$ we have $\shC \subset \cup_{g \in I(f)}C(g)$ and in fact
\begin{equation}\label{eq:conerepres}
	\bar{\shC}=\cup_{g\in I(f)}C(g).
\end{equation}
It is immediate that this construction does not depend on the choice of $\shY\to S$ in the sense that if $\shY'\to S$ is another model and $\phi\colon \shY \dashrightarrow \shY'$ is a birational isomorphism, the identification $\phi_*\Morifan(\shY/S)\cong\Morifan(\shY'/S)$ is compatible with the construction of $\shM$.

It is known that the closures of the connected components of $\shM$ define a fan, the \emph{secondary fan}  $\Morifan^{\operatorname {2nd}}(\shY/S)$
of $\shY$, as can be shown by adapting the techniques in \cite{HKY}. However, as no published proof is 
available, we offer an alternative proof of this fact via the results of this paper.

By  Proposition \ref{orbitmain}, there is a unique  cone $\sigma_\mathscr{P}$ associated to $\YYP$.  Write $\shC_\mathscr{P}$ for the connected component of 
$\shM$ such that  $\shC\cap \Int \sigma_\mathscr{P}\neq \varnothing$.

\begin{lemma}\label{C_P}Let $\shC$ be a connected component of $\shM$.  Let $\bar{\shC}=\cup_{g\in I(f)}C(g)$, 
with maps  $g\colon\shY\dashrightarrow \shY_g $. Then the following holds:
\begin{itemize}
\item[(i)]	If $\shC=\shC_\mathscr{P}$, then  $(\shY_g)_c\in \PMod_2(\mathscr{P})$ for all $g$.
\item[(ii)] If  $\shC\neq \shC_\mathscr{P}$, then $(\shY_g)_c\in \PMod_2(\mathscr{T})$ for all $g$.
\end{itemize}
\end{lemma}
\begin{proof}
This follows from the definition of $\shM$, Lemma \ref{onetypeIIflop} and Proposition \ref{orbitmain} : there is only one cone corresponding to $\YYP$ and any model $\shY'$ of class $\mathscr{P}$ can be obtained from $\YYP$ by a sequence of type I flops. 
\end{proof}

We will call the components of $\shM$ other than $\shC_\mathscr{P}$ components of type $\mathscr{T}$.
We now calculate the number of connected components of $\shM$. We first prove some lemmas on sequences of flops. Recall that a model $\shY$  is symmetric if there is an automorphism  of the special fibre $\shY_c$ that identifies to components of $\shY_c$, cf. Definition \ref{def:symmetric}.

\begin{lemma}\label{nobadsequence}Let $\shY$ be a model 
with $\shY_c\in \PMod_2(\mathscr{T})$.  Assume $\shY$ is symmetric. Let 
$F\in \Mov(\shY/S)$ 
be a divisor and suppose
\begin{equation}\label{Iseq} \shY \dashrightarrow \shY_1\dashrightarrow \dots \dashrightarrow \shY_n\dashrightarrow \shY\end{equation}   
 is a sequence of 
  $F$-flops. Then at least one flop in this sequence is of type II.
  
\end{lemma}
\begin{proof}We show that the assumption that all flops are of type I leads to a contradiction.
Consider the model  $\YYP$. By Proposition \ref{typeIIflopexist}, any elementary modification of type 
II 
of $\YP$
lifts to a type II flop on $\YYP$.
Let $\shY_I$ be the model obtained by applying a single type II flop $
\phi$ to $\YYP$. 
Then $\shY_I$ is symmetric. By Corollary \ref{cor:seqtypeI}, 
there is a composition of type I 
flops $\YYT \dashrightarrow \shY_I$ and $\YYT \dashrightarrow \shY$. 
Thus there is a 
composition of 
maps
 \[ \shY_I \dashrightarrow \YYT \dashrightarrow \shY  \dashrightarrow\shY  
 \dashrightarrow\YYT\dashrightarrow \shY_I\]  that factors into a sequence of  $H$-flops of type I for some 
 $H\in \Mov(\shY_I/S)$. Hence it is enough to show the Lemma for $\shY=\shY_I$. In this case, by applying a type II flop $\shY_I\dashrightarrow \YYP$,   
 Sequence \ref{Iseq} induces a sequence 
\[ \YYP \dashrightarrow \shY'_1\dashrightarrow \dots \dashrightarrow \shY'_n\dashrightarrow\YYP\]
of $(\phi^{-1})_*F$-flops, 
as $\phi$ contracts a curve that is disjoint from all the exceptional loci. 
From  Corollary \ref{autocoro} it follows that $n=0$.  Hence a sequence as (\ref{Iseq}) does not exist. 
\end{proof}

\begin{lemma}\label{lem:typeIflopsT}
Let $\shY$ be a model of the DNV family of degree $2$ with $\shY_c\in \PMod_2(\mathscr{T})$. Assume $\shY$ is not symmetric. Then there sequence of type I flops 
\[
 \shY \dashrightarrow \shY_1 \dashrightarrow \dots \dashrightarrow \shY_n \dashrightarrow \shY
\] such that the composition $\gamma$ of these flops is not an automorphism. 
\end{lemma}
\begin{proof}Write $\shY_c=Y_1\cup Y_2\cup Y_3$ with $Y_1$ the special component. As $\shY$ is not symmetric, $Y_2$ and $Y_3$ are not isomorphic. 
Given $\shY$, by Proposition \ref{cor:seqtypeI}, there exists a sequence $\phi\colon\shY\dashrightarrow \YYT$ of type I flops. Necessarily, $\phi$ maps  the special component to the special component.
Let $\psi$ be the automorphism from Proposition \ref{prop:syminvolution}	. Consider the composition $\gamma=\phi^{-1}\circ \psi \circ \phi \colon \shY \dashrightarrow \shY$. By construction, $\gamma$ maps $Y_2$ to $Y_3$ and $Y_3$ to $Y_2$. As these components are not isomorphic, it follows that   $\gamma$ 
is not an automorphism.  Being a composition of small modifications, $\gamma$ has a factorisation 
\begin{equation}\label{seq:alltypeI}
 \shY \dashrightarrow \shY_1 \dashrightarrow \dots \dashrightarrow \shY_n \dashrightarrow \shY
\end{equation} into $F$-flops for a divisor $F\in\Mov(\shY/S)$. Note that the smooth component $D_\omega$ of the restriction of the double curve of $\shY$ to $Y_1$ is disjoint from $\Ex(\gamma)$. Hence, by Corollary \ref{classflop},  all flops in the factorisation (\ref{seq:alltypeI}) are of type I and we thus obtain the desired sequence.
\end{proof}

\begin{proposition}\label{numberconessec}The topological space $\shM$ has $4$ connected components. These are given by $\shC_\mathscr{P}$ and  $3$ components of type $\mathscr{T}$.
\end{proposition}
\begin{proof} Here we chose $\YYT$ as a reference model and hence
consider $\Morifan(\YYT/S)$.  By Proposition \ref{orbitmain}, there are 3 cones in the orbit 
of $\Nef(\YYT/S)$. We can write these as $C(f_1)$, $C(f_2)$, $C(f_3)$  where 
$f_1=\id_{\YYT}$ and $f_2,f_3\in \Bir(\YYT/S)$ are birational 
automorphisms of $\YYT\to S$.   Let $\shC_i$ 
denote the connected component of $\shM$ such 
that $C(f_i)\subset \bar{\shC_i}$. By Lemma \ref{nobadsequence} applied to $\YYT$, the $\shC_i$ are all distinct.
Let $\sigma$ be a cone associated to a model $\shY'$ of class 
$\mathscr{T}$. We will show that 
there is an $i\in \{1,2,3\}$ such that $
\sigma\subset  \bar{\shC_i}$. This implies the 
result, as all cones such that the associated 
model has class $\mathscr{P}$
are 
contained in the closure of $\shC_\mathscr{P}$.
By Corollary \ref{cor:seqtypeI}, 
there is a composition of type I flops $f\colon \YYT \dashrightarrow \shY'$. This defines cones 
$C(f\circ f_i)$, $i=1,\dots 3$ with associated model $\shY'$. The cone $C(f\circ f_i)$ is contained in $
\bar{\shC_i}$ by construction. So we are done if $\shY'$ is symmetric, as  there are then precisely $3$ cones with associated model $\shY'$, by Proposition \ref{orbitmain}.

 Suppose  $\shY'$ is not symmetric.  By Proposition \ref{orbitmain}, there are $6$ cones associated to 
 $\shY'$. We shall show that each of these cones lies in some $\bar{\shC_i}$.
 Since $\shY'$ is not symmetric,  it follows from Lemma \ref{lem:typeIflopsT} that there is a non-trivial sequence of type I flops 
 $ \gamma\colon\shY'\dashrightarrow \shY'$.  Then $C(\gamma\circ f\circ f_i)$ 
  is not equal to 
$C(f\circ  f_i)$, 
as $\gamma$ is not an automorphism and 
$C(\gamma \circ f \circ f_i )\subset \bar{\shC_i}$ by 
definition of $\shC_i$. This shows that all cones associated to $\shY'$ are contained in some $\bar{\shC_i}$. 

\end{proof}

\subsection{Flopping along a line}
Recall that we are, by Theorem \ref{ThmGHKS}, in a Mori Dream space situation (in degree $2$). 
The following is then a standard construction: 
let $F$ be a $\QQ$-divisor in $\Mov(\shY/S)$. Suppose $F$ is not nef on $\shY$
and there is a cone $\sigma_F\in \Morifan_{\operatorname{max}}(\shY/S)$ with $F\in \Int  \sigma_F$. 
Let  $A$ be an ample divisor and define $L=L(F,A)$ to be the line segment connecting $F$ and $A$. 
Since $L$ is spanned by interior points of a convex cone it is itself contained in the interior. 
Suppose that for any facet $\tau\in\Morifan(\shY/S)$, $L\cap \tau\neq \varnothing$ implies $L\cap \tau\subset \Relint(\tau)$. This means that the line $L$ intersects maximal cones and their facets as nicely as possible.
Note that (by convexity), this implies that if $\sigma\in\Morifan_{\operatorname{max}}(\shY/S)$, at most two facets of a cone  $\sigma$ are met, and the only maximal cones $\sigma$ such that there is a unique facet $\tau\subset \sigma$ meeting $L$ are $\sigma_F$ and $\sigma_A=\Nef(\shY/S)$.  Let $\{\gamma_i\}_i$ be the collection of  maximal cones of $\Morifan(\shY/S)$ such that $L\cap \Int \gamma_i\neq 0$. Note that this collection is finite by Theorem \ref{ThmGHKS}(i).   
Denote the unique facet of $\sigma_A$ met by $L$ 
by $\tau$ and let $R$ be the extremal ray corresponding to $\tau$. Then $F$ is stricly negative on $R$. Consider the contraction morphism $\contr_R\colon \shY\to \shZ$. It is a small contraction as $\tau$ is interior and hence defines an $F$-flop $\phi\colon\shY\dashrightarrow \shY^+$. 
This gives  a canonical linear  isomorphism $\phi_*\colon \N^1(\shY/S)\to \N^1(\shY^+/S)$. 
Choose an ample divisor $A^+\in\Pic(\shY^+/S)$ on $\phi_*(L)$.  Now we can consider the divisor $\phi_*F$ and the line segment $L(\phi_*F,A^+)$ and repeat the argument for this data. In this way we obtain a finite sequence of $F$-flops \[\shY\dashrightarrow\shY_1\dashrightarrow\dots \dashrightarrow\shY_n.\]  
We will call this sequence a  \emph{sequence of F-flops induced by $L$}.
We note that this may depend on the choice of $A$ and hence $L$, but not on the choice of $A^+$ and  subsequent choices.
The truncations $\phi_l\colon\shY\dashrightarrow \shY_l$ define  maximal cones of  $\Morifan(\shY/S)$ and by construction, for each $\gamma_i$ there is an $l$ such that $\gamma_i=\phi_l^*(\Nef(\shY_l/S))$. Note also that $\gamma_n=\sigma_F$ by construction.

 \begin{definition}Let $\shC$ be a connected component of $\shM$ with closure
 $\bar{\shC}$. 
 A \emph{$\shC$-test segment} is a line segment $L(p,q)$ such that
  $p,q$ are points in $\shC$, such that there are cones $\sigma_p, \sigma_q\in \Morifan_{\operatorname{max}}(\shY/S)$ with $p\in \Int (\sigma_p)$ and $q\in \Int(\sigma_q)$.
A test segment $L$ is called \emph{nice} if for any interior facet $\tau\in\Morifan(\shY/S)$, $L\cap \tau\neq \varnothing$ implies $L\cap \tau\subset \Relint(\tau)$.
\end{definition}

Given a test segment $L(p,q)$,
we assume for simplicity that $\sigma_q=\Nef(\shY/S)$. By choosing a different model, which does not change the geometry of the Mori fan, we can always assume 
that we are in this situation.

We now fix a connected component $\shC$ of $\shM$ with closure $\bar{\shC}$. 
Let $L=L(p,q)$ be a line segment. We define 
\[\operatorname{M}(L)=\{ \sigma\in \Morifan_{\operatorname{max}}(\shY/S) \mid  L\cap \sigma\neq \varnothing\},\]
\[\operatorname{I}(L)=\{ \sigma \in \operatorname{M}(L) \mid L\cap \Int( \sigma) \neq \varnothing\}.\]
For a cone $\sigma$ we write $F(\sigma)$ for the set of subcones of the Mori fan of codimension $1$ that are contained in $\sigma$, i.e. the facets of $\sigma$. We set 
\[\operatorname{N}(L)=\{ \sigma\in \operatorname{I}(L)\mid \forall \tau\in F(\sigma), L\cap\tau \subset  \Relint(\tau)\}.\]

Note that the conditions to be in $I(L)$ and $N(L)$ are open conditions, so if $L'$ is a line segment contained in a small cylinder containing $L$, then $\operatorname{I}(L)\subset \operatorname{I}(L')$ and 
$\operatorname{N}(L)\subset \operatorname{N}(L')$.
Also note that if $L$ is a nice test segment, $\operatorname{M}(L)=\operatorname{I}(L)=\operatorname{N}(L)$.

\begin{proposition}\label{mainprop} Let $\bar{\shC}$ be  the closure of the connected component $\shC$ of $\shM$.
Let $L=L(p,q)$ be a nice $\shC$-test segment where both $p$ and $q$ are divisors (with integral coefficients).
Let 
\[
\phi_L\colon  \shY\dashrightarrow\shY_1\dashrightarrow\dots \dashrightarrow \shY_n
\]
be the
sequence of flops induced by $L$. 
Then $\gamma\in \operatorname{M}(L)$ implies $\gamma \subset \bar{\shC}$.
\end{proposition}
\begin{proof}Let $\sigma_p$ and $\sigma_q$ be the cones containing $p$ and $q$, respectively. We assume $\sigma_q=\Nef(\shY/S)$.
 By Lemma \ref{C_P}, there are two cases: either $\shC=\shC_\mathscr{P}$ or  the two models corresponding to $\sigma_{p}$ and $\sigma_{q}$ are of type $\mathscr{T}$. 

 We first consider  $\shC=\shC_\mathscr{P}$. Then, as $L$ is a nice test segment, it follows from the definition of $\shC$ that  $\shY$ 
 and $\shY_n$ have 
dual intersection complex $\mathscr{P}$. 
 Then, by
Lemma \ref{onetypeIIflop}, all flops in the sequence $\phi_L$
are of type I and thus all $\shY_i$ 
have dual intersection complex $\mathscr{P}$. Hence all cones associated to $\shY_i$ are in $\bar{\shC}$, implying the claim.

Now consider the only other possible case, namely that  the two models corresponding to $\sigma_{p}$ and $\sigma_{q}$ are of type $\mathscr{T}$.
Then $\shY_{\sigma_p}$ and $\shY$ have  dual intersection complex $\mathscr{T}$. 
By construction of $\shC$ we have a map $\psi\colon\shY\dashrightarrow \shY_{\sigma_p}$  that is given by a series of type I flops. Note that these are not necessarily $F$-flops for a divisor $F\in \Mov(\shY/S)$. 
However, the map $\psi$ can also be written as a sequence 
\[\shY\dashrightarrow\shY'_1\dashrightarrow\dots 
\dashrightarrow \shY_{\sigma_p}\] of $F$-flops for some divisor $F\in \Mov(\shY/S)$. 
Any flop in this sequence is necessarily of type I. By definition of the Mori fan, we have an isomorphism $\pi\colon\shY_{\sigma_p}\cong \shY_n$ giving rise to a commutative diagram 
\[
\xymatrix{\shY\ar@{-->}[r]^\psi\ar@{-->}[dr]_{\phi_L} & \shY_{\sigma_p} \ar[d]^\pi\\
& \shY_n.}
\]
As $\psi$ is a composition of type I flops, it maps  smooth components of $\shY_c$ to smooth components and the same is true for $\pi$. Hence $\phi_L$ 
also maps the special component $(\shY_c)_\omega$ to the special component $((\shY_n)_c)_\omega$. We claim that this implies that all flops in  the factorisation of $\phi$ 
are of type I.
To show the claim, we assume there is a type II flop in the sequence and derive a contradiction. Let  
\begin{equation}\label{tail}
\shY_k\dashrightarrow\shY_{k+1}\dashrightarrow\dots 
\dashrightarrow\shY_n
\end{equation}
 be the tail of the sequence where $\phi_{k}\colon\shY_k\dashrightarrow\shY_{k+1}$ is the first type II flop. 
Being type II,  writing $(\shY_{k+1})_c=Y_1\cup Y_2\cup Y_3$, the  flopped curve $C$ is given  by  a component of the double curve, say  by $D_{12}=Y_1\cap Y_2$  on $Y_1$ and $D_{21}$ on $Y_2$, with $D_{21}^2=D_{12}^2=-1$. Also, the birational transform of $(\shY_c)_\omega$ on $\shY_{k+1}$ is $Y_3$.   Because $\shY_n$ has dual intersection complex $\mathscr{T}$, there is at least $1$ more type II flop in Sequence \ref{tail}, say $\phi_l$.
By Lemma \ref{onetypeIIflop} this  is the only further  type II  flop. Because  $\phi$  maps  $(\shY_c)_\omega$ to  $((\shY_n)_c)_\omega$, 
the exceptional locus $\Ex{\phi_l}$ is given by the birational transform $C_{l}$ of $C$ on $\shY_{l}$. Also, the preimage of $C_{l}$ 
under $\nu\colon (\shY_{l})_c^\nu\to (\shY_{l})_c$ consist of two $(-1)$-curves. Hence all of the flops $\phi_s$, $l-1\geq s\geq k+1$ have exceptional loci disjoint from the birational transform $C_s$ of $C$ on $\shY_s$. But this implies that for the birational transform $F_l$ of $F$ on $\shY_l$ we have $F_l.C_l>0$, as $F_{k+1}.C>0$. But the contraction defining  $\phi_l$ contracts $C_l$, a contradiction. Hence \emph{all} flops in $\phi$ are of type I. In particular, by Equation \ref{eq:conerepres}, all maximal cones met by $L$ are in the closure of $\shC$ and thus $\gamma \subset \bar{\shC}$.
\end{proof}

\subsection{Reduction} In this section we show that it is enough to check convexity on line segments defining sequences
of $F$-flops as in the previous section. The idea is as follows: if $L$ is a line segment through two points  in a connected component $\shC$, by definition of the connected components, it is enough to check that any maximal cone  $\sigma$ of the Morifan met by $L$ is contained in the closure $\bar{\shC}$  of $\shC$. For this, one checks that  there is a nice test segment $L'$ from some cone in $\shC$ that meets $\sigma$, which implies that all cones met by $L'$ are contained in $\bar{\shC}$. Such a test segment can be found   by wiggling the line segment $L$,  as for one it is enough to check convexity on the interior of the connected components  
and moreover  $\Morifan(\shY/S)$ has only finitely many cones. We give the details. The following lemma, which we include for completeness, is elementary.

\begin{lemma}\label{lem:convex}
Let $U$ be a convex set in a normed vector space $V$. Then the closure $\bar{U}$ and the interior $U^o$ are convex.
\end{lemma}
\begin{lemma}\label{lemma1}Let $\bar{\shC}$ be the closure of a  connected component of $\shM$ and let $p_0,q_0$ be points in the interior $\bar{\shC}^o$. Set $L_0=L(p_0,q_0)$ and assume 
$\gamma\in \Morifan_{\operatorname{max}}(\shY/S)$. If $L_0\cap\gamma\neq \varnothing$, then there is a $\shC$-test segment $L$ with $\gamma\in \operatorname{N}(L)$,  $
\operatorname{I}(L_0)\subset \operatorname{I}(L)$ and $\operatorname{N}(L_0)\subset \operatorname{N}(L)$.

\end{lemma}
\begin{proof} Given data as in the statement of the lemma, suppose  $L_0\cap\gamma\neq\varnothing$.  
 Then by translating both $p_0$ and $q_0$ by a small amount we can find points $p_1, q_1\in \shC^o$
 such that $  L_1=(p_1,q_1)$ intersects $\gamma$ in its interior.
 If the translation is small enough we can assume that $\operatorname{I}(L_0)\subset \operatorname{I}(L_1)$ and  $\operatorname{N}(L_0)\subset \operatorname{N}(L_1)$, as the relevant conditions are open and there are only finitely many cones in $\Morifan(\shY/S)$.
  
  Let $\tau_1\in F(\gamma)$ such that $L_1\cap \tau_1=\{x\}$, with $x\notin \Relint(\tau_1)$. Let $\tau_2$ be a facet of  $\tau_1$ with $x\in\tau_2$.  
  Let $v$ be a vector tangent to $\tau_1$ 
  but not tangent to $\tau_2$ to  such that $x+v\in \Relint(\tau_1)$.
   Then by translating both $p_1$ and $q_2$ by some $\varepsilon v$  we can find points $p_2, q_2\in \shC^o$
 such that $  L_2=(p_2,q_2)$ intersects $\gamma$ in its interior,  $L_2\cap \tau_1\subset \Relint(\tau_1)$, $\operatorname{I}(L_1)\subset \operatorname{I}(L_2)$
 and  $\operatorname{N}(L_1)\subset \operatorname{N}(L_2)$.  Repeating the last step if necessary produces a line segment $L=(p',q')$ with   $\gamma\in \operatorname{N}(L)$, $\operatorname{I}(L_0)\subset \operatorname{I}(L)$ 
 and $\operatorname{N}(L_0)\subset \operatorname{N}(L)$.
  We can also assume that neither  $p'$ nor  $q'$ is contained in a facet of $\Morifan(\shY/S)$. So there are cones    $\sigma_{p_1}, \sigma_{q_1}\in \Morifan_{\operatorname{max}}(\shY/S)$ with $p_1\in \Int (\sigma_{p_1})$ and $q_2\in \Int(\sigma_{q_1})$, so $L$ is indeed a test segment.
 
\end{proof}
\begin{lemma}\label{main2nd}Consider a connected component $\shC$ of $\shM$ and its closure $\bar{\shC}$. Let $p_0,q_0$ be points  in the interior $\bar{\shC}^o$ of the closure and set $L_0=L(p_0,q_0)$. Let $
\gamma\in \Morifan_{\operatorname{max}}(\shY/S)$. If $L_0\cap\gamma\neq \varnothing$,  then there exist $
\ZZ$-divisors $p$ and $q$ such that  $L(p,q)$ is a nice test segment with $L\cap\gamma\neq 
\varnothing$. 
Also, $\operatorname{N}(L_0)\subset \operatorname{N}(L)$ and  $\operatorname{I}
(L_0)\subset \operatorname{I}(L)$.
\end{lemma}
\begin{proof}We first show the existence of a test sequence $L=L(p,q)$ as claimed without the requirement that $p,q$ are $\ZZ$-divisors. 
Given  data as in the statement of the lemma, by Lemma \ref{lemma1}, there is a test segment $L_1$ with $\gamma\in \operatorname{N}(L_1)$.
Arguing inductively,
let $L_i$ be a test segment with $\gamma\in \operatorname{N}(L_i)$.
 Let $\sigma$ be a cone in $\operatorname{M}(L_i)\backslash \operatorname{N}(L_i)$. Again by Lemma \ref{lemma1}, we find a test segment  $L_{i+1}$ with $\operatorname{I}(L_i)\subset 
 \operatorname{I}(L_{i+1})$ and $\operatorname{N}(L_i)\subset \operatorname{N}(L_{i+1})$ and  $|\operatorname{N}(L_i)|< |\operatorname{N}(L_{i+1})|$.  
 Hence we obtain a sequence $\{L_i\}_i$ of test segments with $|\operatorname{N}(L_i)|< |\operatorname{N}(L_{i+1})|$. As $|\Morifan_{\operatorname{max}}(\shY/S)|<\infty$, this is a finite sequence $\{ L_1, L_2,\dots, L_n\}$, and thus $\operatorname{M}(L_n)=\operatorname{N}(L_n)$.  Setting $L=L_n$ proves the claim. 
 
 Now, given $L=L(p,q)$, perturbing a little, we can assume that $p,q$ are classes  of $\QQ$-divisors. Then there is an $n$ such that $L(np,nq)$
 is a test segment with the desired properties.
\end{proof}

\begin{theorem}\label{teo:countingsecondary}
Let $\shY\to S $ be a model of the DNV family of degree $2$.  The closures  $\bar{\shC}_i,  i=1,\dots 4$ of the connected components of $\shM$ are convex cones. The collection \[\Sigma_\shM=\{ \sigma \subset \shM\mid \sigma \text{ is a face of some } \bar{\shC}_i\}\] defines a finite fan of rational polyhedral cones.
\end{theorem}
\begin{definition}
The fan $\Sigma_\shM$ is the \emph{secondary fan}  of $\shY$, denoted by $\Morifan^{\operatorname {2nd}}(\shY/S)$.	
\end{definition}

\begin{proof}[Proof of Theorem \ref{teo:countingsecondary}] All that remains to be shown is that the connected components of $\shM$ have convex closures and that $\bar{C}_i\cap\bar{C}_j\in\Sigma_\shM$.
Let $\shC$ be a connected component. It is enough to show that its interior $\shC^o$ is convex. 
Let $p_0,q_0$ be points in $\shC^o$. Let $\gamma$ be a maximal cone meeting $L=L(p_0,q_0)$. By Lemma \ref{main2nd}, we find a nice
 $\bar{\shC}$-test segment $L(p,q)$ with $\gamma\in \operatorname{N}(L(p,q))$, where $p$ and $q$ can be chosen to be classes of divisors.
 By Proposition \ref{mainprop}, $\gamma\subset\bar{\shC}$. Let $\partial\Mov(\shY/S)$ be the boundary of the moving cone.
 Now, we can decompose $\bar{\shC}$
 as 
 \[ \bar{\shC}=\shC^0\cup B_1 \cup B_2 \] where $B_1\subset \partial\Mov(\shY/S)$ 
 and $B_2\subset \cup_{\tau\in\shF}|\tau|$.  
 where $\shF$ is the set of all facets of type II of the Mori fan.
 By convexity of $\Mov(\shY/S)$ 
 it follows that $L\subset \Mov(\shY/S)^o$. As we have just seen, all maximal cones $\gamma$  meeting $L=L(p_0,q_0)$ are contained in $\bar{\shC}$. It follows that all facets  $\tau$ met by $L$ are of type I, as else the maximal cones containing $\tau$ would correspond to models of the DNV family that do not have the same class. So $ L\cap B_2=\varnothing$
 and thus  $L\subset \shC^o$, implying that $\shC^o$ is convex.  It follows from Lemma \ref{lem:convex}  that  $\bar{\shC}$ is convex.
  By Proposition \ref{numberconessec}, we obtain  $4$ maximal cones $\bar{\shC}_i,  i=1,\dots 4$. It remains to check that $\Sigma_\shM$ is indeed a fan. This will follow from the following description of the $\bar{\shC}_i$. Assume $\shC_1=\shC_\mathscr{P}$. 
 From the construction, we immediately have the following description of the maximal cones. 
For each maximal cone $\bar{\shC}_i$ 
the facets contained in $\partial\Mov(\shY/S)$ are unions of  maximal cones in 
\[ \cup_{\{\sigma\in \Morifan_{\operatorname{max}}(\shY/S) \mid \sigma\subset \bar{\shC}_i\}}\partial\Mov(\shY/S)\cap \sigma.\] 
It follows from the convexity of the cones $\bar{\shC}_i$,   that each $\bar{\shC}_i$, $i=2,3,4$  has a unique facet $H_{\shC_i}$  meeting the interior of $\Mov(\shY/S)$. $H_{\shC_i}$ is the union of  type II facets that are contained in the $\bar{\shC}_i$, formally

 \[ H_{\shC_i}=\cup_{\tau\in \shF: \tau\subset \bar{\shC}_i}|\tau| .\]

For $i=1$, one has again the boundary facets and also the facets $H_{\shC_i}$. 
Note that by definition of $\shM$ the intersection of two cones of type $\mathscr{T}$ of  $\Sigma_\shM$ is at least of codimension $2$. 
It follows  $\bar{\shC}_i\cap\bar{\shC}_j\in\Sigma_\shM$.  
So $\Sigma_\shM$ is indeed a fan. Being a coarsening of a finite fan consisting of rational polyhedral cones, all cones of  $\Sigma_\shM$ are also rational polyhedral.
 \end{proof}

\begin{figure}\centering
\begin{tikzpicture}[scale=1.1]
\draw({0.25*3-3},{0.25*3*sqrt(3)})--({0.75*3-3},{0.75*3*sqrt(3)});
%\draw ({0.25*3-3},{0.25*3*sqrt(3)})--++(105:{sqrt(0.5*((1.5)^2*3+(1.5)^2)});
%\draw ({0.75*3-3},{0.75*3*sqrt(3)})--++(195:{sqrt(0.5*((1.5)^2*3+(1.5)^2)});

\draw ({0.25*3-3},{0.25*3*sqrt(3)})--({0.5*3-3-0.3*sqrt(3)},{0.5*3*sqrt(3)+0.3*1});
\draw ({0.75*3-3},{0.75*3*sqrt(3)})--({0.5*3-3-0.3*sqrt(3)},{0.5*3*sqrt(3)+0.3*1});

%second triangle
\draw({-0.25*3+3},{0.25*3*sqrt(3)})--({-0.75*3+3},{0.75*3*sqrt(3)});

\draw({-0.25*3+3},{0.25*3*sqrt(3)})--({-0.5*3+3+0.3*sqrt(3)},{0.5*3*sqrt(3)+0.3*1});
\draw({-0.75*3+3},{0.75*3*sqrt(3)})--({-0.5*3+3+0.3*sqrt(3)},{0.5*3*sqrt(3)+0.3*1});

% third triangle
\draw({0.25*6-3},{0})--({0.75*6-3},{0});
\draw({0.25*6-3},{0})--(0,{-0.3*2});
\draw({0.75*6-3},{0})--(0,{-0.3*2});

%small upper 
%\draw[red] ({0.75*3-3},{0.75*3*sqrt(3)})--({-0.75*3+3},{0.75*3*sqrt(3)});
\draw ({0.75*3-3},{0.75*3*sqrt(3)})--({0.75*3-3+0.25*3},{0.75*3*sqrt(3)+0.1*2});
\draw ({-0.75*3+3},{0.75*3*sqrt(3)})--({0.75*3-3+0.25*3},{0.75*3*sqrt(3)+0.1*2});

%small left
%\draw[red] ({0.25*3-3},{0.25*3*sqrt(3)})--({0.25*6-3},{0});
%\draw  ({0.25*3-3-0.125*9+1.5},{0.25*3*sqrt(3)-0.125*3*sqrt(3)})-- ({0.25*3-3-0.125*9+1.5-0.125*sqrt(3)},{0.25*3*sqrt(3)-0.125*3*sqrt(3)-0.125*1});
\draw ({0.25*3-3},{0.25*3*sqrt(3)})--({0.25*3-3-0.125*9+1.5-0.1*sqrt(3)},{0.25*3*sqrt(3)-0.125*3*sqrt(3)-0.1*1});
\draw ({0.25*6-3},{0})--({0.25*3-3-0.125*9+1.5-0.1*sqrt(3)},{0.25*3*sqrt(3)-0.125*3*sqrt(3)-0.1*1});

%small right
%\draw[red] ({-0.25*3+3},{0.25*3*sqrt(3)})--({0.75*6-3},{0});
%\draw  ({0.75*6-3+0.125*3},{0.125*3*sqrt(3)})-- ({0.75*6-3+0.125*3+0.125*sqrt(3)},{0.125*3*sqrt(3)-0.125*1});

\draw({-0.25*3+3},{0.25*3*sqrt(3)})-- ({0.75*6-3+0.125*3+0.1*sqrt(3)},{0.125*3*sqrt(3)-0.1*1});

\draw ({0.75*6-3},{0})-- ({0.75*6-3+0.125*3+0.1*sqrt(3)},{0.125*3*sqrt(3)-0.1*1});

\end{tikzpicture}
\caption{A schematic picture of $\Morifan^{\operatorname{2nd}}(\shY/S)$. The central cone is $\shC_\mathscr{P}$.}

\label{fig:Secfan}
\end{figure}

\begin{remark}\label{orbits:2nd} The  $S_3$-action on $\YYP$  induces an action on the maximal cones of $\Morifan^{\operatorname {2nd}}(\shY/S)$. Write $\bar{\shC}_i,  i=1,\dots 4$ for the maximal cones. Assume $\shC_1=\shC_\mathscr{P}$.  As $\YYP$ has as unique associated cone, this action leaves $\shC_\mathscr{P}$ invariant. Combining  the arguments in Proposition \ref{mainprop} and Proposition \ref{orbitmain} shows that the $S_3$-action is indeed a permutation action on the cones $\bar{C}_i, i=2,\dots 4$. 
Figure \ref{fig:Secfan} provides a symbolic description of the structure of  $\Morifan^{\operatorname{2nd}}(\shY/S)$.
\end{remark}

\bibliographystyle{myamsalpha}
\bibliography{HL_Morifan.bib}
	
\end{document}